\documentclass{amsart}

\usepackage[utf8]{inputenc}
\usepackage[english]{babel}
\usepackage{amsmath}
\usepackage{amsfonts}
\usepackage{amscd}
\usepackage{amssymb}
\usepackage{amsthm}
\usepackage[initials]{amsrefs}
\usepackage{tikz-cd}
\usepackage{hyperref}
\usepackage{diagbox}
\usepackage{graphicx}
\usepackage[left=2.5cm,right=2.5cm, top=3cm,bottom=3cm]{geometry}
\usepackage{array,caption}
\usepackage{float}
\usepackage{soul}

\allowdisplaybreaks

\numberwithin{equation}{subsection}

\newtheorem{prop}[equation]{Proposition}
\newtheorem*{prop*}{Proposition}
\newtheorem{lmm}[equation]{Lemma}
\newtheorem{thm}[equation]{Theorem}

\newtheorem*{thm*}{Theorem}
\newtheorem{varthm}{Theorem}

\newtheorem{conj}{Conjecture}
\newtheorem{cor}[equation]{Corollary}
\theoremstyle{definition}
\newtheorem{dfn}[equation]{Definition}
\newtheorem{ntt}[equation]{Notation}
\newtheorem{exmp}[equation]{Example}
\newtheorem*{ntt*}{Notation}
\theoremstyle{remark}
\newtheorem{rmk}[equation]{Remark}

\DeclareMathOperator{\Hom}{Hom}

\DeclareMathOperator{\vp}{\varphi}
\DeclareMathOperator{\PP}{\mathbb P}
\DeclareMathOperator{\Q}{\mathbb Q}
\DeclareMathOperator{\Z}{\mathbb Z}
\DeclareMathOperator{\R}{\mathbb R}
\DeclareMathOperator{\Co}{\mathbb C}
\DeclareMathOperator{\Oh}{\mathcal O}
\DeclareMathOperator{\HH}{\mathbb H}
\DeclareMathOperator{\Lk}{\mathrm Lk}

\DeclareMathOperator{\id}{id}
\DeclareMathOperator{\rk}{rk}
\DeclareMathOperator{\im}{Im}
\DeclareMathOperator{\Sing}{Sing}

\DeclareMathOperator{\Tot}{Tot}
\DeclareMathOperator{\rreg}{reg}
\DeclareMathOperator{\aut}{Aut}
\DeclareMathOperator{\paut}{PAut}
\DeclareMathOperator{\e}{\varepsilon}
\DeclareMathOperator{\dec}{dec}
\newcommand{\Ext}{\mathrm{Ext}}
\newcommand{\reg}[1]{\Gamma_{\mathrm{reg}}(#1)}
\newcommand{\nonzero}[1]{\Gamma(#1)_{0}}
\newcommand{\bm}[1]{H_{#1}^{BM}}
\newcommand{\fr}[1]{H^{#1}_{\mathrm{fr}}}
\newcommand{\naive}[2]{#1_{n#2}}

\title{
Topology of spaces of regular sections and applications to automorphism groups}
\author{Alexey Gorinov}
\address{Faculty of Mathematics, HSE University, 6 Usacheva ulitsa, Moscow 119048, Russia}
\email{agorinov@hse.ru, gorinov@mccme.ru}
\author{Nikolay Konovalov}
\address{\parbox{\linewidth}{Faculty of Mathematics, HSE University, 6 Usacheva ulitsa, Moscow 119048, Russia \\Department of Mathematics, University of Notre Dame, 255 Hurley Hall, Notre Dame, IN 46556, USA}}
\email{nkonoval@nd.edu, nikolay.konovalov.p@gmail.com}
\date{}

\begin{document}

\begin{abstract}
Let $G$ be a complex connected reductive algebraic group that acts on a smooth complex algebraic variety $X$, and let $E$ be a $G$-equivariant algebraic vector bundle over $X$. A section of $E$ is {\it regular} if it is transversal to the zero section. Let $U\subset\Gamma(X,E)$ be the subset of regular sections. We give a sufficient condition in terms of topological invariants of $E$ and $X$ that implies that every orbit map $O\colon G\to U$ induces a surjection in rational cohomology. {Under natural assumptions on $X$ and $E$ this condition is also necessary.}

If {the} condition is satisfied, then (1) the geometric quotient $U/G$ exists; (2) there is an isomorphism $H^*(U,\Q)\cong H^*(G,\Q)\otimes H^*(U/G,\Q)$ of cohomology rings; (3) the order of the stabiliser $G_s,s\in U$ divides a certain expression that can be explicitly calculated e.g.\ if $X$ is a compact homogeneous space. In some cases (e.g.\ if $E$ is a line bundle) we also prove similar statements for the space of the zero loci of $s\in U$. We apply these results to several explicit examples which include hypersurfaces in projective spaces, non-degenerate quadrics and complete flag varieties of the simple Lie groups of rank 2, and also certain Fano varieties of dimension 3 and~4. 
\end{abstract}
\maketitle

\section*{Introduction}

A section of an algebraic vector bundle $E$ over a smooth complex algebraic variety $X$ is called {\it regular} if it is transversal to the zero section. The zero locus of a regular section is a smooth subvariety of $X$.
In this paper we prove several results about the topology of the space $\reg{X,E}$ of regular sections. More specifically, let $G$ be a complex affine group that acts algebraically on $X$, and suppose that this action lifts to an algebraic action of $G$ on the total space of $E$ by fibrewise linear automorphisms. In this case we say that $E$ is {\it $G$-equivariant}.

The prototype for our results is the following theorem of C.\ Peters and J.\ Steenbrink \cite{PS03}: if $G=GL_{n+1}(\Co), X=\PP^n(\Co), E=\Oh(d),d\geq 3$, then there is an isomorphism
\begin{equation}\label{prototype}
H^*(\reg{X,E},\Q)\cong H^*(G,\Q)\otimes H^*(\reg{X,E}/G,\Q).
\end{equation}
A natural question is whether this generalises to other equivariant bundles. We give a necessary and sufficient condition for this to be true in terms of topological invariants of $E$ and $X$, provided $X$ is compact, the jet bundle $E'=J(\Oh_{\PP(E^*)}(1))$ is globally generated by holonomic sections (see Section~{\ref{secsmoothsec}}), and $c_{\dim \PP(E^*)}(E')\neq 0$. 
This condition (see Corollary~{\ref{maincorappl}} and Theorem~\ref{B} below) can be explicitly checked e.g.\ if $X=G/P, P$ a parabolic subgroup. We give a few examples in Theorem \ref{mainthmprojhyp}. When $G=GL_{n+1}(\Co),X=\PP^n(\Co), E=\Oh(d)$, the condition translates simply as $d\not\in\{0,1,2\}$.

%
One motivation for results of this type is deducing the cohomology of the quotient $\reg{X,E}/G$ from the cohomology of $\reg{X,E}$, or vice versa. Of these two spaces, $\reg{X,E}$ is usually easier to deal with; in particular, its cohomology can be calculated using conical resolutions, see~\cite{Vas99}, \cite{Gor05} and \cite{Tom05}. On the other hand, the quotient $\reg{X,E}/G$ is often more interesting than $\reg{X,E}$ because it is the moduli space of objects of some kind; for example, the moduli spaces of smooth curves of small genus canonically decompose into pieces of the form $\reg{X,E}/G$ for appropriate $X,E$ and $G$, see \cite{Tom05} and Remarks~\ref{m5} and \ref{m8}, where we give a few simple applications to moduli spaces.

As in \cite{PS03}, in order to prove \ref{prototype} we construct certain cohomology classes $\Lk(y),y\in H_*(X)$ of $\reg{X,E}$, which we will call {\it {linking classes}} (see Definition~\ref{linkingclasssing}), and calculate the pullbacks of these classes under an orbit map $O\colon G\to\reg{X,E}$ (Theorem \ref{maintheorem}). If it so happens that {in rational cohomology the pullbacks} multiplicatively generate $H^*(G,\Q)$, then \ref{prototype} follows, essentially by the Leray-Hirsch principle, see Theorem \ref{thmquotslice}. We conjecture that for line bundles the ``divisibility phenomenon'' \ref{prototype} is generic in $E$ for a given $X=G/P$; see Section \ref{genconj} for a precise statement.

The key step of this strategy is the first one, namely constructing the linking classes and calculating their pullbacks under the orbit maps. The geometric construction of these classes given in~{\cite{PS03}} for $X=\PP^n(\Co), E=\Oh(d)$ admits a straightforward generalisation to the case of an arbitrary equivariant bundle. This generalisation is however difficult to work with in practice, so instead we take a different approach, which seems better suited for our purposes. We check however that if $X=\PP^n(\Co), E=\Oh(d)$, then our construction and the one by C.~Peters and J.~Steenbrink in~{\cite{PS03}} give the same answer, see Proposition~{\ref{geometricdescription2}}. 

\smallskip

We will now explain these results in more detail assuming for simplicity that $X$ is compact, $G$ is reductive and connected, and $E$ is a line bundle, which we denote $L$. 
Set $d=\dim_{\Co} X$. Recall that the {\it first jet bundle} $E'=J(L)$ of $L$ has rank $d+1$ and there is a map $j\colon \Gamma(X,L)\to \Gamma(X,E')$ such that $s\in\reg{X,L}$ if and only if $j(s)$ does not vanish anywhere. The Euler class of $E'$ is zero for dimension reasons, so there exists a secondary Thom class $a_{E'}\in H^{2d+1}(\Tot_0(E'),\Z)$ where $\Tot_0(E')$ is the total space of $E'$ minus the zero section (see Section \ref{secthomeuler}). For $y\in H_*(X,\Z)$ we then set $\Lk(y)=-j_{ev}^*(a_{E'})/y$ where ${j_{ev}}\colon\reg{X,L}\times X\to \Tot_0(E')$ is the jet evaluation map defined by ${j_{ev}}(s,x)=j(s)(x)$, and $/$ denotes the slant product $$H^p(\reg{X,L}\times X,\Z)\otimes H_q(X,\Z)\to H^{p-q}(\reg{X,L},\Z).$$

\begin{varthm}
\label{A}
The equivariant Euler class $e_G(E')\in H^*_G(X,\Q)$ (see Definition~\ref{equivarianteulerclass}) has a decomposition
$$e_G(E')=\sum_i \beta^*(a_i)b_i$$
where $a_i\in \widetilde{H}^*(BG,\Q), b_i\in H^*_G(X,\Q),$ and $\beta^*\colon H^*(BG,\Q) \to H_G^*(X,\Q)$ is the structure map. Moreover, modulo decomposable elements we have
\begin{equation}\label{orbit_intro}
O^*(\Lk(y)) = \sum \langle \alpha^* (b_i), y\rangle \bar\gamma(a_i) \in H^*(G,\Q).
\end{equation}

Here $\alpha^*\colon H^*_G(X,\Q) \to H^*(X,\Q)$ is the restriction map, and $\bar\gamma$ is the composite $$H^*(BG,\Q)\to H^*(\Sigma G,\Q)\to H^{*-1}(G,\Q)$$ where the first arrow is induced by the map $\gamma\colon\Sigma G \to BG$ obtained from the canonical homotopy equivalence $G\simeq \Omega BG$ by the suspension-loop space adjunction.
\end{varthm}
We note that the right hand side of equation \ref{orbit_intro} is integral, because the left hand side is, but the individual ingredients of the right hand side may not be.

As a consequence of Theorem~\ref{A} we have the following necessary and sufficient condition for the existence 
of an isomorphism~\ref{prototype}.
We keep the notation of Theorem~\ref{A}.

\begin{varthm}
\label{B}
Suppose that $\reg{X,L}$ is the complement of a divisor in $\Gamma(X,L)$, and $E'$ is globally generated by holonomic sections (see Section~\ref{secsmoothsec}). Let $\reg{X,L}/G$ be the good categorical quotient. Then the subring 
$$\Lambda^*_{E'}=\left.\left\langle \sum_i \langle \alpha^*(b_i), y\rangle \bar\gamma(a_i) \;\; \right\vert\;\; y\in H_*(X,\Z)\right\rangle \subset H^*(G,\Q)$$
is integral, i.e.\ it is a subring of $\fr{*}(G,\Z)=H^*(G,\Z)/\mathrm{torsion}$, and
the following statements are equivalent:
\begin{enumerate}
\item The ring $\Lambda^*_{E'}$ has finite index in $\fr{*}(G,\Z)$.
\item There is a mixed Hodge substructure $\subset H^*(\reg{X,L},\Q)$ that maps isomorphically to $H^*(G,\Q)$ under $O^*$ for every orbit map $O:G\to\reg{X,L}$.
\item For $E=L$ there exists an isomorphism~\ref{prototype} that is a map of mixed Hodge structures, rings and $H^*(\reg{X,L}/G,\Q)$-modules.
\end{enumerate}
Moreover, if either of these statements holds, then the $G$-stabiliser of every $s\in\reg{X,L}$ is finite, which in turn implies that the categorical quotient $\reg{X,L}/G$ is in fact geometric.
\end{varthm}
We note that the (geometric or just categorical) quotient $\reg{X,L}/G$ is necessarily affine. Theorems~\ref{A} and~\ref{B} follow from Corollary~\ref{maincorollary2} and Theorem~\ref{mainthmappl} respectively. Both theorems can be easily extended to vector bundles of arbitrary rank, see Theorem~\ref{maintheorem} and Corollary~\ref{maincorappl}. To do this we use the Cayley trick, which consists in replacing a vector bundle $E$ on $X$ with the line bundle $\Oh_{\PP(E^*)}(1)$ on $\PP(E^*)$, see Section~\ref{secsmoothsec}. 

\smallskip

The {linking} classes $\Lk(y)$ are defined over $\Z$, and in Theorem~\ref{A} we calculate their pullbacks under~$O$ in the integral cohomology of $G$ modulo torsion. As a consequence, we are able to obtain bounds on the orders of the $G$-automorphism groups of elements of $\reg{X,L}$, as well as several related automorphism groups. Here is an example. 

\begin{varthm}
\label{C}
We keep the notation and assumptions of Theorem B, and additionally assume that $G$ is semi-simple. Suppose that the index $[H^{\mathrm{top}}(G,\Z):\Lambda^{\mathrm{top}}_{E'}]$ is finite where $\mathrm{top}=\dim_{\Co} G$ is the top degree in which the cohomology of $G$ is non-zero. Then for every $s\in \reg{X,L}$
\begin{enumerate}
\item the order of the $G$-stabiliser of $s$ divides $[H^{\mathrm{top}}(G,\Z):\Lambda^{\mathrm{top}}_{E'}]$;
\item the order of the setwise $G$-stabiliser of the zero locus of $s$ divides 
\begin{equation}\label{auto_zero_loc_line_bund}
[H^{\mathrm{top}}(G,\Z):\Lambda^{\mathrm{top}}_{E'}]\cdot \langle c_d(E'),[X]\rangle.
\end{equation}
\end{enumerate}
\end{varthm}

We note that $\langle c_d(E'),[X]\rangle$ is often the degree of the discriminant hypersurface $\Gamma(X,L)\setminus\reg{X,L}$, see Proposition~\ref{degreediscr} for details. The first part of Theorem~\ref{C} follows from Theorem~\ref{mainthmappl} and the second part from Corollaries~\ref{maincorappl} and \ref{descentc*2ndrev} and Lemmas~\ref{automorphismsofzeros} and \ref{linebundletrivial}. Same as Theorems~\ref{A} and~\ref{B}, the first part of Theorem~\ref{C} generalises to higher rank bundles using the Cayley trick (see Corollary~\ref{maincorappl}). So does the second part under the additional assumption that the fibrewise automorphism group of the vector bundle acts transitively on the regular sections with given zero locus; see Section~\ref{sectionsvsloci}, where we give a sufficient condition for this to be true and an example which shows that it is not true in general.

If this assumption is true however, then the higher rank analogue of part~2 of Theorem~\ref{C} follows from Lemma~\ref{automorphismsofzeros} and Corollary~\ref{maincorappl}. We give a few examples in Section~\ref{fano}. In practice, we also need the reductive part of the quotient of the fibrewise automorphism group by the scalar automorphisms to be semi-simple: otherwise the analogue of~\ref{auto_zero_loc_line_bund} for the group $\widetilde{G}$ (see Section~\ref{sectionsvsloci}) will be infinite.


As an illustration of Theorem~\ref{C}, the order of the projective automorphism group of every smooth projective hypersurface of degree $d$ divides
\begin{equation}\label{prototype2}
(d-1)^n \prod_{i=2}^{n+1} ((d-1)^{n+1}+(-1)^{i+1}(d-1)^{n+1-i}).
\end{equation}


More examples of this type are given in Theorem \ref{mainthmprojhyp}; in Corollary \ref{maincorappl} and Sections \ref{genalg} and \ref{completeflagvarieties} we describe a general procedure for obtaining similar formulas for equivariant line bundles over an arbitrary homogeneous space $G/P$. We note that V.~Gonz\'alez-Aguilera and A.~Liendo show in~\cite{GAL13} that the primes that occur in \ref{prototype2} are precisely the primes that can be realised as the order of an automorphism of some smooth projective hypersurface of degree $d$. We conjecture that for a line bundle over $G/P$ expression~\ref{auto_zero_loc_line_bund} (provided it is finite and non-zero) contains all primes that occur as the orders of $G$-automorphisms of the zero loci of regular sections, see Section \ref{parsimconj}.

\smallskip

The paper is organised as follows. In Section~\ref{secthomeuler} we introduce certain characteristic classes that we call the {\it secondary} Thom and equivariant Euler classes. The reason they are called secondary is that they are defined for oriented vector bundles whose ordinary Euler class vanishes. {The main example we will be interested in is the first jet bundle of a line bundle.} 

In Section~\ref{maps} we define the map~$S$ (Definition~\ref{mainhomomorphism}), our main technical ingredient, and prove some of its properties. The key results in Section~\ref{1stsec} are Propositions~\ref{secthomseceuler}, \ref{mainproposition} and Remark~\ref{rmkSS1}, which is the main tool for calculating $S$ in the following sections. Proposition {\ref{transgression}} {relates} the Euler class of an oriented vector bundle {to} the transgression in the Leray spectral sequence of the corresponding spherical bundle {(see also Remark~{\ref{rmktransgression}})}.

Section~{\ref{sectdiscrim}} contains our main technical results: we construct the {\it linking class homomorphism} for spaces of nowhere vanishing sections {($\Lk_V^{\Sigma}$, Definition~{\ref{linkingclasshomomorphism}}) and for spaces of regular sections ($\Lk_V^{\Sing}$, Definition~{\ref{linkingclasssing}}). For $X$ compact algebraic the image of $\Lk_V^{\Sing}$ is precisely the first non-trivial term of the weight filtration on $H^*(\reg{X,E},\Q)$, see Proposition~{\ref{mhsmotivation2}}. We express the composition of $\Lk_V^{\Sigma}$ and $\Lk_V^{\Sing}$} with $O^*$ in Corollary~\ref{prepmaincorollary} and Theorem~\ref{maintheorem} in terms of~$S$ applied to the equivariant Euler class of $J(\Oh_{\PP (E^*)}(1))$.

In Section \ref{sectapplications} we explain how the results of the preceding sections can be used to prove analogues of isomorphism \ref{prototype} and to obtain information about the orders of automorphism groups (Theorem \ref{thmquotslice} and Proposition \ref{stab}). {Proposition~{\ref{degreediscr}} relates the degree of a variety ``swept out'' by vector subspaces to the integral of the Euler class of a certain vector bundle (see also Remark~{\ref{rmkdegreediscr}}). The main results of Section~{\ref{sectapplications}} are summarised in Theorem~{\ref{mainthmappl}} and Corollary~{\ref{maincorappl}}.}

In Section \ref{genalg} we give a general recipe to calculate the linking classes explicitly for compact homogeneous spaces $G/P$ where $G$ is a connected affine group and $P$ is a parabolic subgroup ({at least when $G$ is of classical type}, see the beginning of Section~\ref{genidealbargamma}). We then apply this recipe to hypersurfaces in projective spaces, non-degenerate quadrics and the complete flag variety of $SL_3(\Co)$, and to certain Fano varieties of dimension 3 and 4. {These applications are summarised in Theorem~{\ref{mainthmprojhyp}}; the list is not exhaustive, and the purpose of Sections {\ref{projquad}}-{\ref{fano}} is to illustrate the kind of problems that can be handled using our methods, {and also to show how the calculations work in concrete examples}. In Remarks {\ref{pointsp1}}-{\ref{knownlcm}} we 
compare Theorem~{\ref{mainthmprojhyp}} with known results about automorphism groups.}

Finally, in Section \ref{secconjectures} we propose a few conjectures and discuss some evidence in favour of these.

There are two appendices in the paper. In Appendix \ref{app_proofs} we prove Propositions \ref{transgression} and \ref{degreediscr}. {These results are not needed for our main applications (at least not in full generality), but we believe they are of independent interest.} {In Appendix~{\ref{app_tables}} we give numerical examples for our results in Section~\ref{secexamples} on common multiples of the orders of the automorphism groups.}

The proofs of our results in Sections \ref{1stsec} and \ref{sectdiscrim} 
{are mostly homotopy theoretic}, so it seems likely that these results can be extended in two diverging directions, one being spaces of regular $C^\infty$-smooth sections of real equivariant vector bundles, and the other equivariant vector bundles in finite characteristic. In this paper we focus on those constructions that translate directly into
both these situations, and give a first series of applications to complex algebraic varieties.

{\bf Acknowledgments.} We are grateful to {Sasha Berdnikov, }Misha Finkelberg, Sasha Kuznetsov, Lyonya Rybnikov {and Kostya Shramov} for helpful conversations.

{\bf Notation.} Here we describe some notation that will be used throughout the paper. If $E$ is a vector bundle over a topological space $X$, we denote the total space of $E$ by $\Tot(E)$ and the total space minus the zero section by $\Tot_0(E)$. The pullback of $E$ along $f\colon Y\to X$ is denoted $f^*(E)$. A one point space will be denoted $pt$. For a based space $X$ we use $\Omega X,\Sigma X$ and $X_+$ to denote the loop space, the (reduced) suspension and the space $X\sqcup pt$ respectively; the base point of $X_+$ is the added point. If two spaces $X,Y$ are homotopy equivalent, we will write $X\simeq Y$.

If $X$ is a topological space, we denote the ring $H^*(X,\Z)/\rm{torsion}$ by $\fr{*}(X,\Z)$. We will often view $\fr{*}(X,\Z)$ as a subring of $H^*(X,\Q)$. If $R$ a commutative ring with identity, then we use $\langle -,-\rangle$ to denote the evaluation pairing $H^n(X,R)\otimes_R H_n(X,R)\to R$. The {\it Borel-Moore homology groups} $\bm{*}(X,R)$ of $X$ with coefficients in $R$ are the homology groups of the complex of locally finite singular chains with coefficients in $R$. If $X$ is a finite CW-complex and $Y\subset X$ is a subcomplex, then $\bm{*}(X\setminus Y,R)\cong H_*(X,Y,R)$; if moreover $R$ is a field, then $\bm{*}(X\setminus Y,R)\cong\Hom_R (H^*_c(X\setminus Y,R),R)$.

Suppose $X,Y$ are topological spaces; let $p_1\colon X\times Y\to X$ and $p_2\colon X\times Y\to Y$ be the projections. If $E_1$ and $E_2$ are vector bundles over $X$ and $Y$ respectively, then $E_1\boxtimes E_2$ denotes the exterior product of $E_1$ and $E_2$, i.e.\ the vector bundle $p_1^*(E_1)\otimes p_2^*(E_2)$. Similarly, if $x_1\in H^m(X,R),x_2\in H^n(X,R)$, then $x_1\times x_2$ denotes the class $p_1^*(x_1)\smile p_2^*(x_2)$. We will often omit $\smile$ when writing the cup products of cohomology classes.

Suppose a group $G$ acts on a set $X$. We will denote the result of applying $g\in G$ to $x\in X$ by $g\cdot x$ or $gx$, respectively $x\cdot g$ or $xg$ if the action is on the left, respectively on the right. All our group actions will be left group actions unless specified otherwise. We will use $G_x$, respectively $G(x)$ to denote the stabiliser of $x\in X$, respectively the orbit of $x$. The {\it orbit map} $G\to X$ given by $g\mapsto g\cdot x$ will be denoted $O_x$, or $O$ when it is clear or irrelevant which $x$ we take.

By the dimension of a complex analytic variety (in particular, a complex affine group) we will always mean the complex dimension. Similarly, the rank of a complex vector bundle will be understood to mean the complex rank.

Let $E$ be a holomorphic vector bundle over a complex analytic variety $X$. The space of holomorphic sections of $E$ is denoted $\Gamma(X,E)$. If $X$ is irreducible, then the {\it fundamental class} of $X$ is an element of $\bm{2\dim X}(X,\Z)$; we will denote this element by $[X]$.

\section{Secondary characteristic classes}\label{1stsec}
Suppose $E\to X$ is a real oriented vector bundle with vanishing Euler class. In Section~\ref{secthomeuler} we define a {\it secondary Thom class} $a_E\in H^{\rk E-1}(\Tot E_0)$ (Definition~\ref{secthom}) and prove some of its basic properties. If in addition $G$ is a topological group acting on $X$ and $E$ is $G$-equivariant, then we define a {\it secondary $G$-equivariant Euler class} $se_G(E)\in H^{\rk E-1}(G\times X)$ of $E$ (Definition~\ref{seceqeuler}). Neither secondary class is uniquely defined; in particular, the secondary Euler class is defined up to the {\it indeterminacy submodule} $M_X^*$ (Definition~\ref{indeterminacy submodule}). In Proposition~\ref{secthomseceuler} we compare the secondary Thom and Euler classes using Proposition~\ref{transgression}, which expresses the Euler class of an oriented vector bundle in terms of the transgression in the Leray-Serre spectral sequence of the spherisation. We also define the {\it naive homotopy quotient} $X_{nG}$ (Definition~\ref{nhq}) as the homotopy colimit of a certain diagram, which gives us a long exact sequence (sequence~\ref{lesnhq}) that we will need in the sequel.

In Section~\ref{maps} we introduce our main technical ingredients: the ideals $I(X,R)$ and $I_1(X,R)$ of $H^*(X_{hG},R)$ (Notations~\ref{mainidealeq} and \ref{subideal1}) and the map $S\colon I^*(X,R) \to H^{*-1}(G\times X,R)/M_X^{*-1}$ (Definition~\ref{mainhomomorphism}). The results that will enable us to calculate $S$ in the upcoming sections are Proposition~\ref{mainproposition} and Remark~\ref{rmkSS1}. Finally, in Section~\ref{secsmoothproj} we describe a few simplifications assuming $X$ smooth compact complex algebraic. In particular, it turns out that in this case the ideals $I(X,\Q)$ and $I_1(X,\Q)$ coincide (Lemma~\ref{s1sproj}).

\subsection{Secondary Thom and Euler classes}\label{secthomeuler}
In this subsection $X$ will be a topological space, $E$ will be a real vector bundle over $X$ of rank $r$, and $R$ will denote a commutative ring with identity. Let $p_0\colon \Tot_0(E) \to X$ be the projection map, and let $E_x$ be the fibre of $E$ over $x \in X$. We set $E_{0,x}=E_x\setminus\{0\}$.

Recall that $E$ is {\it $R$-orientable} if there is a {\it Thom class} $u_E\in H^r(\Tot(E),\Tot_0(E),R)$ such that $u_E$ restricted to the couple $(E_x,E_{0,x})$ is an $R$-generator of $H^r(E_x,E_{0,x},R)\cong R$ for every $x\in X$. Once we have chosen a Thom class $u_E$ we will say that $E$ is {\it $R$-oriented} and call $u_E$ an {\it $R$-orientation} of $E$. If $E$ is $R$-oriented, then the {\it Euler class} $e(E)$ is the pullback of $u_E$ under the zero section map $(X,\varnothing)\to (\Tot(E),\Tot_0(E))$.

Note that if $\mathop{\mathrm{char}}R=2$, then all vector bundles are $R$-orientable, and otherwise a vector bundle is $R$-orientable if and only if it is $\Z$-orientable.

\begin{dfn}
\label{secthom}
A cohomology class $a_E \in H^{r-1}(\Tot_0(E), R)$ is called {\it a secondary Thom class of $E$} if and only if its restriction to each fibre $E_{0,x}$ is a generator of the $R$-module $H^{r-1}(E_{0,x}, R)\cong H^{r-1}(S^{r-1},R)$.
\end{dfn}
\begin{prop}
\label{existencesecthom}
 A secondary Thom class of a vector bundle $E$ exists if and only if $E$ is $R$-orientable and the Euler class $e(E)$ is zero.
\end{prop}
\begin{proof}
Consider the exact sequence:
\begin{equation}
\label{lessecthom}
H^*(\Tot(E), R) \xrightarrow{i^*}  H^*(\Tot_0(E), R) \xrightarrow{\delta} H^{*+1}(\Tot(E),\Tot_0(E), R) \xrightarrow{j^*}  H^{*+1}(\Tot(E), R). 
\end{equation}
Here $i$ and $j$ are the inclusions $\Tot_0(E)\to \Tot(E)$ and $(\Tot(E), \varnothing)\to (\Tot(E),\Tot_0(E))$ respectively.
Exact sequence~\ref{lessecthom} is functorial with respect to fibrewise isomorphisms of vector bundles. \par
Suppose that a secondary Thom class $a_E \in H^{r-1}(\Tot_0(E), R)$ exists. Then $\delta(a_E)$ is a Thom class of the vector bundle $E$, which can be checked by restricting $\delta(a_E)$ to the fibre over an $x\in X$. 
So $E$ is $R$-orientable and the Euler class of $E$ is $j^*(\delta(a_E))=0$. \par
Now suppose that $E$ is $R$-orientable and the Euler class of $E$ is zero. Then there exists a Thom class $u_E \in H^{r}(\Tot(E),\Tot_0(E),R)$ and $j^*(u_E)=0$. So there is an $a_E\in H^{r-1}(\Tot_0(E),R)$ such that $\delta(a_E)=u_E$. Using \ref{lessecthom} again and restricting to the fibre over $x\in X$ 
we see that $a_E$ is a secondary Thom class. 
\end{proof}

\begin{prop}
\label{functorsecthom}
Let $f\colon Y\to X$ be a continuous map, and let $E'$ and $E$ be vector bundles of real rank $r$ over $Y$ and $X$ respectively. We suppose that $f$ is covered by a map $F\colon \Tot(E')\to \Tot(E)$ which is a linear isomorphism when restricted to each fibre. Suppose $a_E \in H^{r-1}(\Tot_0(E),R)$ is a secondary Thom class of the vector bundle $E$. Then $F^*(a_E)$ is a secondary Thom class of the vector bundle $E'$. 
\end{prop}
\begin{proof}
This follows from the functoriality property of the Thom classes and sequence~\ref{lessecthom}.
\end{proof}
\begin{dfn}
\label{orientedsecthom}
Let $E$ be an $R$-oriented vector bundle with the Thom class $u_E \in H^{r}(\Tot(E),\Tot_0(E),R)$. A~secondary Thom class $a_E$ is called \it{$R$-oriented }\rm if and only if $\delta(a_E)=u_E$. Here $\delta$ is the homomorphism from~\ref{lessecthom}.
\end{dfn}
\begin{prop}
\label{indsecthom}
An $R$-oriented secondary Thom class  $a_E \in H^{r-1}(\Tot_0(E),R)$ is uniquely defined up to the group $\im(p_0^*)$. I.e., if $a_E$ and $a'_E$ are two $R$-oriented secondary Thom classes, then there exists a class $b \in H^{r-1}(X,R)$ such that $a_E=a'_E +p_0^*(b)$. 
\end{prop}
\begin{proof}
Using~\ref{lessecthom} we see that if $a_E$ and $a'_E$ are two $R$-oriented secondary Thom classes, then $a_E-a'_E \in \im(i^*)$; the latter group coincides with the image of $p_0^*$.
\end{proof}

\begin{prop}\label{transgression}
Suppose $X$ is path connected, $E$ is $R$-oriented and $r=\rk E>1$. Let $(E_*^{*,*})$ be the Leray-Serre spectral sequence of $p_0:\Tot_0(E)\to X$ for the cohomology with coefficients in $R$. Let $x_0\in X$ be a point. The group $E_2^{0,r-1}\cong H^{r-1}(E_{0,x_0},R)\cong R$ has a preferred generator $\mathbf{a}$, namely the one that corresponds to the generator of $H^{r-1}(E_{0,x_0},R)\cong H^r(E_{x_0},E_{0,x_0},R)$ obtained by restricting the Thom class of $E$. Then the transgression 
\begin{equation}\label{formula_transgression}
d_r(\mathbf{a})=-e(E)\in H^r(X,R)\cong E_r^{r,0}.
\end{equation}
\end{prop}

\begin{rmk}\label{rmktransgression}
{This result is very likely classical, but we were unable to locate it in the literature. For all our applications (e.g.\ Corollaries~{\ref{maincorappl}} and {\ref{stabilizerisproduct}}) it would suffice to know that $d_r(\mathbf{a})$ and $e(E)$ generate the same cyclic subgroup, which is much easier to prove than formula~\ref{formula_transgression}. This formula, however, allows us to make the statements of our key results (such as Corollary~{\ref{prepmaincorollary}} and Theorem~{\ref{maintheorem}}) more precise, and we also believe it may come in useful in other examples than those that we consider in this paper, so we prove Proposition~{\ref{transgression}} in Appendix~{\ref{app_proofs}}}.
\end{rmk}

\begin{rmk}\label{consistency}
{Both the differential $d_r$ and the preferred generator $\mathbf{a}$ depend on the differential $\delta$ on singular cochains. For definiteness, here and elsewhere below we take $\delta$ to be the dual of the differential on singular chains }(i.e.\ {no extra sign), cf.~{\cite{Switzer75}}, {\cite{Sp81}} and {\cite{H02}}. We note however that if one replaces $\delta$ in each degree $i$ by $n_i\delta$ with $n_i\in\{\pm 1\}$, then the sign in formula~{\ref{formula_transgression}} will not change.}
\end{rmk}

\begin{prop}
\label{injectivity}
The Euler class $e(E)$ is zero if and only if the map $p_0^*\colon H^{*}(X, R) \to H^{*}(\Tot_0(E),R)$ is injective. 
\end{prop}
\begin{proof}
The Leray-Serre spectral sequence of the fibre bundle $p_0\colon\Tot_0 E \to X$ degenerates (i.e., all differentials are zero starting from $d_2$) if and only if the Euler class $e(E)$ is zero.
\end{proof}

Let $G$ be a topological group that acts on $X$ continuously on the left. We say that the vector bundle $E$ is {\it $G$-equivariant} if there is an action of $G$ on $\Tot(E)$ which is compatible with the action of $G$ on $X$ and linear when restricted to each fibre. 

Unless stated otherwise we will assume in the rest of this subsection that $E$ is $G$-equivariant and $R$-oriented, and the action of $G$ preserves the $R$-orientation of $E$. Let $\vp\colon G\times X \to X$ be the action map, and let $p_2\colon G\times X \to X$ be the projection map. Let $\psi$ be an isomorphism of the vector bundles $\vp^*(E)$ and $p_2^*(E)$.
Note that if $E$ is a $G$-equivariant vector bundle on $X$, then there is a preferred choice of $\psi$: the total space of $\vp^*(E)$ (respectively of $p_2^*(E)$) is the space of all triples $$\{(g,x,e)\mid g\in G, x\in X, e\in\Tot(E)\}$$ such that $e\in E_{gx}$ (respectively $e\in E_x$). The isomorphism $\psi:\vp^*(E)\to p_2^*(E)$ is then given by $(g,x,e)\mapsto (g,x,g^{-1}e)$. Note also that the spaces $\Tot_0(\vp^*(E))$ and $\Tot_0(p_2^*(E))$ can be identified with $G\times\Tot_0(E)$:
\begin{align}
\label{equivarianttot}
\eta:\Tot_0(\vp^*(E))\to G\times\Tot_0(E) ,& (g,x,e)\mapsto (g,g^{-1}e)\\
\label{equivarianttot1}
\vartheta:\Tot_0(p_2^*(E))\to G\times\Tot_0(E) ,& (g,x,e)\mapsto (g,e).
\end{align}

These isomorphisms commute with $\psi$, i.e.\ $\vartheta\circ\psi=\eta$. Let us denote the action map $G\times \Tot_0(E) \to \Tot_0(E)$ and the projection $G \times \Tot_0(E) \to \Tot_0(E)$ by $\Phi$ and $P_2$ respectively.
\begin{lmm} 
\label{actionsecthom}
For any $R$-oriented secondary Thom class $a_E\in H^*(\Tot_0(E),R)$ of a $G$-equivariant $R$-oriented vector bundle $E$ the difference $\Phi^*(a_E) - P_2^*(a_E)\in H^*(G\times \Tot_0(E),R)$ belongs to the image of the homomorphism $(\id \times p_0)^*$.
\end{lmm}
\begin{proof}
If we pull $P_2^*(a_E)$ back along $\vartheta$, we will get a secondary Thom class of the bundle $p_2^*(E)$ by Proposition~\ref{functorsecthom}. Similarly, by pulling $\Phi^*(a_E)$ back along $\psi^{-1}\circ\eta$ we get a secondary Thom class of the same bundle. 
Let us check that these two secondary Thom classes correspond to the same $R$-orientation, i.e., that they map to the same Thom class under the connecting homomorphism $$H^*(\Tot_0(p_2^*(E)),R)\to H^{*+1}(\Tot(p_2^*(E)),\Tot_0(p_2^*(E)),R).$$ 
It suffices to check this over each point $(g,x)$ of $G\times X$. Since the action of $G$ is orientation-preserving, the maps $\Phi\circ\eta\circ\psi^{-1}=\Phi\circ\vartheta$ and $P_2\circ\vartheta$ induce the same orientation on the fibre of $p_2^*(E)$ over $(g,x)$. The lemma now follows from Proposition \ref{indsecthom}.
\end{proof}
\begin{ntt}
\label{partialaction}
Suppose that $X$ is path connected and the Euler class of $E$ is zero; let $a_E$ be a secondary Thom class. Since the cohomology group $H^{r-1}(\Tot_0 (E),R)$ is isomorphic as an $R$-module to the direct sum $p_0^*(H^{r-1}(X),R)\oplus R[a_E]$, there is a homomorphism of $R$-modules
\begin{equation*}
\partial \Phi^*\colon H^{r-1}(\Tot_0 (E),R)\to H^{r-1}(G\times X, R)
\end{equation*}
such that $(\id \times p_0)^*\circ(\partial \Phi^*) = \Phi^* - P_2^*.$ In order to construct such a map $\partial\Phi^*$ recall that $p_0^*$ is injective (Proposition~\ref{injectivity}). So we can set $\partial\Phi(p_0^*(x))=\varphi^*(x)-p_2^*(x)$ for $x\in H^{r-1}(X,R)$, and we define $\partial \Phi^*(a_E)$ using Lemma~\ref{actionsecthom}.
\end{ntt} 
\begin{dfn}
\label{indeterminacy submodule}
The graded $R$-module 
\begin{equation*}
M_X^*:= \im\big((\vp^*-p_2^*)\colon H^*(X,R)\to H^*(G\times X,R)\big)
\end{equation*}
 will be called the \it{indeterminacy submodule}\rm. We will omit the subscript if the space $X$ is clear from the context.
\end{dfn}
\begin{cor}
\label{patrialactioniswelldefined}
Let $a_E$ and $a'_E$ be $R$-oriented secondary Thom classes of $E$.  
Then the difference 
\begin{equation*}
\partial \Phi^*(a_E)-\partial \Phi^*(a'_E)
\end{equation*}
belongs to the indeterminacy submodule $M^{r-1}$.
\end{cor}
\begin{proof}
Indeed, by Proposition~\ref{indsecthom} the difference $a_E -a'_E$ is equal to $p^*_0(b)$ for some $b \in H^{r-1}(X,R)$. Since the vector bundle $E$ is $G$-equivariant, we have
\begin{equation*}
(\id \times p_0)^*\circ(\partial \Phi^*)(a_E-a'_E)=(\Phi^*-P_2^*)(a_E-a'_E) =(\Phi^*-P_2^*)(p_0^*(b))=(\id\times p_0)^*\big((\vp^*-p_2^*)(b)\big). 
\end{equation*}
The last equality here follows from the fact that $p_0\circ \Phi=\varphi\circ (\id\times p_0)$, and similarly for $P_2$ and $p_2$. The map $(\id\times p_0)^*$ is injective, as $p_0$ is, so we conclude that $\partial \Phi^*(a_E - a'_E)$ is $(\vp^*-p_2^*)(b)\in M^{r-1}$.
\end{proof}
We now wish to describe the cohomology class $\partial \Phi^*(a_E)\in H^{r-1}(X,R)/M^{r-1}$ in more concrete terms. Recall that if we are given two continuous maps
\begin{equation}
\label{nhqdiagram}
\begin{tikzcd}[column sep=large]
A \arrow[shift left=.75ex]{r}{f}
  \arrow[shift right=.75ex,swap]{r}{g}
&
B
\end{tikzcd}
\end{equation}
of topological spaces, then the homotopy colimit of \ref{nhqdiagram} is called the {\it homotopy coequaliser} of $f$ and $g$. On the level of spaces it can be constructed as follows: we take the cylinders of $f$ and $g$ and identify the copies of both $A$ and $B$. The resulting space $C$ is equipped with a map $q:B\to C$.
\begin{rmk}
\label{puppe}
Given a homotopy coequaliser diagram
\begin{equation*}
\begin{tikzcd}[column sep=large]
A \arrow[shift left=.75ex]{r}{f}
  \arrow[shift right=.75ex,swap]{r}{g}
&
B \arrow{r}{q}
&
C.
\end{tikzcd}
\end{equation*}
there exists a long exact sequence of cohomology groups:
\begin{align}\label{longexactcoeqgen}
\cdots&\to  H^{*-1}(B,R)\xrightarrow{f^*-g^*} H^{*-1}(A,R)\xrightarrow{\delta} H^{*}(C,R) \xrightarrow{q^*} \\ 
&\to H^{*}(B,R) \xrightarrow{f^*-g^*} H^{*}(A,R)\xrightarrow{\delta} H^{*+1}(C,R)\to \cdots  \nonumber
\end{align}
Notice that the homomorphism $\delta\colon H^{*-1}(A,R)\to H^{*}(C,R)$ in this sequence is induced by the natural map 
\begin{equation}
\label{puppemap}
\Delta\colon C \to \Sigma(A_+).
\end{equation}
and
\begin{equation*}
B \xrightarrow{q} C \xrightarrow{\Delta} \Sigma(A_+),
\end{equation*}
is a cofibre sequence.
\end{rmk}

\begin{dfn} 
\label{nhq}
Recall that $\vp$ denotes the action map $G\times X\to X$. We define the {\it naive homotopy quotient $\naive{X}{G}$ of $X$ by the $G$-action }\rm as the homotopy coequaliser of the maps $\vp$ and $p_2$. We will denote the map $X\to \naive{X}{G}$ by $q_X$ (or just $q$ if it is clear which space $X$ is meant).
\end{dfn}

Recall that $EG$ is a contractible topological space with a continuous right free $G$-action such that the quotient map $EG\to (EG)/G$ is a principal $G$-bundle.

\begin{dfn}
\label{hq}
The {\it homotopy quotient} or the {\it Borel quotient} $X_{hG}$ of $X$ by $G$ is the quotient space $(X\times EG)/G$ with respect to the diagonal action of $G$ given by $g(x,e)=(gx,eg^{-1})$. There is a natural map $\alpha\colon X\to X_{hG}$ and the homotopy fibre of $\alpha$ is~$G$.
\end{dfn}

\begin{rmk}
\label{justification of naive homotopy quetient}
Note that the ordinary colimit in the category of topological spaces of the diagram
\begin{equation*}
\begin{tikzcd}[column sep=large]
G\times X \arrow[shift left=.75ex]{r}{\varphi}
  \arrow[shift right=.75ex,swap]{r}{p_2}
&
X
\end{tikzcd}
\end{equation*}
is the quotient space $X/G$, which explains the analogy between $X/G$ and $\naive{X}{G}$. We call $\naive{X}{G}$ the \it{naive }\rm homotopy quotient because it can be viewed as a first approximation to the ``true'' homotopy quotient $X_{hG}$, see Definition~\ref{hq}. Indeed, the space $X_{hG}$ is homotopy equivalent to the geometric realisation of the following simplicial bar construction:
\begin{equation*}
\begin{tikzcd}[column sep=4ex]
\cdots
\arrow[r, shift left=.5ex, dashrightarrow] 
\arrow[r, shift right=0.5ex, dashleftarrow]
& G\times G\times G\times X
\arrow[r, shift left=3ex, rightarrow] 
\arrow[r, shift right=2ex, leftarrow]
\arrow[r, shift left=1ex, rightarrow] 
\arrow[r, shift right=0ex, leftarrow]
\arrow[r, shift right=1ex, rightarrow] 
\arrow[r, shift left=2ex, leftarrow]
\arrow[r, shift right=3ex, rightarrow]
&
G\times G\times X
\arrow[r, shift left=2ex, rightarrow] 
\arrow[r, shift right=1ex, leftarrow]
\arrow[r, shift left=0ex, rightarrow] 
\arrow[r, shift left=1ex, leftarrow]
\arrow[r, shift right=2ex, rightarrow] 
&
G\times X
\arrow[r, shift left=1ex, rightarrow] 
\arrow[r, shift right=0ex, leftarrow]
\arrow[r, shift right=1ex, rightarrow] 
&
X.
\end{tikzcd}    
\end{equation*}
Here each face map is either the action map $\varphi$, the multiplication map of the group $G$, or the projection. The degeneracy maps are induced by the inclusion of the unit element into $G$, see~\cite{May75}*{Section~7} for more details. Under mild conditions on $X$ and $G$ this simplicial space is good in the sense of G.~Segal. Therefore by~\cite{Seg74}*{Appendix~A} its geometric realisation is weakly homotopy equivalent to the geometric realisation of the corresponding semisimplical space
\begin{equation*}
\begin{tikzcd}[column sep=4ex]
\cdots \arrow[dashrightarrow]{r}
&
G\times G \times G \times X \arrow[shift left=1.5ex]{r}
  \arrow[shift left=.5ex]{r}
  \arrow[shift right=.5ex]{r}
  \arrow[shift right=1.5ex]{r}
&
G\times G \times X \arrow[shift left=1ex]{r}
  \arrow{r}
  \arrow[shift right=1ex]{r}
&
G\times X \arrow[shift left=.5ex]{r}
  \arrow[shift right=.5ex]{r}
&
X.
\end{tikzcd}    
\end{equation*}

The geometric realisation is the homotopy colimit of this diagram, and the naive homotopy quotient is simply the homotopy colimit of the first stage.

%
\end{rmk}

\begin{ntt}
\label{naivetohomotopy}
Note that the homotopy colimit of a diagram naturally maps to the ordinary colimit of the same diagram. So there exists a natural map from $\naive{X}{G}$ to $X/G$. We will denote the following composition by $\gamma'_X$:
\begin{equation*}
    \gamma'_X\colon \naive{X}{G} \xrightarrow{\simeq} \naive{(X\times EG)}{G} \xrightarrow{\phantom{\simeq}}  (X\times EG)/G=X_{hG}.
\end{equation*}
Moreover, note that the following diagram commutes:
\begin{equation}
\label{compatability of naive and homotopy quotients}
\begin{tikzcd}[column sep=large]
X \arrow{r}{q_X} \arrow[swap]{rd}{\alpha}
& \naive{X}{G} \arrow{d}{\gamma'_{X}} \\
&X_{hG}.
\end{tikzcd}
\end{equation}
\end{ntt}

\begin{exmp}
\label{underpoint}
If $X$ is a one point space, then $\naive{X}{G}\simeq \Sigma (G_+)\simeq \Sigma G \vee S^1$. Moreover, we claim that $\gamma'_{X}$ is then nullhomotopic on the second component of the wedge, $S^1$, and on the first component it is the map $\gamma\colon \Sigma G \to BG$ obtained from the natural homotopy equivalence $G\xrightarrow{\simeq} \Omega BG$ by the suspension-loop space adjunction.

This can be seen as follows. The space $pt_{nG}\simeq \Sigma G_+$ is the quotient of $G\times I$ by the equivalence relation that collapses $G\times\{0,1\}$. Let $e_0\in EG$ be a base point. Let $F:EG\times I\to EG$ be a homotopy such that $F(-,0)$ is the constant map that takes everything to $e_0$ and $F(-,1)$ is the identity.
We will denote the free right action of $G$ on $EG$ by $\cdot$ and use the square brackets denote the image of an element of $EG$ in $BG$.

The space $EG_{nG}$ is the quotient space of $G\times EG\times I$ by an equivalence relation $\sim$. Note that in the definition of the naive homotopy quotient we use left actions, so we need to transform the natural right action of $G$ on $EG$ into a left one as above in Definition~\ref{hq}, and $\sim$ is generated by $$(g_1,e,0)\sim (g_2,e,0)\sim (g_1,e\cdot g_1,1)\sim (g_2,e\cdot g_2,1),g_1,g_2\in G,e\in EG.$$ The map $EG_{nG}\to BG$ is induced by the map $G\times EG\times I\to BG$ that takes $(g,e,t)\in G\times EG\times I$ to $[e]$. The homotopy equivalence $pt_{nG}\to EG_{nG}$ can be given explicitly as the map induced by $$G\times I\ni (g,t)\mapsto (g,F(e_0\cdot g,t),t).$$ The map $\gamma'_{pt}$ is by definition the composition of this map with $EG_{nG}\to BG$. The resulting map $\Sigma G_+\to BG$ is induced by
\begin{equation}\label{sigmagbg}
G\times I\ni (g,t)\mapsto [F(e_0\cdot g,t)],
\end{equation}
and it has the properties that we stated above.

\end{exmp}
We will now discuss vector bundles over the naive homotopy quotient $\naive{X}{G}$.  The category of vector bundles over the space $\naive{X}{G}$ is equivalent to the category of pairs $(E, \psi)$, where $E$ is a vector bundle over $X$ and $\psi \colon\vp^* E \to p_2^* E$ is an isomorphism of vector bundles over $G\times X$.
\begin{ntt} 
\label{nhdescent}
Suppose that $(E,\psi)$ is a $G$-equivariant vector bundle over $X$. Let $\naive{E}{G}$ be the vector bundle over $\naive{X}{G}$ constructed using the pair $(E, \psi)$.
\end{ntt}
Let us apply Remark~\ref{puppe} to the homotopy coequaliser of $\varphi$ and $p_2:G\times X\to X$. The resulting long exact sequence reads:
\begin{align}
\label{lesnhq}
\cdots&\to  H^{*-1}(X,R)\xrightarrow{\vp^*-p_2^*} H^{*-1}(G\times X,R)\xrightarrow{\delta} H^{*}(\naive{X}{G},R) \xrightarrow{q_X^*} \\ 
&\to H^{*}(X,R) \xrightarrow{\vp^*-p_2^*} H^{*}(G\times X,R)\xrightarrow{\delta}  H^{*+1}(\naive{X}{G},R)\to \cdots  \nonumber
\end{align}
\begin{dfn} 
\label{seceqeuler}
Recall that $E$ is a $G$-equivariant real $R$-oriented vector bundle of rank $r$ over $X$ and suppose that the Euler class of $E$ is zero. A cohomology class $se_G(E) \in H^{r-1}(G\times X,R)$ is called {\it a secondary equivariant Euler class} of $E$ if and only if
\begin{equation*}
\delta(se_G(E))=e(\naive{E}{G}) \in H^{r}(\naive{X}{G},R). 
\end{equation*}
Here $\naive{E}{G}$ is the vector bundle defined in Notation~\ref{nhdescent} and $\delta\colon H^{r-1}(G\times X,R)\xrightarrow{\delta} H^{r}(\naive{X}{G},R)$ is the boundary homomorphism from long exact sequence~\ref{lesnhq}
\end{dfn}
\begin{rmk}
A secondary equivariant Euler class of the vector bundle $E$ is unique up to the indeterminacy submodule~$M^{r-1}_X$. Indeed, by long exact sequence~\ref{lesnhq} the kernel of $\delta$ is precisely~$M^{r-1}_X$.
\end{rmk}

\begin{prop} 
\label{secthomseceuler}
Recall that $E$ is a $G$-equivariant $R$-oriented real vector bundle of rank $r$. Suppose $X$ is path-connected and $e(E)=0$. Then for every $R$-oriented secondary Thom class $a_E \in H^{r-1}(\Tot_0(E),R)$ and for every secondary equivariant Euler class $se_G(E) \in H^{r-1}(G\times X, R)$ the equality 
\begin{equation*}
\partial \Phi^*(a_E) = - se_G(E) 
\end{equation*}
holds up to the indeterminacy submodule $M^{r-1}_X$.
\end{prop}
\begin{proof}
To prove the proposition it would suffice to show that $\delta (\partial \Phi^*(a_E)) = -\delta(se_G(E)) =- e(\naive{E}{G})$. 
Let $F_0$ be the fibre of the projection map $p_0\colon \Tot_0(\naive{E}{G}) \to \naive{X}{G}$ over a point $x \in \im(q_X)$. Let $i_x\colon F_0 \to \Tot_0(E)$ denote the embedding of the fibre. Since the vector bundle $E$ is $R$-oriented, there is a preferred $R$-generator $a_{F} \in H^{r-1}(F_0,R)$. The image of $a_F$ under the transgression homomorphism in the Leray-Serre spectral sequence of the bundle $\Tot_0(\naive{E}{G})\to \naive{X}{G}$ is $-e(\naive{E}{G})$ (see Proposition~\ref{transgression}). 

Recall also that the transgression homomorphism can be constructed as explained e.g.\ in \cite{Mc01}*{Section~6.2}, so $-e(\naive{E}{G})=(p_0^*)^{-1}(\delta(a_{F}))$, where $\delta$ and $p_0^*$ are given in the diagram below
 \begin{equation*}
\begin{tikzcd}
H^{r-1}(F_0,R) \arrow{r}{\delta} 
&H^{r}(\Tot_0(\naive{E}{G}), F_0, R)
&H^{r}(\naive{X}{G},R) \arrow{l}[swap]{p_0^*}.
\end{tikzcd}
\end{equation*}
Note that in our case there is a unique element of $H^{r}(\naive{X}{G},R)$ that goes to $\delta(a_F)$ under $p_0^*$. Notice that the topological space $\Tot_0(\naive{E}{G})$ is homeomorphic to the space $\naive{(\Tot_0E)}{G}$. Consider the following commutative diagram:
\begin{equation}
\label{cd11}
\begin{tikzcd}
H^{r-1}(F_0,R) \arrow{r}{\delta} 
&H^{r}(\Tot_0(\naive{E}{G}), F_0, R)
&H^{r}(\naive{X}{G},\{x\},R) \arrow{l}[swap]{p_0^*} \\
H^{r-1}(\Tot_0(E),R)\arrow{r}{\delta} \arrow{u}{i^*_x} \arrow{d}[swap]{\Phi^* - P_2^* }
&H^{r}(\Tot_0(\naive{E}{G}), \Tot_0(E), R) \arrow{u}
&H^{r}(\naive{X}{G}, X, R) \arrow{u}{j^*}\arrow{l} \\
H^{r-1}(G \times \Tot_0(E),R) \arrow{r}{\cong} 
&\widetilde{H}^{r}(\Sigma(G \times \Tot_0(E))_+, R) \arrow{u}{\cong}
&\widetilde{H}^{r}(\Sigma(G \times X)_+,R) \arrow{l} \arrow{u}{\cong}
\end{tikzcd}
\end{equation}

In this diagram the lower left square is constructed as follows: given a homotopy coequaliser diagram
\begin{equation}\label{cd11b}
\begin{tikzcd}[column sep=large]
A \arrow[shift left=.75ex]{r}{f}
  \arrow[shift right=.75ex,swap]{r}{g}
&
B \arrow{r}{q}
&
C\arrow{r}
&
\Sigma A_+
\end{tikzcd}
\end{equation}
as above we have the following commutative square:
\begin{equation}
\label{cd11a}
\begin{tikzcd}
H^{r-1}(B,R)\arrow{r}\arrow{d}[swap]{f^*-g^*} & H^r(C,B,R)\\
H^{r-1}(A,R) \arrow{r}{\cong} & \widetilde{H}^r(\Sigma A_+,R)\arrow{u}{\cong}
\end{tikzcd}
\end{equation}
This gives the lower left square of \ref{cd11} if we set $A=G\times\Tot_0(E), B=\Tot_0(E), C=\Tot_0(\naive{E}{G})$. Furthermore, diagrams \ref{cd11a} are functorial with respect to morphisms of homotopy coequaliser diagrams \ref{cd11b}. The lower right square of \ref{cd11} is obtained from the map of the right arrows of diagrams \ref{cd11a} induced by
\begin{equation}\label{coeqdiagrammap}
\begin{tikzcd}
G\times X
\arrow[shift left=.75ex]{r}{\varphi}
 \arrow[shift right=.75ex,swap]{r}{p_2}
&
X \arrow{r}{q}
&
\naive{X}{G}\arrow{r}
&
\Sigma (G\times X)_+\\
G\times\Tot_0(E) \arrow[shift left=.75ex]{r}{\Phi}
  \arrow[shift right=.75ex,swap]{r}{P_2}
  \arrow{u}{\id\times p_0}
&
\Tot_0(E) \arrow{r}{q}
\arrow{u}{p_0}
&
\Tot_0(\naive{E}{G})\arrow{r}\arrow{u}
&
\Sigma (G\times\Tot_0(E))_+\arrow{u}
\end{tikzcd}
\end{equation}

Now let us apply diagram~\ref{cd11} to our situation. Since $e(E)=0$, there exists an $R$-oriented secondary Euler class~$a_E\in H^{r-1}(\Tot_0(E),R)$ such that $i_x^*(a_E)=a_F$, and so $(p_0^*)^{-1}\circ \delta\circ i_x^*$ applied to $a_E$ is $-e(\naive{E}{G})$.

Let us now see what happens to $a_E$ if we take it along a different path in \ref{cd11}. Recall that $\Phi^* - P_2^*$ decomposes as $(\id\times p_0)^*\circ \partial\Phi^*$. The following diagram is the map of the bottom arrows of diagrams \ref{cd11a} induced by \ref{coeqdiagrammap}.
\begin{equation}\label{cd11c}
\begin{tikzcd}
H^{r-1}(G\times\Tot_0(E),R)\arrow{r}{\cong}& \widetilde{H}^r(\Sigma(G\times\Tot_0(E))_+,R)\\
H^{r-1}(G\times X,R) \arrow{r}{\cong}[swap]{\Sigma}\arrow{u}{\id\times p_0^*} & \widetilde{H}^r(\Sigma (G\times X)_+,R)\arrow{u}
\end{tikzcd}
\end{equation}
It follows from this diagram that in~\ref{cd11}, the image of $a_E$ in $\widetilde{H}^r(\Sigma(G\times\Tot_0(E))_+,R)$ is the same as the image of $\Sigma(\partial\Phi^*(a_E))\in \widetilde{H}^r(\Sigma(G\times X)_+,R)$ where $\Sigma$ denotes the bottom arrow of \ref{cd11c}. Finally, note that the composition of $\Sigma$ and the two right arrows of \ref{cd11} is simply the homomorphism $\delta$ from \ref{lesnhq}. We can now conclude that $\delta(\partial\Phi^*(a_E))= - e(\naive{E}{G})$, which is what we wanted to prove.
\end{proof}

\subsection{Calculating secondary equivariant Euler classes}\label{maps}
Here we will describe a method to calculate the cohomology class~$\partial \Phi^*(a_E)$. In this subsection $G$ is a path-connected topological group that acts continuously on a path-connected topological space $X$, and $E$ is an $R$-oriented $G$-equivariant real vector bundle over $X$ of rank $r$ where $R$ is a commutative ring with identity; we will assume that the action of $G$ preserves the $R$-orientation of $E$. Consider the fibre sequence
\begin{equation}
\label{mainfs}
X \xrightarrow{\alpha} X_{hG} \xrightarrow{\beta} BG.
\end{equation} 
\begin{ntt}
\label{mainidealeq}
We set $I^*(X,R)$ to be the kernel of the natural map $\alpha^*\colon H^*(X_{hG},R)\to H^*(X,R)$.
\end{ntt}

We will now describe one of our main constructions. It follows from commutative diagram \ref{compatability of naive and homotopy quotients} that $(\gamma'_X)^*(I^*(X,R))\subset \ker(q_X^*)\subset H^*(X_{nG},R)$. So using long exact sequence~\ref{lesnhq} we obtain $(\gamma'_X)^*(I^*(X,R)) \subset \delta(H^{*-1}(G\times X,R))$, where $\delta$ is the connecting homomorphism in~\ref{lesnhq}. Recall that $M_X^*$ denotes the indeterminacy submodule (Definition~\ref{indeterminacy submodule}) and $M_X^*$ is precisely the kernel of $\delta$. This motivates the following definition.

\begin{dfn}
\label{mainhomomorphism}
We define a homomorphism
\begin{equation*}
S\colon I^*(X,R) \to H^{*-1}(G\times X,R)/M_X^{*-1}
\end{equation*}
as follows. For an $x\in I^*(X,R)$ we set $S(x)$ to be the element of $H^{*-1}(G\times X,R)/M_X^{*-1}$ that goes to $(\gamma'_X)^*(x)\in H^*(X_{nG},R)$ under the homomorphism induced by $\delta$.
\end{dfn}

The homomorphism $S$ will be a key ingredient in our applications. The first property of $S$ that we will need is that both its source and target have natural structures of (right) $H^*(X_{hG},R)$-modules which make $S$ an $H^*(X_{hG},R)$-linear map. To prove this we need the following lemma.

\begin{lmm}\label{coeqmod}
Suppose
\begin{equation}\label{coeqmoddiag}
\begin{tikzcd}[column sep=large]
A \arrow[shift left=.75ex]{r}{f}
  \arrow[shift right=.75ex,swap]{r}{g}
&
B \arrow{r}{q}\arrow{d}{h} & C\arrow{dl}{i}\\
 & Z &
\end{tikzcd}
\end{equation}
is a homotopy commutative diagram of topological spaces and continuous maps with the top row being a homotopy coequaliser diagram, and such that $h\circ f=h\circ g$. We make $H^*(A,R), H^*(B,R)$ and $H^*(C,R)$ into right graded $H^*(Z,R)$-modules using the maps $h\circ f=h\circ g, h$ and $i$ respectively. Then all maps in long exact sequence~\ref{longexactcoeqgen} will be maps of right $H^*(Z,R)$-modules.
\end{lmm}

\begin{proof}
It suffices to prove the lemma in the case $Z=C, i=\id, h=q$, so let us assume this. All maps in the long exact cohomology sequence of an excisive couple $(A_1,A_2), A_1\supset A_2$ are maps of right $H^*(A_1,R)$-modules, see e.g.\ \cite{Switzer75}*{Proposition 13.56, part (iii)}. Sequence \ref{longexactcoeqgen} is the long exact sequence of the couple $(C,B)$. By the definition of a homotopy coequaliser there is a natural map $F:A\times I\to C$ that takes $A\times\{0,1\}$ to $B$. Let $\bar F:\Sigma A_+\to C/B$ be the resulting map of quotient spaces. Using the commutative diagram 
\begin{equation}\label{diagcoeq}
\begin{tikzcd}
B \arrow[r,"\subset"] & C \arrow[r] & C/B \\
A\times \{0,1\} \arrow[r,"\subset"]\arrow[u,"f\sqcup g"] & A\times I \arrow[r]\arrow[u,"F"] & \Sigma A_+\arrow[u,"\bar F","\cong"']
\end{tikzcd}
\end{equation}
we see that the structure of a right $H^*(C,R)$-module on $H^*(A,R)$ induced by $h\circ f=h\circ g$ is the same as the one induced by $$H^*(A,R)\cong \widetilde{H}^{*+1}(\Sigma A_+,R)\cong \widetilde{H}^{*+1}(C/B,R).$$

In more detail, let $\Sigma:H^*(A,R)\to \widetilde{H}^{*+1}(\Sigma A_+,R)$ denote the suspension isomorphism. Since the lower row of diagram \ref{diagcoeq} induces a long exact sequence of right $H^*(A,R)$-modules, we conclude that $\Sigma$ is a map of right $H^*(A,R)$-modules too. To compare the two $H^*(C,R)$-module structures on $H^*(A,R)$ we need to take elements $a\in H^*(A,R), c\in H^*(C,R)$ and compare $aF^*(c)$ with
$$\Sigma^{-1}\bigg(\bar F^*\Big(\big((\bar F^*)^{-1}(\Sigma a)\big)c\Big)\bigg)=\Sigma^{-1}\big((\Sigma a) F^*(c)\big).$$ The resulting elements of $H^*(A,R)$ coincide because $\Sigma$ is $H^*(A,R)$-linear.
\end{proof}


The following diagram homotopy commutes:
\begin{equation}\label{diagrquot}
\begin{tikzcd}
G\times X \arrow[shift left=.75ex]{r}{p_2}\arrow{r}[shift right=.75ex,swap]{\varphi} &X \arrow{d}{\alpha}\arrow{r}{q_X}& \naive{X}{G} \arrow{r}\arrow{ld}{\gamma'_X}& \Sigma (G\times X)_+ \\
& X_{hG} & &
\end{tikzcd}
\end{equation}
(see Notation \ref{naivetohomotopy}). We introduce $H^*(X_{hG},R)$-module structures on $H^*(G\times X,R)$, $H^*(X,R)$, and $H^*(\naive{X}{G},R)$ as in Lemma~\ref{coeqmod}.


\begin{cor}
\label{functorS}
The homomorphism $S$ is a homomorphism of $H^*(X_{hG},R)$-modules. Moreover, $S$ is a natural transformation of functors from the category of topological $G$-spaces and $G$-equivariant maps to the category of graded abelian groups.
\end{cor}
\begin{proof} By Lemma~\ref{coeqmod} all maps in sequence~\ref{lesnhq} are maps of $H^*(X_{hG},R)$-modules, and hence so is $S$. Finally, the last statement of the corollary, namely that $S$ is a natural transformation, follows from the fact that up to homotopy all maps in diagram~\ref{diagrquot} are.
\end{proof}

\begin{dfn}
\label{equivarianteulerclass}
The homotopy quotient $(\Tot E)_{hG}$ is the total space of of a vector bundle over the quotient $X_{hG}$; we will denote this bundle $E_{hG}$. We define the {\it equivariant Euler class} $e_G(E)$ of $E$ as follows:
\begin{equation*}
e_G(E) = e(E_{hG})\in H^*(X_{hG},R).
\end{equation*}
\end{dfn}
\begin{prop} 
\label{seceulerse}
Suppose that $e(E)=0$, $r=\rk(E)$. 
Then $e_G(E) \in I^r(X,R)$ and for every secondary equivariant Euler class $se_G(E)\in H^{r-1}(G\times X,R)$ the image of $se_G(E)$ in $H^{r-1}(G\times X,R)/M_X^{r-1}$ is $S(e_G(E))$.
\end{prop}
\begin{proof}
We need to show that the image of $S(e_G(E))$ in $H^r(X_{nG},R)$ is $e(\naive{E}{G})$, where $\naive{E}{G}$ is the vector bundle defined in Notation~\ref{nhdescent}. By Definition \ref{mainhomomorphism}, to do this it would suffice to prove that 
$(\gamma'_X)^*(e_G(E))=e(\naive{E}{G}),$
which follows from the fact that
$(\gamma'_X)^*(E_{hG})\cong \naive{E}{G}$ as vector bundles over $\naive{X}{G}$.
\end{proof}
\begin{ntt}
\label{subideal1}
We set $I_1^*(X,R)$ to be the ideal generated by the image of the abelian group $\widetilde{H}^*(BG,R)$ under the homomorphism $\beta^*$. \par
Let us also define the following homomorphism
\begin{align*}
k\colon  \widetilde{H}^*(BG,R) \otimes H^*(X_{hG},R) &  \to \widetilde{H}^*(X_{hG},R), \\
x\otimes y & \mapsto  \beta^*(x) \smile y.
\end{align*}
The image of $k$ is precisely the ideal $I_1^*(X,R)$. 
\end{ntt}


\begin{ntt}
\label{gamma}
Let $\gamma\colon \Sigma G \to BG$ be the adjoint of the homotopy equivalence $G\xrightarrow{\simeq} \Omega BG$, see Example~\ref{underpoint}. Let $\bar\gamma$ be the homomorphism 
\begin{equation*}
H^*(BG,R) \xrightarrow{\gamma^*} H^*(\Sigma G,R)\xrightarrow{} \widetilde{H}^*(\Sigma G_+,R) \xrightarrow{\cong} H^{*-1}(G,R)
\end{equation*}
where the arrow in the middle is induced by $G_+\to G$. Then one can define a homomorphism 
\begin{equation*}
\widetilde{S}\colon \widetilde{H}^p(BG,R) \otimes H^q(X_{hG},R)  \rightarrow \widetilde{H}^{p+q-1}(G\times X,R)
\end{equation*}
{by setting} $x\otimes y \mapsto  \bar\gamma(x)  \times \alpha^*(y)$.
\end{ntt}

\begin{rmk}\label{gamma and transgression}
The homomorphism $\bar\gamma$ is an example of the {\it cohomology suspension} in the special case of the fibre bundle $G\to EG\to BG$,~\cite{S51}*{Chapitre I, n$^\circ$ 3},~\cite{Mc01}*{\S6.2}. In particular, $\bar\gamma\colon H^p(BG,R)\to H^{p-1}(G,R)$ is the right inverse of the transgression $d_p\colon E^{0,p-1}_p \to E^{p,0}_p$ in the Leray-Serre spectral sequence {of this} fibre bundle, i.e.\ the following diagram commutes~\cite{Mc01}*{Proposition 6.10}:
\begin{equation*}
\begin{tikzcd}
H^p(BG,R)\arrow{r}{\bar\gamma} \arrow[d,twoheadrightarrow]
&H^{p-1}(G,R) \\
E^{p,0}_p 
&E^{0,p-1}_p. \arrow{l}{\cong}[swap]{d_p} \arrow[u, hook, swap]
\end{tikzcd}
\end{equation*}
\end{rmk}

We now prove the main statement about the homomorphism $S$, which will allow us to compute it.
\begin{prop}
\label{mainproposition}
The following diagram is commutative:
\begin{equation*}
\begin{tikzcd}[column sep=large]
 \widetilde{H}^*(BG,R) \otimes H^*(X_{hG},R) \arrow{r}{\widetilde{S}} \arrow{d}{k}
&H^{*-1}(G\times X,R) \arrow{d}\\
I^*(X,R) \arrow{r}{S}
&H^{*-1}(G\times X,R)/M^{*-1}_X.
\end{tikzcd}
\end{equation*}
\end{prop}
%
\begin{proof} 
Since all homomorphisms in the diagram are maps of $H^*(X_{hG},R)$-modules,
in order to prove the proposition it would suffice to show that for every $x\in \widetilde{H}^*(BG,R)$ the image of $\widetilde{S}(x\otimes 1)$ in $H^{*-1}(G\times X,R)/M^{*-1}_X$ is
\begin{equation*}
S(k(x\otimes 1))=S(\beta^*(x)).
\end{equation*}
This amounts to checking that
\begin{equation}
\label{sufficient}
\delta(\widetilde{S}(x\otimes 1))=(\gamma'_X)^*(\beta^*(x)) 
\end{equation}
for all $x\in \widetilde{H}^*(BG,R)$, see Definition \ref{mainhomomorphism}. \par
Since the map $\gamma'_X\colon \naive{X}{G} \to X_{hG}$ is natural with respect to $X$, we get the homotopy commutative diagram:
\begin{equation*}
\begin{tikzcd}\label{auxdiag}
\naive{X}{G}\arrow{r}{\varepsilon} \arrow{d}{\gamma'_X}
&\naive{pt}{G} \arrow{d}{\gamma'_{pt}} \\
X_{hG} \arrow{r}{\beta}
&BG
\end{tikzcd}
\end{equation*}
Recall that the space $\naive{pt}{G}$ is homotopy equivalent to $\Sigma (G_+)\simeq S^1\vee\Sigma G$; the map $\gamma'_{pt}:\naive{pt}{G}\to BG$ was described in Example~\ref{underpoint}. We denote the resulting map $\Sigma G_+\to BG$ by $\gamma'$ dropping the subscript.
%

Let $\Delta\colon \naive{X}{G} \to \Sigma(G\times X)_+$ be the connecting morphism \ref{puppemap} in the Puppe sequence for the coequaliser diagram of $\varphi$ and $p_2$ (see Definition~\ref{nhq}) and let $p_1\colon G\times X \to G$ be the projection map. Recall that the induced homomorphism
\begin{equation*}
H^{*-1}(G\times X,R)\xrightarrow{\cong} \widetilde{H}^*(\Sigma(G\times X)_+,R) \xrightarrow{\Delta^*} H^*(\naive{X}{G},R) 
\end{equation*}
coincides with the connecting homomorphism $\delta$ in exact sequence~\ref{lesnhq}. Furthermore, using~\ref{auxdiag} we get the commutative diagram
\begin{equation*}
\begin{tikzcd}[column sep=large]
&\widetilde{H}^*(\Sigma(G\times X)_+,R) \arrow{dl}[swap]{\Delta^*} & H^{*-1}(G\times X,R)\arrow{l}[swap]{\cong}\\
H^*(\naive{X}{G},R) & \widetilde{H}^*(\Sigma G_+,R) \arrow{u}[swap]{(\Sigma p_{1+})^*} \arrow{l}[swap]{ \varepsilon^*}  & H^{*-1}(G,R)\arrow{u}[swap]{p_1^*}\arrow{l}[swap]{\cong}\\
H^*(X_{hG},R) \arrow{u}{(\gamma'_X)^*} &\widetilde{H}^*(BG,R)\arrow{l}[swap]{\beta^*}\arrow{ur}[swap]{\bar\gamma}\arrow{u}[swap]{(\gamma')^*}.&
\end{tikzcd}
\end{equation*}
Here the upper left triangle commutes, because the map $\Delta$ is
natural with respect to $X$.
We can now apply the diagram to an $x\in \widetilde{H}^*(BG,R)$ and get
\begin{equation*}
\delta(\widetilde{S}(x\otimes 1)) =\delta(\bar\gamma(x)\times 1)=
(\gamma'_X)^*(\beta^*(x)).
\end{equation*}
\end{proof}
\begin{ntt}
\label{weakmainhomomorphism}
Let $S_1\colon I_1^*(X,R) \to H^{*-1}(G\times X,R)/M^{*-1}_X$ be the restriction of $S$ to $I_1^*(X,R)$.
\end{ntt}
\begin{rmk}
\label{rmkSS1}
For our main results we need to calculate the homomorphism $S$. There does not seem to be an easy way to do this in general, but Proposition~\ref{mainproposition} gives us an explicit formula for $S_1$. Indeed, if an element $x$ belongs to $I_1^*(X,R)$, then $x=\sum \beta^*(a_i) b_i$ with $a_i \in \widetilde{H}^*(BG,R)$ and $b_i\in H^*(X_{hG},R)$, so using Proposition~\ref{mainproposition} we conclude that $S(x)=S_1(x)$ is the image of
\begin{equation}\label{formulas1}
\widetilde{S}\left(\sum a_i\otimes b_i\right)=\sum \bar\gamma(a_i) \times \alpha^*(b_i)\in H^{*-1}(G\times X,R).
\end{equation}
in $H^{*-1}(G\times X,R)/M^{*-1}_X$.

Moreover, in Lemma~\ref{s1sproj} we will see that if $X$ is a smooth compact complex algebraic variety, the ideals $I^*_1(X,\Q)$ and $I^*(X,\Q)$ coincide.
\end{rmk}

We will now discuss functorial properties of $S$. As we already saw in Corollary~\ref{functorS}, the homomorphism $S\colon I^*(X,R)\to H^{*-1}(G\times X,R)/M^{*-1}_X$ is natural with respect to $G$-equivariant morphisms of topological spaces. More generally, if $g:G_1\to G_2$ is a continuous homomorphism of topological groups, $X_i,i=1,2$ is a $G_i$-space, and $f:X_1\to X_2$ is a continous map that is equivariant with respect to $g$, then the following diagram commutes.
\begin{equation}\label{functorialityofs}
\begin{tikzcd}[column sep=large]
I^*(X_2,R) \arrow{r}{S} \arrow{d}{f^*}
&H^{*-1}(G_2\times X_2,R)/M^{*-1}_{X_2} \arrow{d}{(g\times f)^*}\\
I^*(X_1,R) \arrow{r}{S}
&H^{*-1}(G_1\times X_1,R)/M^{*-1}_{X_1}.
\end{tikzcd}
\end{equation}
Observe that the homomorphism $S_1$ is also natural, i.e.\ if one replaces in the diagram above $S$ with $S_1$ and $I$ with $I_1$, the resulting diagram will commute.

We also note the following property of the homomorphism $S$.
\begin{prop} 
\label{kernel}
If $x\in I^*(X,R)^2\subset H^*(X_{hG},R)$, then $S(x)=0$.
\end{prop}
\begin{proof}
Let $y,z$ be elements of $I^*(X,R)$. It is enough to show that $(\gamma'_X)^*(y\smile z) = 0$. But $(\gamma'_X)^*(y)$ and $(\gamma'_X)^*(z)$ are in the image of $\delta$ in exact sequence~\ref{lesnhq}. The product of any two elements in the image of $\delta$ is zero, and the proposition follows.
\end{proof}

\subsection{Smooth compact complex algebraic varieties}\label{secsmoothproj}
In this subsection $X$ will be a smooth {compact} complex algebraic variety and~$G$ will be a connected complex affine group that acts algebraically on $X$. To begin with, we show that under these assumptions the homomorphisms $S$ and $S_1$ coincide at least rationally.

\begin{lmm}\label{s1sproj}
The Leray-Serre spectral sequence of the fibration $X\to X_{hG}\to BG$ degenerates at $E_2$ over~$\Q$. As a consequence $I^*(X,\Q)= I_1^*(X,\Q)\cong H^*(X,\Q)\otimes\tilde H^*(BG,\Q)$, and so in rational cohomology we have $S=S_1$.
\end{lmm}
\begin{proof}
The group $G$ is a closed subgroup of some $GL_m(\Co)$, so up to homotopy equivalence $BG$ is a direct limit $\varinjlim B_i$, where each $B_i$ is a smooth complex algebraic variety. The weight filtration of $H^p(B_i,\Q)$ starts in weight $p$, i.e.\ $W_{p-1}(H^p(B_i,\Q))=0$, and $H^q(X,\Q)$ is pure of weight $q$. The degeneration of the spectral sequence now follows from Hodge theory, cf.\ \cite{Deligne1974}. 
\end{proof}
\begin{lmm} 
\label{actproj}
Recall that above we set $\vp\colon G\times X \to X$ to be the $G$-action morphism. The homomorphisms $\vp^*, p_2^*\colon H^*(X,\Q) \to H^*(G\times X,\Q)$ coincide.
\end{lmm}
\begin{proof}
The lemma follows from these observations: first, the group $H^p(X,\Q)$ is pure of weight $p$, whereas $W_q(H^q(G,\Q))=0$ unless $q=0$. So the image of both $\varphi^*$ and $p_2^*$ is inside $H^0(G,\Q)\otimes H^*(X,\Q)$. Secondly, both $\varphi$ and $p_2$ when composed with the map $x\mapsto (e,x)$ from $X$ to $G\times X$ give the identity of $X$.
\end{proof}
\begin{cor}
\label{indettorsion}
The integral indeterminacy submodule $M^*_X\subset H^*(G\times X,\Z)$ consists of torsion elements.
\end{cor}
\begin{proof}
It follows from the definition of $M^*_X$ that $M_X^*\otimes \Q$ is the image of the homomorphism
\begin{equation*}
\vp^* -p_2^* \colon H^*(X,\Q) \to H^*(G\times X,\Q),
\end{equation*}
which is zero by Lemma \ref{actproj}.
\end{proof}

\begin{cor}
\label{seceulersub}
Let $E$ be a $G$-equivariant $R$-oriented vector bundle over $X$ with $e(E)=0$. Then for every secondary equivariant Euler class $se_G(E)$ the following identity holds in $H^*(G\times X,\Q)$:
\begin{equation*}
se_G(E)=S(e_G(E))=S_1(e_G(E)). 
\end{equation*}
\end{cor}
\begin{proof}
This follows from Propositions~\ref{seceulerse} and~\ref{s1sproj}.
\end{proof}

\begin{cor}
\label{integerlatticeandS}
The following diagram commutes
$$
\begin{tikzcd}[column sep=large]
I^*(X,\Z) \arrow{r} \arrow{d}{S}
&I^*(X,\Q) \arrow{d}{S} 
&I_1^*(X,\Q) \arrow{d}{S_1}\arrow[swap]{l}{=} \\
\fr{*-1}(G\times X,\Z) \arrow[hookrightarrow]{r}
&H^{*-1}(G\times X,\Q)
&H^{*-1}(G\times X,\Q) \arrow[swap]{l}{=}.
\end{tikzcd}
$$\qed
\end{cor}

\begin{rmk}\label{imagebargammaprimitive}
{It follows from Remark~{\ref{gamma and transgression}} that} the image of $\bar\gamma$ in $H^*(G,\Q)$ is the graded vector space $P^*_{\Q}$ of the primitive elements of $H^*(G,\Q)$. This vector space is concentrated in odd degrees and freely generates $H^*(G,\Q)$ as a graded commutative $\Q$-algebra. It will be important for us in the sequel that the integral analogue of this is also true: the graded group $P^*=P^*_{\Q}\cap\fr{*}(G,\Z)$ of the primitive elements of $\fr{*}(G,\Z)$ freely generates $\fr{*}(G,\Z)$ as a graded commutative ring, see e.g.\ \cite{MT91}*{Chapter~VII, Theorem 1.22}. Note that $P^*$ is a direct summand of $\fr{*}(G,\Z)$ and that $\rk P^{2l-1}=\dim_{\Q}P^{2l-1}_{\Q}$ is the number of free polynomial generators of $H^*(BG,\Q)$ of degree $2l$.

We also note for the sequel that if $G$ is complex affine, then every element of $P^*_{\Q}$ has weight equal the degree plus 1; so every decomposable element of $H^*(G,\Q)$ has weight equal the degree plus 2 or greater.
\end{rmk}

\begin{rmk}\label{imageofs}
It follows from the previous remark, Lemma~\ref{s1sproj} and formula \ref{formulas1} that the image of $S$ in $H^*(G\times X,\Q)/M^*_X=H^*(G\times X,\Q)$ is $P^*_{\Q}\otimes H^*(X,\Q)$.
\end{rmk}

\begin{rmk}
By Remark~\ref{imagebargammaprimitive}, over the rationals the kernel of $\bar\gamma$ {in positive degrees} is the space of decomposable elements of $\tilde H^*(BG,\Q)$. One can prove that the kernel of $S:I(X,\Q)\to H^*(G\times X,\Q)$ is $(I(X,\Q))^2=(I_1(X,\Q))^2$. We will not use the latter statement in the sequel.
\end{rmk}

\section{Discriminant complements and orbit maps}\label{sectdiscrim}
In this section we state and prove our main results. 
As above, let $G$ be a topological group. Given a topological $G$-space $X$ and a $G$-equivariant oriented vector bundle $E$ over $X$ with Euler class 0, we use any secondary Thom class of $E$ to define the cohomology classes~$\Lk^{\Sigma}(y)\in H^*(C^0(X,E)_0,\Z)$ (see Definition~\ref{linkingclasshomomorphism}); here $y$ is an arbitray class $\in H_*(X,\Z)$ and $C^0(X,E)_0$ denotes the space of nowhere vanishing continuous sections of $E$. 
Furthermore, we explain how to calculate $O^*(\Lk^{\Sigma}(y))\in H^*(G,\Z)$ in terms of the equivariant Euler class of~$E$ (Corollaries~\ref{prepmaincorollary} and \ref{maincorollary}).

As an application, suppose that $X$ is a smooth complex analytic variety, $E$ is holomorphic and $G$ acts on $X$ and $E$ by biholomorphic transformations; now we no longer assume $e(E)=0$. In Section \ref{secsmoothsec} we define the space $\reg{X,E}\subset \Gamma(X,E)$ of regular sections of $E$ and use our results from Section~\ref{linkingclasssubsect} to define the classes~$\Lk^{\Sing}(y)\in H^*(\reg{X,E},\Z)$ (Definition~\ref{linkingclasssing}) and calculate their pullbacks under the orbit map $O^*\colon H^*(\reg{X,E})\to H^*(G)$ (Theorem~\ref{maintheorem} and Corollary~\ref{maincorollary2}).  In Propositions~\ref{geometricdescription1} and \ref{geometricdescription2} we explain the geometric motivation behind $\Lk^{\Sigma}(y)$ and $\Lk^{\Sing}(y)${, and in Propositions~{\ref{mhsmotivation}} and {\ref{mhsmotivation2}} we relate these classes to the weight filtration}.

\subsection{Linking class homomorphism: spaces of nowhere vanishing sections}\label{linkingclasssubsect}
Unless stated otherwise, in Section \ref{linkingclasssubsect} $X$ is a topological space and $E$ is a real vector bundle over $X$ of rank $r$.
\begin{ntt}
\label{sigma}
We set $C^0(X,E)$ to be the vector space of continuous sections of $E$. Let $V$ be a subset of $C^0(X,E)$ and set
\begin{equation*}
\Sigma_{V}(E)= \{s\in V\mid \mbox{ there is an $x \in X$ such that $s(x)=0$}\}.
\end{equation*}
Let $V_0$ be the complement of $\Sigma_V(E)$ in $V$. We define the {\it evaluation map}
\begin{align}
\label{evaluationmap}
ev_V\colon V_0 \times X &\to \Tot_0(E), \mbox{ } ev_V(s,x)= s(x). \nonumber
\end{align}
\end{ntt}
In the sequel we assume that $V$ has been equipped with a topology that makes the evaluation map $ev_V$ continuous.

\begin{exmp}
If $X$ is locally compact, then we can take $V=C^0(X,E)$ with the compact-open topology. (If $X$ is not assumed locally compact, then $ev_V$ is still continuous if one replaces all topological spaces involved with their compactly generated versions.)
\end{exmp}

\begin{exmp}
Let $X$ be a smooth manifold and let $E$ be a smooth vector bundle over $X$. Then one can take $V=$ the vector space of smooth sections $C^{\infty}(X,E)$ with the Whitney topology.
\end{exmp}
\begin{exmp}
Let $X$ be a compact complex analytic variety and let $E$ be a holomorphic vector bundle over $X$. Then one can take $V=$ the space of holomorphic sections $\Gamma(X,E)$. 
\end{exmp}

In the rest of Section \ref{linkingclasssubsect} we suppose that $E$ is $\Z$-oriented and the Euler class $e(E)\in H^r(X,\Z)$ is zero. Then there exists an oriented secondary Thom class $a_E\in H^*(\Tot_0,\Z)$~(Definition~\ref{secthom}).
\begin{dfn}
\label{linkingclasshomomorphism}
We define the {\it linking class homomorphism} $\Lk^\Sigma_V$ using the following rule:
\begin{align*}
\Lk^\Sigma_V\colon H_p(X,\Z) &\to H^{r-p-1}(V_0,\Z), \\
\Lk^\Sigma_V(y)&= -(ev^*_V(a_E))/y.
\end{align*}
Here $/y\colon H^{r-1}(V_0\times X,\Z) \to H^{r-p-1}(V_0,\Z)$ denotes the slant product homomorphism with the homology class $y\in H_p(X,\Z)$; see e.g.\ \cite{Sp81}*{Chapter 6, \S 1} for details. {Elements of the image of $\Lk^\Sigma_V$ will be called} {\it{linking classes}}.  
\end{dfn}
\begin{rmk}
{We use the minus sign here so as to eliminate signs from the formula in our main theorem~{\ref{maintheorem}}}.
%
\end{rmk}
\begin{prop}
\label{linkingclasshomomorphismwelldefined}
If $p\neq r-1$, then the homomorphism $\Lk^\Sigma_V$ is well defined, i.e.\ it does not depend on the choice of the oriented secondary Thom class.
\end{prop}
\begin{proof}
Indeed, by Proposition~\ref{indsecthom} any two oriented secondary Thom classes $a_E$ and $a'_E$ differ by $p_0^*(b)\in H^{r-1}(\Tot_0(E))$ for some $b\in H^{r-1}(X)$ where $p_0\colon \Tot_0(E) \to X$ is the projection map. Therefore, 
\begin{equation*}
ev_V^*(a_E) -ev^*_V(a'_E)\in \im p_2^* \subset H^{r-1}(V_0\times X).
\end{equation*}
Here $p_2\colon V_0 \times X\to X$ is the projection map. But if $y\in H_p(X)$ and $p\neq r-1$, then $(\im p_2^*)/y=0$.
\end{proof}

Here are some basic properties of the linking class homomorphism.

\begin{prop}
\label{lcnaturality2}

Let $f\colon X' \to X$ be a continuous map, and let $E$ be a vector bundle over $X$ with $e(E)=0$. We let $C^0f\colon C^0(X,E) \to C^0(X', f^*E)$ be the induced linear map between the vector spaces of global sections. Let $V\subset C^0(X,E)$ and $V'\subset C^0(X',f^*E)$ be subsets topologised so that $ev_V$ and $ev_{V'}$ are continuous, $C^0f(V) \subset V'$ and $C^0 f$ is continuous on $V$. Then the following diagram commutes:

\begin{equation*}
\begin{tikzcd}
H_p(X',\Z) \arrow{r}{\Lk^\Sigma_{V'}} \arrow{d}[swap]{f_*}
&H^{r-p-1}(V'_0,\Z)\arrow{d}{(C^0f)^*} \\
H_p(X,\Z) \arrow{r}{\Lk^\Sigma_V}
&H^{r-p-1}(V_0,\Z).
\end{tikzcd}
\end{equation*}

\end{prop}

\begin{proof}
We use the commutative diagram
\begin{equation*}
\begin{tikzcd}[row sep=tiny]
 & X'\times V_0'\arrow{r}{ev_{V'}} & \Tot_0 (E')\arrow{dd}\\
 X'\times V_0 \arrow{ru}{\id\times C^0f}\arrow[rd,"f\times\id"'] & & \\
 & X\times V_0\arrow{r}{ev_{V}} & \Tot_0 (E).
 \end{tikzcd}
 \end{equation*}
\end{proof}
\begin{cor}
\label{lcnaturality1}
Let $E$ be a vector bundle over $X$ with $e(E)=0$, and let $V\subset V'$ be subsets of $C^0(X,E)$ topologised so that $ev_V, ev_{V'}$ and the inclusion $i:V\to V'$ are continuous. Then the following diagram commutes:

\begin{equation*}
\begin{tikzcd}
H_p(X,\Z) \arrow{r}{\Lk^\Sigma_{V'}} \arrow{rd}[swap]{\Lk^\Sigma_{V}}
&H^{r-p-1}(V'_0,\Z)\arrow{d}{i^*} \\
&H^{r-p-1}(V_0,\Z).
\end{tikzcd}
\end{equation*}\qed
\end{cor}

We will now give a geometric interpretation for the classes $\Lk^\Sigma_V$. To do this we need a technical result, Proposition~\ref{twodefinitionsoflinkingclasses}.

If $M$ is an oriented (but not necessarily compact) topological manifold of dimension $m$ without boundary, then the {\it Poincar\'e-Lefschetz duality} for $M$ is the isomorphism $H^*(M,\Z)\to H_{m-*}^{BM}(M,\Z)$ given by the $\frown$-product with the fundamental class $[M]\in H_{m}^{BM}(M,\Z)$. If $Y\subset\R^m$ is a (closed) sub-polyhedron, then by composing the Poincar\'e-Lefschetz duality for $\R^m\setminus Y$ with the connecting homomorphism $H^{BM}_*(\R^m\setminus Y,\Z)\to H_{*-1}^{BM}(Y,\Z)$ in the Borel-Moore homology long exact sequence of the couple $(\R^m,Y)$ we get the {\it Alexander duality isomorphism} $\tilde H^*(\R^m\setminus Y,\Z)\cong H^{BM}_{m-*-1}(Y,\Z)$.

Let $E$ be an oriented vector bundle with $e(E)=0$ over an oriented smooth closed manifold $X$ of dimension $m$, and let $V$ be a finite-dimensional vector subspace of $C^\infty(X,E)$ that {\it generates} $E$ at each point of $X$, i.e.\ the evaluation map $V\times X\to\Tot(E)$ is surjective. In this case the following subset of $V\times X$ {(cf.\ Proposition~{\ref{degreediscr}})}
\begin{equation*}
\widetilde \Sigma_V(E)=\{(s,x)\in V\times X\,\,\, \vert\,\,\, s(x)=0\}
\end{equation*}
is an oriented vector bundle over $X$; let $n$ be the rank of this bundle.

Let $p_X:V\times X\to X$ and $p_V:V\times X\to V$ be the projections. Under our assumptions there is a fundamental class $[\widetilde{\Sigma}_V(E)]\in H^{BM}_{m+n} (\widetilde{\Sigma}_V(E),\Z)$, and the {\it Thom isomorphism} $u:H_*(X,\Z)\to H^{BM}_{*+n}(\widetilde{\Sigma}_V(E),\Z)$ is given by $u(y)=p_X^*(Dy)\frown [\widetilde{\Sigma}_V(E)]$ where $Dy$ is the Poincar\'e dual cohomology class of $y$.

There are two maps $H_*(X,\Z)\to H_{*+n}^{BM}(\Sigma_V(E),\Z)$. One is given by $$y\mapsto (p_V)_*(u(y))= (p_V)_*(p_X^*(Dy)\frown [\widetilde{\Sigma}_V(E)]),$$ and the other map is
$$y\mapsto D^A(ev(a_E)/y)$$ where $D^A:\tilde H^*(V_0,\Z)\to \bm{\dim V-*-1}(\Sigma_V(E),\Z)$ is the Alexander duality isomorphism.

\begin{prop}\label{twodefinitionsoflinkingclasses}
In each degree these two maps coincide up to sign.
\end{prop}

\begin{proof}
We will use $\partial$ to denote the connecting homomorphisms in the long exact sequences of the Borel-Moore homology groups. The evaluation map
$$ev:V_0\times X\to \Tot_0(E)$$
extends to $(V\times X)\setminus\widetilde{\Sigma}_V(E)$. Set $v_E\in H_*^{BM}((V\times X)\setminus\widetilde{\Sigma}_V(E),\Z)$ to be the Poincar\'e-Lefschetz dual of $ev^* (a_E)$. We have $\partial(v_E)=[\widetilde{\Sigma}_V(E)]$ in the Borel-Moore long exact sequence of the couple $(V\times X,\widetilde{\Sigma}_V(E))$. This is a long exact sequence of $H^*(X,\Z)$-modules, so for every $z\in H^*(X,\Z)$ we have $\partial (p_X^*(z)\frown v_E)=p_X^*(z)\frown [\widetilde{\Sigma}_V(E)]$. The map 
$p_V:(V\times X,\widetilde{\Sigma}_V(E))\to (V,\Sigma_V(E))$
is proper, so we have $$\partial((p_V)_*(p_X^*(z)\frown v_E))=(p_V)_*(p_X^*(z)\frown [\widetilde{\Sigma}_V(E)]),$$
which implies that the Alexander dual $\in H^*(V_0,\Z)$ of $(p_V)_*(p_X^*(z)\frown [\widetilde{\Sigma}_V(E)])$ is the Poincar\'e-Lefschetz dual of $(p_V)_*(p_X^*(z)\frown v_E)$. Let us give another expression for the latter class.

Set $v'_E$ to be the restriction of $v_E$ to $V_0\times X$. We have then $v'_E=ev^*(a_E)\frown ([V_0]\times [X])$, so
$$p^*_X(z)\frown v_E'=(p^*_X(z) \smile ev^*(a_E))\frown ([V_0]\times [X]) =\pm ev^*(a_E)\frown ([V_0]\times y)$$ where $y=z\frown [X]$ is the Poincar\'e dual of $z$. Using e.g.\ \cite{Switzer75}*{Proposition 13.61, part (vi)} we get
$$(p_V)_*(p_X^*(z)\frown v_E)=(p_V)_*(p_X^*(z)\frown v'_E)=\pm (ev^*(a_E)/y)\frown [V_0]),$$ which is $\pm$ the Poincar\'e-Lefschetz dual of $ev^*(a_E)/y$.
\end{proof}

\bigskip

For the next two propositions we suppose that $X$ is a smooth {compact} complex algebraic variety of (complex) dimension $d$ and $E$ is a globally generated holomorphic vector bundle of (complex) rank $r>d$. Then $e(E)=0$, and we set $V=\Gamma(X,E)\subset C^0(X,E)$ and $N=\dim_{\Co} V$. By our assumptions $\widetilde{\Sigma}_V(E)\to X$ is a vector bundle of rank $N-r$. We will now explain how the linking class homomorphism is related to {actual} linking numbers. 

Suppose $y\in H_{2p}(X,\Z)$ is represented by an irreducible subvariety $Y\subset X$ of dimension $p$. Let {$\Sigma_V^Y(E)\subset \Sigma_V(E)$} denote the {irreducible subvariety formed by} all sections that vanish at some point of $Y$. There exists a fundamental class $[\Sigma_V^Y(E)]$ in the Borel-Moore homology of $\Sigma_V(E)$, and by the Alexander duality this class gives us a cohomology class $\Lk^{\Sigma}_Y \in H^*(\nonzero{X,E})$ that can be viewed as the ``linking number class of $\Sigma_V^Y(E)$'', as the value of $\Lk^{\Sigma}_Y$ on a cycle $c\subset V_0$ is the linking number of $c$ and $\Sigma_V^Y(E)$ inside~$V$.

\begin{prop}
\label{geometricdescription1}
Set $\widetilde\Sigma_V^Y(E)=p_X^{-1}(Y)\subset \widetilde\Sigma_V(E)$. Suppose that $p_V\colon \widetilde \Sigma^Y_V(E) \to \Sigma_V^Y(E)$ is generically finite of degree $k$. Then $\Lk^{\Sigma}_V(y) = -k\Lk^{\Sigma}_Y$.
\end{prop}

\begin{proof}
Note that $u(y)$ is represented by $\widetilde \Sigma^Y_V(E)$. Hence $(p_V)_{*}(u(y))=(p_V)_{*}[\widetilde \Sigma^Y_V(E)]$ is the $k$ times the fundamental class of the set-theoretical image $p_V(\widetilde \Sigma^Y_V(E))=\Sigma^Y_V(E)\subset \Sigma_V(E)$. 
\end{proof}

\begin{rmk}
Let $Z(s)\subset X$ be the zero locus of $s\in C^0(X,E)$. The number $k$ of Proposition~\ref{geometricdescription1} can be interpreted as $|Z(s)\cap Y|$ for a general section $s \in \Sigma_V(E)$. This number is finite if and only if $\langle c_p(E), [Y]\rangle \neq 0$, see Proposition~\ref{degreediscr}.
\end{rmk}

\begin{rmk}
If $p_V\colon \widetilde \Sigma^Y_V(E)\to \Sigma_V^Y(E)$ is not generically finite then $(p_V)_{*}(u(y))=0$ and $\Lk^{\Sigma}_V(y)=0$.
\end{rmk}

Proposition~\ref{twodefinitionsoflinkingclasses} also allows one to interpret the linking class homomorphism in terms of Hodge theory. The topological space $\nonzero{X,E}$ is a smooth complex algebraic variety, so its cohomology groups $H^*(\nonzero{X,E},\Q)$ carry a mixed Hodge structure. As $\nonzero{X,E}$ is an open subset of an affine space, we have $W_p H^p(\nonzero{X,E},\Q)=0$ for all $p>0$, and it is natural to ask what the next term $W_{p+1} H^p(\nonzero{X,E},\Q)$ of the weight filtration looks like.

\begin{prop}
\label{mhsmotivation}
If $2r-p-1>0$, then the image $\im(\Lk^{\Sigma}_V)$ of the linking class homomorphism 
$$\Lk^{\Sigma}_V\colon H_p(X,\Q)\to H^{2r-p-1}(\nonzero{X,E},\Q)$$
coincides with $W_{2r-p}H^{2r-p-1}(\nonzero{X,E},\Q)$.
\end{prop}

\begin{proof}
Set $p'=p+2N-2r$. Since the Alexander duality isomorphism is compatible with mixed Hodge structures (up to Tate twists), it is enough to prove that the image of $(p_V)_{*}\circ u$ coincides with $W_{-p'}\bm{p'}(\Sigma_V(E),\Q)$. Indeed, $u\colon H_p(X) \xrightarrow{\cong} \bm{p'}(\widetilde\Sigma_V(E))(r-N)$ is an isomorphism of Hodge structures of pure weight $-p$, and $(p_V)_{*}\colon \bm{p'}(\widetilde\Sigma_V(E)) \to \bm{p'}(\Sigma_V(E))$ is a surjection on $W_{-p'}$, which follows e.g.\ by~\cite{PS08}*{Theorem~5.41}.
\end{proof}

\bigskip

In the rest of Section \ref{linkingclasssubsect}, $G$ will be a topological group acting continuously on $X$, and the vector bundle $E$ will be assumed $G$-equivariant; recall that we also assume $e(E)=0$. We wish to calculate the pullbacks of the cohomology classes $\Lk_V^\Sigma(y)$ under orbit maps. 
As in Section \ref{1stsec}, we let $p_2\colon G\times X\to X$ and $P_2\colon G\times\Tot_0(E)\to\Tot_0(E)$ be the projections, and we denote the $G$-actions on $X$ and $\Tot(E)$ by $\vp\colon G\times X \to X$ and $\Phi \colon G\times \Tot(E) \to \Tot(E)$ respectively. 
Let $V$ be a $G$-stable vector subspace of $C^0(X,E)$.

\begin{exmp}\label{holo}
If $G,X,E$ and the actions of $G$ on $X$ and $E$ are complex analytic, and $X$ is compact, then one can take $V=$ the vector space $\Gamma(X,E)$ of holomorphic sections.
\end{exmp}

Suppose that $V_0$ is non-empty and take a section $s \in V_0$. Recall that the {\it orbit map} is the map $O\colon G \to V_0$ given by $O(g) = g(s)$ (see Introduction).

\begin{rmk}\label{orbitmapshomotopic} If $V_0$ is path connected (e.g.\ when $G,X,E$ are as in Example \ref{holo}), then all orbit maps are homotopic. Indeed, let $\gamma:I\to V_0$ be a path in $V_0$. Then the orbit maps constructed using $s_0=\gamma(0)$ and $s_1=\gamma(1)$ are homotopic via the homotopy that takes $(g,t)\in G\times I$ to $g\gamma(t)$.
\end{rmk}

In the rest of the subsection $s_0$ will be a fixed element of $V_0$ and $O$ will be the orbit map constructed using~$s_0$. The following commutative diagram will be crucial for us. 
\begin{equation}
\label{cdevact}
\begin{tikzcd}[column sep=large]
G\times \{s_0\}\times X \arrow{r}{\iota}
& G\times V_0\times X \arrow{r}{\id \times ev_V} \arrow{d}[swap]{A}
&G\times \Tot_0(E) \arrow{d}{\Phi}\\
&V_0\times X \arrow{r}{ev_V}
&\Tot_0(E).
\end{tikzcd}
\end{equation}
Here the map $A\colon G\times V_0\times X \to V_0\times X$ is given by $(g,s,x )\mapsto (gs,gx)$, and $\iota$ denotes the inclusion. 

\begin{ntt}\label{decq}
Let ${{\dec}}^i_{\Q} (G)$ be the subgroup of $H^i(G,\Z)$ that consists of all elements that become decomposable in $H^*(G,\Q)$. We will denote this subgroup by $\dec^i_{\Q}$ when this is unlikely to lead to confusion.
\end{ntt}

\begin{prop}
\label{linkingsecthom}
If $y\in H_p(X,\Z)$, then for every oriented secondary Thom class $a_E \in H^{r-1}(\Tot_0 E,\Z)$ the following identity holds in $H^{r-1-p}(G,\Z)$ modulo ${{\dec}}^{r-1-p}_{\Q} (G)$:
\begin{equation}\label{linkingsecthomeq}
O^*(\Lk^{\Sigma}_V(y)) = -(ev_V \circ A \circ \iota)^*(a_E)/y + (s_0\circ\varphi)^*(a_E)/y.
\end{equation}
\end{prop}
\begin{proof}
Notice that the following diagram commutes
\begin{equation}\label{maindiagmainlemma}
\begin{tikzcd}
G\times \{s_0\}\times X \arrow{rr}{\iota} \arrow{d}[swap]{\id \times O\times \id}
&&G\times V_0 \times X \arrow{d}{A}\\
G \times V_0\times X \arrow{r}
& V_0 \times G\times X \arrow{r}{\id\times \vp}
&V_0 \times X
\end{tikzcd}
\end{equation}
where the left arrow in the bottom row transposes the first two factors. Take a $b\in H^{r-1}(V_0\times X,\Z)$ and a $y\in H_p(X,\Z)$, and set $b'=(\id\times\varphi)^*(b)$. From the diagram we get
\begin{equation}\label{mainlemmaaux1}
((A \circ \iota)^*b)/ y=({O'})^*(b'/y)
\end{equation}
where $O':G\to V_0\times G$ takes $g\in G$ to $(g(s_0),g)$.

We now compare this with $O^*(b/y)$. Let $i_1:V_0\to V_0\times G$ be the inclusion given by $v\mapsto (v,e)$. Using the map $$i_1\times\id_X:V_0\times X\to V_0\times G\times X$$
we see that $i_1^*(b'/y)=b/y$, so
\begin{equation}\label{mainlemmaaux2}
O^*(b/y)=(i_1\circ O)^*(b'/y).
\end{equation}

Let $i_2\colon G\to V_0\times G$ be the inclusion given by $g\mapsto (s_0,g)$. The maps $O'$, $i_1\circ O$, and $i_2$ can be decomposed as follows:
\begin{align*}
O'\colon G\xrightarrow{\Delta}G\times G \xrightarrow{O\times\id_G}V_0\times G,\\
i_1\circ O\colon G\xrightarrow{j_1}G\times G \xrightarrow{O\times\id_G}V_0\times G,\\
i_2\colon G\xrightarrow{j_2}G\times G \xrightarrow{O\times\id_G}V_0\times G.
\end{align*}
Here the map $\Delta$ is the diagonal inclusion, $g\mapsto (g,g)$, the map $j_1$ is given by $g\mapsto (g,e)$, and $j_2$ is given by $g\mapsto (e,g)$. It follows from the K\"unneth formula that the image of
\begin{equation}\label{mainlemmaaux3}
({O'})^*-(i_1\circ O)^*-i^*_2:H^*(V_0\times G,\Q)\to H^*(G,\Q)
\end{equation}
consists of decomposable elements. Finally, to finish the proof we set $b=-ev_V^*(a_E)$ and apply \ref{mainlemmaaux3} to $b'/y=-\big ((\id\times\varphi)^*(ev_V^*(a_E))\big)\bigm/ y$. We conclude using \ref{mainlemmaaux1} and \ref{mainlemmaaux2} that
\begin{equation*}
-(ev_V \circ A \circ \iota)^*(a_E)/y - O^*(\Lk^{\Sigma}_V(y))  + i_2^*\Big (\big ((\id\times\varphi)^*(ev_V^*(a_E))\big)\bigm/y\Big ) \in {{\dec}}^{r-1-p}_{\Q} (G).
\end{equation*}
It remains to calculate the third term in this sum. Using the commutative diagram
\begin{equation*}
\begin{tikzcd}
V_0\times G\times X\arrow{r}{\id\times\varphi} & V_0\times X\arrow{r}{ev_V} & \Tot_0(E)\\
G\times X\arrow{r}{\varphi}\arrow{u}{i_2\times\id_X} & X \arrow{r}{s_0}& \Tot_0(E)\arrow{u}[swap]{\id}
\end{tikzcd}
\end{equation*}
we see that $i_2^*\Big (\big ((\id\times\varphi)^*(ev_V^*(a_E))\big)\bigm/y\Big ) = \varphi^*(s_0^*(a_E))/y$.
%
\end{proof}
\begin{rmk}\label{linkingsecthomeqsmoothproj} Suppose $X$ is smooth compact complex algebraic, $G$ is complex affine algebraic and the action of $G$ on $X$ is algebraic. Then the map $\varphi^*:H^*(X,\Q)\to H^*(G\times X,\Q)$ is the same as $p_2^*$ (Lemma~\ref{actproj}). This implies that the map $(\id\times\varphi)^*:H^*(V_0\times X,\Q)\to  H^*(V_0\times G\times X,\Q)$ is the same as $(\id\times p_2)^*$. So using diagram \ref{maindiagmainlemma} we see that under the above assumptions in rational cohomology we have
\begin{equation*}
O^*(\Lk^{\Sigma}_V(y)) = -(ev_V \circ A \circ \iota)^*(a_E)/y.
\end{equation*}
\end{rmk}
\begin{dfn}
\label{bigindet}
The {\it big indeterminacy submodule} $bM^*_X \subset H^*(G\times X,\Z)$ is the sum of the indeterminacy submodule $M^*_X$ and $p_2^*(H^*(X,\Z))$. As usual, we will omit the subscript if it is clear which space $X$ is meant.
\end{dfn}
\begin{lmm}
\label{partialactionsecthom}
Recall that we assume that $e(E)=0$.
For every secondary Thom class $a_E$ the following identity holds modulo the big indeterminacy submodule $bM^{*}_X\subset H^{*}(G\times X,\Z)$:
\begin{equation*}
\partial \Phi^*(a_E)=(\Phi\circ (\id\times ev_V)\circ \iota)^*(a_E)
\end{equation*}
\end{lmm}
\begin{proof}
As in Section~\ref{1stsec}, let $p_0\colon \Tot_0(E)\to X$ be the projection map and recall that the homomorphism
\begin{equation*}
(\id\times p_0)^*\colon H^{*}(G\times X,\Z)\to H^{*}(G\times \Tot_0(E),\Z)
\end{equation*}
is injective (Proposition~\ref{injectivity}). So to prove the lemma it would suffice to check that 
\begin{equation}\label{partialactionsecthomformula}
(\id\times p_0)^*(\partial \Phi^*(a_E))= (\id\times p_0)^*((\Phi\circ (\id\times ev_V)\circ \iota)^*(a_E)) \mod (\id\times p_0)^*(bM^*_X).
\end{equation}
The left hand side of this is equal $\Phi^*(a_E)-P_2^*(a_E)$. Set $b= \Phi^*(a_E)-P_2^*(a_E)$, see Notation~\ref{partialaction}. Formula~\ref{partialactionsecthomformula} then reads
\begin{equation*}
b= (\id\times p_0)^*\big (((\id\times ev_V)\circ \iota)^*(b+P_2^*(a_E))\big )\mod (\id\times p_0)^*(bM^*_X). 
\end{equation*}
Let $q\colon G\times V_0\times X \to G\times X$ be the projection map. Notice that the following diagram commutes
\begin{equation*}
\begin{tikzcd}[column sep=large]
G\times V_0\times X\arrow{r}{\id\times ev_V} \arrow{d}{q}
&G\times \Tot_0(E) \arrow{d}{\id\times p_0} \\
G\times X \arrow{r}{\id}
&G\times X.
\end{tikzcd}
\end{equation*}
Moreover, the map $\iota\colon G\times X\cong G\times \{s_0\} \times X\to G\times V_0\times X$ is a section of the projection map $q$. Since $b= (\id\times p_0)^*(c)$ for some $c\in H^*(G\times X,\Z)$ by Lemma~\ref{actionsecthom}, we see that
\begin{equation*}
(\id\times p_0)^*\big (((\id\times ev_V)\circ \iota)^*(b)\big )= (\id\times p_0)^*(\iota^*(q^*(c))= (\id\times p_0)^*(c)=b. 
\end{equation*}
To complete the proof of the lemma it suffices to show that $((\id\times ev_V)\circ \iota)^*(P_2^*(a_E))\in  bM^*_X$, which follows from
$((\id\times ev_V)\circ \iota)^*(P_2^*(a_E)) =p_2^*(s_0^*a_E) \in bM^*_X$.
\end{proof}
\begin{ntt}
Let $y\in H_p(X,\Z)$ be a homology class. We denote by $bM^q_{X,y}$ the image of the big indeterminacy submodule $bM^{q+p}_X$ under the slant product homomorphism $/y\colon H^*(G\times X,\Z) \to H^{*-p}(G,\Z)$.
\end{ntt}
Notice that if $x\in H^q(X,\Z)$ and $y\in H_p(X,\Z)$ such that $p\neq q$, then
\begin{equation}\label{slantvanishmodbm}
\varphi^*(x)/y=p_2^*(x)/y=0 \mod bM^{q-p}_{X,y}.
\end{equation}

\begin{cor}
\label{prepmaincorollary}
Recall that we assume that $e(E)=0$, and $r$ denotes the rank of $E$. 
The following identity holds for every $y\in H_p(X,\Z)$ provided $p\neq r-1$: 
\begin{equation*}
O^* (\Lk^{\Sigma}_V(y)) = S(e_G(E))/y \mod \left(bM^{r-p-1}_{X,y}+\dec\nolimits^{r-p-1}_{\Q}\right).
\end{equation*}
\end{cor} 

Recall that the target of $S$ is $H^{*}(G\times X,\Z)/M_X^{*}$, see Definition~\ref{mainhomomorphism}. Since $(M^*_X)/y\subset bM_{X,y}^*$, the~above identity makes sense.

\begin{proof} 
\begin{align*} S(e_G(E))/ y\; =& \;se_G(E)/y \mod bM^{r-p-1}_{X,y} &\text{(Proposition~\ref{seceulerse})} \nonumber
\; \\=& \;-\partial \Phi^*(a_E)/y \mod bM^{r-p-1}_{X,y} &\text{(Proposition~\ref{secthomseceuler})} \nonumber
\; \\=& \;-(\Phi\circ (\id\times ev_V)\circ \iota)^*(a_E)/y \mod bM^{r-p-1}_{X,y} &\text{(Lemma~\ref{partialactionsecthom})} 
\; \\=& \;-(ev_V \circ A \circ \iota)^*(a_E)/y \mod bM^{r-p-1}_{X,y}&\text{(Commutative diagram~\ref{cdevact})}\nonumber
\; \\=& \; O^*(\Lk^{\Sigma}_V(y)) \mod\left( bM^{r-p-1}_{X,y}+ \dec\nolimits^{r-p-1}_{\Q}\right)&\text{(Proposition~\ref{linkingsecthom} and formula~\ref{slantvanishmodbm}).} \nonumber
\end{align*}
\end{proof}

\begin{cor}
\label{maincorollary}
We use the notation of Corollary~\ref{prepmaincorollary}.
Suppose moreover that $X$ is a CW-complex of dimension $<r-1$ (so $e(E)=0$ is automatic). Then the following identity holds in $H^*(G,\Z)$ for every $y\in H_p(X,\Z)$: 
\begin{equation*}
O^*(\Lk^{\Sigma}_V(y)) = S(e_G(E))/y\mod {\dec}\nolimits^{r-p-1}_{\Q}. 
\end{equation*}
\end{cor}
\begin{proof}
By Corollary~\ref{prepmaincorollary} it suffices to prove that $bM^{r-1}_{X}=0$, as $bM^{r-p-1}_{X,y}=bM^{r-1}_{X}/y$. We now use the fact that $bM^*_X=\im(\varphi^*-p_2^*)+\im p_2^*=0$ in degree $r-1$.
\end{proof}

%
Corollary \ref{maincorollary} allows one to construct a subring of $\fr{*}(G,\Z)$ that is included in the image of the cohomology map induced by any orbit map $O:G\to V_0$. In the sequel we will often use these subrings, so we introduce the following notation.

\begin{ntt}\label{lambdae}
Recall that above we set $G$ to be a topological group and $E$ a $\Z$-oriented $G$-equivariant vector bundle over a topological $G$-space $X$ with $e(E)=0$. We let $\Lambda^*_{E,G}$, or simply $\Lambda^*_E$ if it is clear which group $G$ is meant, denote the subring of $\fr{*}(G,\Z)$ generated by $S(e_G(E))/y,y\in H_*(X,\Z)$.
\end{ntt}

\subsection{Linking class homomorphism: spaces of regular sections}\label{secsmoothsec}

{To begin with, suppose} $X$ is a $C^\infty$-smooth manifold and $E$ is a smooth real vector bundle over $X$ of rank $r$. Note that $X$ embeds in $\Tot(E)$ as the zero section. We will now explain how to apply the results of the previous subsection to spaces of regular sections.
\begin{dfn}
\label{transversality}
Let $s\in C^{\infty}(X,E)$ be a smooth section of $E$. An $x\in X$ is called {\it a singular point} of the zero locusof $s$ if the submanifolds $X$ and $s(X)$ embedded into $\Tot(E)$ intersect non-transversally at $x$.
\end{dfn} 

\begin{ntt}
\label{sing}
Let us equip the vector space $C^{\infty}(X,E)$ with the standard Whitney $C^\infty$-topology. Let $V$ be a subset of $C^\infty(X,E)$. The subset 
\begin{equation*}
\Sing_V(E)= \{s\in V\mid \mbox{ the zero locus of $s$ has a singular point}\}.
\end{equation*}
of $V$ will be called the {\it discriminant}. Let $V_{\rreg}$ be the complement of $\Sing_V(E)$ in $V$. Elements of $\Sing_V(E)$, respectively $V_{\rreg}$ will be called {\it singular}, respectively {\it regular}. \par
\end{ntt}

\begin{exmp}
\label{bigrank}
Suppose that $\rk(E)>\dim(X)$. Then $V_{\rreg}= V_0$ for any $V\subset C^\infty(X,E)$. 
\end{exmp}

Given a vector bundle $E$ as above there is a vector bundle $E'$ such that there is a continuous inclusion $j\colon C^{\infty}(X,E) \to C^{0}(X,E')$ and $j(V_{\rreg})\subset j(V)_0$ for any subset $V\subset C^{\infty}(X,E)$. The holomorphic and $C^\infty$-versions are slightly different.

The $C^\infty$ case is discussed e.g.\ in \cite{S89}*{Chapter 4}: we take $E'=J_{\mathrm{top}}(E)$, the smooth jet bundle over $X$, which is part of an exact sequence
\begin{equation}
\label{sesjetbundletop}
0\to T^* X \otimes_{\mathbb{R}} E \to J_{\mathrm{top}}(E) \to E \to 0.
\end{equation}

In this paper however we are mainly interested in the case when $X$ is smooth complex algebraic and $E$ is an algebraic vector bundle. Then we take $E'$ to be the algebraic (first) jet bundle $J(E)$, see e.g.\ \cite{EGA4}*{Chapter 16.7}, and there is a natural exact sequence
\begin{equation}
\label{sesjetbundle}
0\to \Omega_X \otimes_{\mathcal{O}_X} E \to J(E) \to E \to 0.
\end{equation}

This construction extends in a straightforward way to the case when $X$ is smooth complex analytic and $E$ is a holomorphic vector bundle over $X$, which is what we will assume until the end of Section \ref{secsmoothsec}. If $V=\Gamma(X,E)$, the space of homolorphic sections of $E$, we will write $\Sing(E)$ instead of $\Sing_V(E)$.

Note that the rank $\rk_{\Co}(J(E))$ is $\rk_{\Co}(E)(\dim_{\Co}(X)+1)$ and that there is a natural linear injective map
\begin{equation*}
\label{jetsection}
j \colon \Gamma(X,E) \to \Gamma(X,J(E)) \hookrightarrow C^0(X,J(E)).
\end{equation*}
By local arguments the preimage $j^{-1}(j(V)_0)$ is equal to the set $V_{\rreg}$, so we get a natural continuous map $j\colon V_{\rreg} \to j(V)_0$.

Recall that the Cayley trick allows one to identify $\Gamma(X,E)$ with the space of holomorphic sections of a certain line bundle in such a way that regular sections correspond to regular sections. In more detail, let $\PP(E^*)$ be the projectivisation of the dual vector bundle $E^*$ and let $\pi\colon \PP(E^*)\to X$ be the projection. The direct image $\pi_*(\Oh_{\PP(E^*)}(1))$ of the relative twisting line bundle $\Oh_{\PP(E^*)}(1)$ is canonically isomorphic to the vector bundle~$E$ itself. So there is an isomorphism
\begin{equation*}
\Gamma\pi_*\colon \Gamma(\PP(E^*),\Oh_{\PP(E^*)}(1)) \xrightarrow{\cong} \Gamma(X,E).
\end{equation*}
There is a natural (continuous) linear map
\begin{equation*}
j_E\colon \Gamma(X,E) \xrightarrow{\Gamma\pi_{*}^{-1}}\Gamma(\PP(E^*),\Oh_{\PP(E^*)}(1)) \xrightarrow{j} C^0\big(\PP(E^*),J(\Oh_{\PP(E^*)}(1))\big).
\end{equation*}
{A section $j_E(s),s\in\Gamma(X,E)$ of $J(\Oh_{\PP(E^*)}(1))$ will be called} {\it holonomic.}
Let $V$ be any subset of~$\Gamma(X,E)$. A local calculation shows that $j_E(V_{\rreg})\subset j_E(V)_0$. Recall that for every subset $W\subset C^0\big(\PP(E^*),J(\Oh_{\PP(E^*)}(1))\big)$ we have a homomorphism $\Lk^{\Sigma}_W\colon H_*(\PP(E^*))\to H^*(W_0)$ (see~Definition~\ref{linkingclasshomomorphism}).
\begin{dfn}
\label{linkingclasssing}
Set $d=\dim_{\Co} X$ and $r=\rk_{\Co} E$. We define the {\it linking class homomorphism}
\begin{equation*}
\Lk^{\Sing}_V\colon H_p(\PP(E^*),\Z) \to H^{2(r+d)-p-1}(V_{\rreg},\Z) 
\end{equation*}
by the formula $\Lk^{\Sing}_V(y)=j_E^*\big(\Lk^{\Sigma}_{j_E(V)}(y)\big)$ for $y\in H_p(\PP(E^*),\Z)$. {Elements of the image of $\Lk^{\Sing}_V$ will be referred to as} {\it{linking classes}}. If $V=\Gamma(X,E)$, we will often write $\Lk$ instead of $\Lk^{\Sing}_V$.
\end{dfn}
\begin{rmk}
\label{rankonecoincidence}
If $E$ is a line bundle, then everything simplifies: we have $\PP(E^*)=X$, $\Oh_{\PP(E^*)}(1)=E$, and $j_E=j$. So for $y\in H_*(X,\Z)$ the cohomology class $\Lk^{\Sing}_V(y)$ is given by the formula $\Lk^{\Sing}_V(y)=j^*\big(\Lk^{\Sigma}_{j(V)}(y)\big)$, cf.\ the introduction.
\end{rmk}

\bigskip

The classes $\Lk^{\Sing}_V(y)$ have a geometric interpretation similar to the one given in Propositions~{\ref{geometricdescription1}} and {\ref{mhsmotivation}} for $\Lk^\Sigma_V(y)$. 
%
%
Set $V=\Gamma(X,E)$ and $L=\Oh_{\PP(E^*)}(1)$. We take $X$ to be smooth compact algebraic and we assume that
$J(\Oh_{\PP^*(E)}(1))$ is globally generated by holonomic sections, i.e.\ if the following {map} is surjective: 
\begin{equation*}
\Gamma(\PP(E^*),L)\times\PP(E^*)  \xrightarrow{j\times \id}  \Gamma(\PP(E^*),J(L))\times\PP(E^*)  \xrightarrow{ev} \Tot(J(L)).
\end{equation*}

Then
\begin{equation*}
\widetilde {\Sing}(E)=\{(s,x)\in \Gamma(X,E)\times\PP(E^*) \,\,\, \vert\,\,\, j_E(s)(x)=0\}
\end{equation*}
is the total space of a vector bundle over $\PP(E^*)$.
{Suppose} that $y\in H_{2p}(X,\Z)$ is {represented by an} irreducible algebraic subvariety $Y\subset X$. Let $\Sing^Y(E)$ denote the subset of $\Sing(E)$ that consists of sections that have a singular point in~$Y$. Then $\Sing^Y(E)$ is an irreducible subvariety of $\Sing(E)$, so there exist a fundamental class $[\Sing^Y(E)] \in \bm{*}(\Sing(E))$ and the Alexander dual cohomology class $\Lk^{\Sing}_Y \in H^*(\reg{X,E})$. These classes generalise those defined by C.~Peters and J.~Steenbrink \cite{PS03} for $E=\Oh(d),X=\PP^n(\Co)$. In the following proposition we express them in terms of the linking class homomorphism.
%

\begin{prop}
\label{geometricdescription2}
Set ${\widetilde\Sing}{}^Y(E)=p_X^{-1}(Y)\subset \widetilde\Sing(E)$. Suppose that $p_V\colon {\widetilde\Sing}{}^Y(E) \to \Sing^Y(E)$ is generically finite of degree $k$. Then $\Lk^{\Sing}_V(\pi^{!}y) = -k\Lk^{\Sing}_Y$, where $\pi^!y=[\pi^{-1}(Y)]$.\qed
\end{prop}

\begin{rmk}
Let $\Sing(s)\subset X$ be the singular locus of $s\in \Gamma(X,E)$. The number $k$ of Proposition~\ref{geometricdescription2} can be interpreted as $|\Sing(s)\cap Y|$ for a general section $s \in \Sing(E)$. This number is finite if and only if $\langle c_{p}(J(L)), [\pi^{-1}Y]\rangle \neq 0$, see Proposition~\ref{degreediscr}.
\end{rmk}

\begin{prop}
\label{mhsmotivation2}
If $p>1$, then the image $\im(\Lk^{\Sing}_V)$ of the linking class homomorphism 
$$\Lk^{\Sing}_V\colon H_{2r+2d-p}(\PP(E^*),\Q)\to H^{p-1}(\reg{X,E},\Q)$$
coincides with $W_{p}H^{p-1}(\reg{X,E},\Q)$. \qed
\end{prop}

\bigskip

Now we study the behaviour of the cohomology classes $\Lk^{\Sing}_V(y)$ under a group action. Let $G$ be a topological group that acts on $X$ by biholomorphic transformations, and suppose $E$ is $G$-equivariant. Then~$G$ also acts on $E^*$, $\PP(E^*)$, $\Oh_{\PP(E^*)}(1)$, and $J(\Oh_{\PP(E^*)}(1))$, the isomorphisms $\pi_*(\Oh_{\PP(E^*)}(1)) \cong E$ and $\Gamma \pi^*$ are $G$-equivariant, and so are exact sequence~\ref{sesjetbundle} and the inclusion $j_E\colon \Gamma(X,E) \to C^{0}(\PP(E^*),J(\Oh_{\PP(E^*)}(1)))$.

Let~$V$ be a $G$-stable subset of $\Gamma(X,E)$. Then $V_{\rreg}$ is also $G$-stable. Suppose that $V_{\rreg}$ is non-empty. Fix an $s_0\in V_{\rreg}$ and consider the orbit map $O\colon G \to V_{\rreg}$, $O(g)=gs_0$. We want to calculate the following composition:
\begin{equation*}
O^*\circ \Lk_V^{\Sing}\colon H_p(\PP(E^*),\Z) \to H^{2(d+r)-1-p}(V_{\rreg},\Z) \to H^{2(d+r)-1-p}(G,\Z).
\end{equation*}

\begin{thm}
\label{maintheorem}
Let $G$ be a topological group that acts on a smooth complex analytic variety $X$ by biholomorphic transformations. Suppose that $E$ is a $G$-equivariant holomorphic vector bundle over~$X$ and let $V$ be a $G$-stable subset of $\Gamma(X,E)$. Then the following identity holds in $H^*(G,\Z)$ modulo $\dec^*_{\Q}$ for every $y\in H_*(\PP(E^*),\Z)$: 
\begin{equation*}
O^*\circ \Lk^{\Sing}_V(y) = S\Big(e_G\big(J(\Oh_{\PP(E^*)}(1))\big)\Big)\Bigm/y.
\end{equation*}
Suppose that moreover $G,X,E$ and the actions of $G$ on $E$ and $X$ are algebraic, and that $G$ is affine and $X$ is compact. Then the above identity is true modulo torsion.
\end{thm}
\begin{proof}
By Definition~\ref{linkingclasssing} we have
\begin{equation*}
O^*\circ \Lk^{\Sing}_V(y) = O^*\circ j_E^*\circ \Lk^{\Sigma}_{j_E(V)}(y). 
\end{equation*}
Since the map $j_E$ is $G$-equivariant, the composition $O^*\circ j_E^*$ is $O^*_1$ where $O_1$ is the orbit map of $j_E(s_0)$. Applying Corollary~\ref{maincorollary} we see that modulo $\dec_{\Q}^*$ we have
\begin{equation*}
O^*\circ \Lk^{\Sing}_V(y) = O^*_1\circ \Lk^{\Sigma}_{j_E(V)}(y) = S\Big(e_G\big(J(\Oh_{\PP(E^*)}(1))\big)\Big)\Bigm/y
\end{equation*}

To prove that in the algebraic case the identity in fact holds modulo torsion and not just $\dec^*_{\Q}$, we compare the weights. By Remark~\ref{imagebargammaprimitive}, the weight of every decomposable element of $H^*(G,\Q)$ is at least the degree plus 2. On the other hand, rationally both $O^*\circ \Lk^{\Sing}_V(y)$ and $S\Big(e_G\big(J(\Oh_{\PP(E^*)}(1))\big)\Big)\Bigm/y$ are pure of weight equal the degree plus 1: for former class this follows straight from the definition, and for the latter we use Corollary~\ref{seceulersub}, formula~\ref{formulas1} and Remark~\ref{imagebargammaprimitive}.
\end{proof}

\begin{cor}
\label{maincorollary2}
Let $G$ be an affine complex algebraic group, and let $L$ be a $G$-equivariant algebraic line bundle over a smooth compact complex algebraic $G$-variety $X$. Suppose $V$ is a $G$-stable subset of $\Gamma(X,L)$. Then the following identity holds modulo torsion in $H^*(G,\Z)$ for every $y\in H_*(X,\Z)$: 
\begin{equation*}
O^*\circ \Lk^{\Sing}_V(y) = S\big(e_G(J(L))\big)\bigm/y.
\end{equation*}
\end{cor}

Unless stated otherwise, Theorem \ref{maintheorem} and Corollary \ref{maincorollary2} are applied below to $V=\Gamma(X,E)$.

\section{Applications}\label{sectapplications}

Let $U$ be a path-connected topological space, and let $G$ be a topological group that acts on $U$ continuously on the left. Note that since $U$ is assumed path connected, all orbit maps $G\to U$ are homotopic (cf.\ Remark~\ref{orbitmapshomotopic}). We will say that {\it the division theorem holds for the action of $G$ on $U$} if the homomorphism $H^*(U,\Q)\to H^*(G,\Q)$ induced by some (and hence every) orbit map is surjective. This is what one might expect to get e.g.\ by applying Theorem~\ref{maintheorem} to $U$ equal the space of regular sections of some holomorphic vector bundle, and in this section we describe a few consequences. In particular, in Section \ref{quotslices} we prove (Theorem~\ref{thmquotslice}) that, assuming $U$ complex affine algebraic and $G$ complex reductive, the division theorem for the action of $G$ on $U$ implies that $H^*(U,\Q)$ is ``divisible'' by $H^*(G,\Q)$, which explains our terminology.

The integral analogue of the division theorem is almost always false. By measuring the extent to which it is false we are able to obtain restrictions on the orders of the $G$-automorphism groups of elements of $U$ (Proposition~\ref{stab}). In Section~\ref{sectionsvsloci} we compare the automorphism groups of sections to those of their zero loci. If $E$ is a $G$-homogeneous vector bundle, we construct a split extension $\widetilde{G}$ of $G$ that acts on $E$ (Notation~\ref{tildeg}). The stabiliser $\widetilde{G}_s$ of a section $s$ of $E$ maps to the stabiliser $G_{Z(s)}$ of the zero locus $Z(s)$ of $s$, and under certain hypotheses this map is surjective (see Lemma~\ref{automorphismsofzeros} and Remark~\ref{autobarg}). In order to check these hypotheses one can use Lemmas~\ref{transitive} and \ref{Koszul condition}. If $E$ has no non-trivial fibrewise automorphisms (covering the identity of the base), then $\widetilde{G}\cong \Co^*\times G$, which leads to simplifications that we describe in Corollary~\ref{descentc*2ndrev}. A related result is Proposition~\ref{degreediscr}, which allows one to calculate the degree of the discriminant variety in many cases of interest (see also Remark~\ref{rmkdegreediscr}).

We summarise the results of Section~\ref{sectapplications} in Theorem~\ref{mainthmappl} and Corollaries~\ref{maincorappl} and \ref{stabilizerisproduct}.

\subsection{Quotients and slices}\label{quotslices}

\begin{thm}\label{thmquotslice}
Let $G$ be a connected complex reductive group that acts algebraically on an affine complex algebraic variety $U$. Suppose the division theorem holds for this action, and let $O:G\to U$ be an orbit map. Then the following is true.
\begin{enumerate}
\item The stabiliser $G_x$ of every point $x\in U$ is finite and the geometric quotient of $U$ by $G$ exists and is affine.
\item There is an \'etale slice, respectively a smooth complex analytic slice through every point of $U$, respectively through every smooth point of $U$.
\item There is an isomorphism
$$H^*(U,\Q)\cong H^*(G,\Q)\otimes H^*(U/G,\Q)$$
of rings that can be made compatible with mixed Hodge structures provided $H^*(U,\Q)$ has a mixed Hodge substructure ${}'H^*$ such that $O^*:{}'H^*\to H^*(G,\Q)$ is an isomorphism.
\end{enumerate}
\end{thm}

\begin{proof} We take an orbit $G(x),x\in U$ of minimal dimension. It is closed, so it is an affine variety, and so the stabiliser $G_x$ of $x$ is reductive by Matsushima's criterion, see \cite{richardson}. If $\dim_{\Co} G_x>0$, then the map $G\to G/G_x$ does not induce a surjective map in rational cohomology, and hence nor does the orbit map of $G\ni g\mapsto gx\in U$.

So all stabilisers are finite, which means that the orbit of every element of $U$ is closed and the geometric quotient of $U$ by $G$ exists, see e.g.\ \cite{New78}*{Chapter 3, \S\S 3,4}. The statement about slices follows from Luna's \'etale slice theorem, see \cite{luna} and \cite{VP89}*{6.6}.

The third statement follows from \cite{PS03}*{Section 2} and the fact that the rational cohomology of a connected Lie group is a free graded commutative algebra.
\end{proof}

\begin{rmk}
One can check the condition in part 3 of Theorem~\ref{thmquotslice} using the following observations. Suppose there is a map of mixed Hodge structures $f:{}'P^*\to H^*(U,\Q)$ such that ${}'P^*$ is polarisable and pure in each degree and $O^*\circ f$ induces a surjection from ${}'P^*$ to $P^*_{\Q}$, the space of primitive elements of $H^*(G,\Q)$, see Remark~\ref{imagebargammaprimitive}. Then there is a mixed Hodge substructure ${}'H^*\subset H^*(U,\Q)$ that maps isomorphically to $H^*(G,\Q)$ under $O^*$.

In particular, if $W_i H^i(U,\Q)=0$ for all $i>0$, which will be the case in all our examples, one can take ${}'P^i=W_{i+1} H^i(U,\Q)$ and $f=\id$.
\end{rmk}

For completeness we compare two natural topologies on the geometric quotient.

\begin{prop}
Suppose $G$ is reductive and the stabiliser of every point of $U$ is finite. Then the geometric quotient with the complex analytic topology coincides with the topological quotient space, i.e.\ the quotient space of $U$ with the complex analytic topology by the equivalence relation given by the action of $G$.
\end{prop}
\begin{proof}
We note that there is a natural bijective continuous map from the topological quotient to the geometric quotient with the complex analytic topology. To show that it is a homeomorphism it would suffice to prove that it is open. By the \'etale slice theorem, the latter statement would follow if we prove that if $x\in U$ and $S$ is an \'etale slice through $x$, then the map
\begin{equation}\label{open}
G\times_{G_x} S\rightarrow (G\times_{G_x} S)/G
\end{equation}
is open in the complex analytic topology. Here $G\times_{G_x} S$ is the quotient of $G\times S$ by the following action of $G_x$:
$$g_1(g,s)=(gg_1^{-1}, g_1s), g_1\in G_x,g\in G,s\in S.$$

It follows from \cite{VP89}*{6.6} that if a finite group $G$ acts on a complex affine variety $W$, then the complex analytic topology on $W/G$ is the same as the quotient of the complex analytic topology on $W$. Using this and the diagram
$$
\begin{tikzcd}
G\times S \arrow[r]\arrow[d] \arrow[drr]& S \arrow[r] & S/G_x\arrow[d,"\cong"]\\
G\times_{G_x} S\arrow [rr] && (G\times_{G_x} S)/G
\end{tikzcd}
$$
we see that map \ref{open} is open.
\end{proof}

\subsection{Automorphism groups}\label{automorphismgroups}

Here we explain how Theorem~\ref{maintheorem} can be applied to get bounds on the order of various automorphism groups. We begin with the following observation.

\begin{prop}\label{stab}
Suppose a connected Lie group $G$ acts continuously with finite stabilisers on a path-connected topological space $U$. Let $k$ be the maximum integer such that $H^k(G,\Z)\neq 0$. Suppose there is a class $a\in H^n(U,\Z)$ such that for some (and hence for every) $G$-orbit map $O:G\to U$ the pullback $O^*(a)$ generates a subgroup of $H^k(G,\Z)\cong\Z$ of index $m$. Then for every $x\in U$ the order of the stabiliser $G_x$ divides $m$.
\end{prop}

\begin{proof}
Take an $x\in U$ and let $K\subset G$ be a maximal compact subgroup of $G$ that contains the stabiliser $G_x$. We now use the following commutative diagram:
\begin{equation*}
\begin{tikzcd}
K\arrow[r,"\simeq"]\arrow[d] & G \arrow[r,"O"]\arrow[d] & U\\
K/G_x\arrow[r,"\simeq"] & G/G_x\arrow[ur] &
\end{tikzcd}
\end{equation*}
\end{proof}


\subsubsection{Automorphisms of sections versus automorphisms of zero loci.}\label{sectionsvsloci} We now explain how to relate the automorphism groups of regular sections to those of their zero loci. 

\begin{ntt} We use $Z(s)$ to denote the zero locus of a section $s$ of a vector bundle.
\end{ntt}

\begin{ntt} If $G$ is a topological group acting on a space $X$ and $Z\subset X$, then we denote the subgroup of $G$ which preserves $Z$ (not necessary pointwise) by $G_Z$.
\end{ntt}

\begin{ntt}
	If $p\colon E\to X$ is a complex vector bundle, then let $\aut_X(E)$ be the group of all couples $(f,g)$ where $g\in\aut(X)$ and $f$ is a fibrewise automorphism of $E$ that covers $g$. Let $\aut(E/X)\subset \aut_X(E)$ be the normal subgroup of all couples $(f,\id_X)$.
\end{ntt}

\begin{ntt}\label{tildeg}
Suppose a group $G$ acts on $X$. Let $\widetilde{G}$ be the fibre product $G\times_{\aut(X)}\aut_X(E)$, and let $p\colon \widetilde{G} \to G$ be the projection map. Denote the image of $G$ in $\aut(X)$ by $\bar G$.
\end{ntt}
\begin{lmm}\label{tildeGandG}
If $E$ is a $G$-equivariant vector bundle, then there exists an exact sequence of groups
\begin{equation*}
	1\to \aut(E/X) \to \widetilde{G} \xrightarrow{p} G \to 1.
\end{equation*}
Moreover, this exact sequence naturally splits. 
\end{lmm}

\begin{proof}
The construction of all maps in the exact sequence is straightforward. Moreover, the $G$-equivariant structure on $E$ allows one to lift the action map $G\to \aut(X)$ lifts to a homomorphism $G\to \aut_X(E)$, which gives us a splitting $G\to \widetilde{G}$.
%
\end{proof}

\begin{rmk}\label{central}
The group $\aut(E/X)$ always contains the central subgroup $\cong\Co^*$ of scalar automorphisms. If $\aut(E/X)$ coincides with this subgroup, then by Lemma~\ref{tildeGandG} we have $\widetilde{G}\cong\Co^*\times G$.
\end{rmk}


From now on we assume that $X$ is smooth complex algebraic, $E$ is algebraic and $G$ acts on $E$ and $X$ by biregular transformations.
Then $\widetilde{G}$ acts on $\reg{X,E}$ and the action of $G$ on $\reg{X,E}$ induced by the splitting $G\to\widetilde{G}$ is the same as the given one. 
Note that the projection $\widetilde{G} \to G$ induces a homomorphism $p_s\colon \widetilde{G}_s \to G_{Z(s)}$. 

\begin{lmm}
\label{automorphismsofzeros}
	Let $s\in\reg{X,E}$ be a regular section. If $\aut(E/X)$ acts transitively on the set of all $s'\in\reg{X,E}$ such that $Z(s')=Z(s)$, then $p_s$ is surjective.
\end{lmm}

\begin{proof}
	If $g\in G_{Z(s)}$, then $Z(gs)=Z(s)$. By our assumption there is an $h\in \aut(E/X)$ such that $hgs = s$. So $hg \in \widetilde{G}_s$, and $p_s(hg)=g$.
\end{proof}

\begin{rmk}\label{autobarg}
	Note that $\ker(p_s)=\ker(p)\cap \widetilde{G}_s = \aut(E/X)_s$. So under the assumptions above we have a split exact sequence
	\begin{equation*}
		1 \to \aut(E/X)_s \to \widetilde{G}_s \to G_{Z(s)} \to 1,
	\end{equation*}
	in particular, $|G_{Z(s)}|$ divides $|\widetilde{G}_{s}|$, and $|\bar G_{Z(s)}|$ divides $\dfrac{|\widetilde{G}_{s}|}{|\ker (G\to\aut X)|}$, assuming all these numbers are finite.
\end{rmk}

In view of Lemma \ref{automorphismsofzeros} we need to know whether the action of $\aut(E/X)$ on the set of regular sections with given zero locus is transitive. From now on we assume that $X$ is compact. Then for line bundles the answer is completely straightforward.

\begin{lmm}\label{linebundletrivial}
If $E$ is a line bundle, then we have $\widetilde{G}\cong \Co^*\times G, \aut(E/X)\cong\Co^*$, and the latter group acts transitively on the set of regular sections with any given zero locus. \qed
\end{lmm}

For general vector bundles there is a sufficient condition, which we will now describe. The group $\aut(E/X)$ is the group of invertible elements of the global $\Hom(E,E)$. 
Let $s_0$ be a regular section of $E$, and let $\mathcal{I}\subset \Oh_X$ be the ideal sheaf of $Z=Z(s_0)$. Then $\Hom(E^*,\mathcal{I})$ embeds in $\Hom(E^*,\Oh_X)=\Gamma(X,E)$ and the intersection $\Hom(E^*,\mathcal{I})\cap \reg{E}$ can be viewed both as the set $U_{\mathrm{surj}}$ of surjective morphisms $\in\Hom(E^*,\mathcal{I})$ and as the set of sections $s$ such that $Z(s)=Z$. 
Notice that $U_{\mathrm{surj}}$ is an irreducible algebraic variety.

The section $s_0\in \Hom(E^*,\mathcal{I})\subset \reg{E}$ induces a homomorphism 
$$\varphi\colon \Hom(E^*,E^*) \to \Hom(E^*,\mathcal{I}).$$

\begin{lmm}\label{transitive}
Suppose that the $\aut(E/X)$-action on $\reg{E}$ has finite stabilisers and the homomorphism $\varphi$ is surjective. Then $\aut(E/X)$ acts transitively on $U_{\mathrm{surj}}$.
\end{lmm}

\begin{proof}
{The map $\varphi$ is assumed surjective, so the $\aut(E/X)$-orbit of $s_0$, which coincides with the image $\varphi(\aut(E/X))$, is open in $U_{\mathrm{surj}}$. Since all stabilisers of the action of $\aut(E/X)$ on $U_{\mathrm{surj}}$ are finite, this action must have only one orbit.}
\end{proof}

There is a cohomological condition that ensures the surjectivity of $\varphi$. Using the Koszul resolution of $\mathcal{I}$
\begin{equation*}
0 \to \det E^* \to \dots \to \Lambda^2E^* \to E^* \to \mathcal{I} \to 0,
\end{equation*}
one can construct a third quadrant spectral sequence that converges to $\Hom(E^*,\mathcal{I})$:
\begin{equation*}
E^{p,q}_2=\Ext^p(E^*, \Lambda^{1-q}E^*) \Rightarrow \Ext^{p+q}(E^*,\mathcal{I}).
\end{equation*}
The edge map $E^{0,0}_2\to E^{0,0}_{\infty}$ in this spectral sequence is precisely the homomorphism $\varphi$.

\begin{lmm}
\label{Koszul condition}
If $\Ext^p(E^*,\Lambda^{p+1}E^*)=0$, $p>0$ then the homomorphism $\varphi\colon\Hom(E^*,E^*)\to \Hom(E^*,\mathcal{I})$ is surjective. \qed
\end{lmm}

\begin{cor}
\label{Sum of line bundles}
Let $L$ be an ample $G$-equivariant line bundle over $X$. Suppose that $E=L^{\oplus r}$, $r=\rk(E)\leq \dim(X)$. Then for any $s\in \reg{X,E}$ the map $p_s\colon \widetilde{G}_s\to G_{Z(s)}$ is an isomorphism.
\end{cor}
\begin{proof}
{First we observe that the action of $\aut(E/X)$ on $\reg{E}$ is in fact free.} Indeed, $\Gamma(X,E)$ can be identified with $\Hom(\Co^r,\Gamma(X,L))$, and $\aut(E/X)$ with $GL_r(\Co)$ {so that} under these identifications the group $\aut(E/X)$ acts by precomposition. Moreover, {the linear maps $\Co^r\to \Gamma(X,L)$ that correspond to regular sections are injective}, and $GL_r(\Co)$ acts freely on the set of such maps.

{Next we note that the hypotheses of Lemma~{\ref{Koszul condition}} hold for $E$ because $H^p(X,(L^{\otimes p})^*)$ vanish for $0<p\leq \rk(E)-1<\dim(X)$, which follows by the Kodaira vanishing theorem. So the map $\varphi$ for $E$ is surjective, and we finish the proof applying Lemmas~{\ref{transitive}} and~{\ref{automorphismsofzeros}}.}
\end{proof}

Finally, we give an example which shows that $\aut(E/X)$ may not act transitively on $U_{\mathrm{surj}}$.

\begin{exmp}
Take $X=\PP^n(\Co), n\geq 2$ and $E=T\PP^n(\Co)$. On the one hand, all automorphisms of $E$ are scalar, and on the other the space $\Gamma(X,E)$ can be identified with the space of traceless endomorphisms of $\Co^{n+1}$ so that the zeroes of a section are the eigenlines of the corresponding endomorphism.
\end{exmp}
 
\subsubsection{$\aut(E/X)=\Co^*$.}\label{autc*} Let $G$ be a complex reductive group, and let $E$ be a $G$-equivariant vector bundle over a smooth compact complex algebraic $G$-variety $X$. It follows from Remark~\ref{autobarg} that (under some assumptions) the order of the $G$-automorphism group of $Z(s),s\in\reg{X,E}$ divides $|\widetilde G_s|$, which in turn, as we will shortly see, divides the index of $\Lambda^{k}_{E', \widetilde{G}}$ (see Notation \ref{lambdae}) in $H^{k}(\widetilde G,\Z)$; here $E'$ is a certain $\widetilde{G}$-equivariant line bundle and $k$ is the top degree in which the cohomology of $\widetilde{G}$ is non zero. We will now explain how calculating the index simplifies if the fibrewise automorphism group $\aut(E/X)$ is $\Co^*$, which will be the case in most examples we consider in Section \ref{secexamples}. Let us choose a generator $\mathbf{x}$ of $H^1(\Co^*,\Z)$.

\begin{lmm}\label{lkfundx}
Let $E'$ be a rank $r$ complex vector bundle over a CW-complex $X'$ of dimension $<2r$. We consider the standard action of $\Co^*$ on the total space of $E'$ (so the resulting action on $X'$ is trivial). Let $V\subset C^0(X,E)$ be a vector subspace, and let $O\colon \Co^*\to V_0$ be an orbit map (recall that $V_0$ denotes the space of nowhere vanishing sections $\in V$, see Section \ref{linkingclasssubsect}). Then if $y\in H_{2r-2}(X',\Z)$, the class $O^*(\Lk_V^\Sigma(y))\in H^1(\Co^*,\Z)$ (see Definition \ref{linkingclasshomomorphism}) is $\pm\langle c_{r-1}(E'),y\rangle\mathbf{x}$.
\end{lmm}

\begin{proof} The homotopy quotient $X'_{h \Co^*}\simeq B\!\Co^*\times X'$, and a straightforward check shows that the vector bundle~$E'_{h\Co^*}$ is $\gamma^1\boxtimes E$, where $\gamma^1$ denotes the tautological line bundle over $B\Co^*$. So
$$e_{\Co^*}(E')=e(E'_{h\Co^*})=\sum_{i=0}^r c^{r-i}\times c_i(E')=\sum_{i=0}^{r-1} c^{r-i}\times c_i(E')$$ where $c=c_1(\gamma^1)$. The ideals $I(X,\Z)$ and $I_1(X,\Z)$ coincide and are generated by $c\times 1$. Using formula~\ref{formulas1} and Proposition~\ref{kernel} we see that for every $y\in H_{2r-2}(X',\Z)$ we have
$$S(e_{\Co^*}(E'))=S(c\times c_{r-1}(E'))=\bar\gamma(c)\times \alpha^*(1\times c_{r-1}(E'))=\bar\gamma(c)\times c_{r-1}(E')\in H^{2r-1}(\Co^*\times X',\Z).$$ By Corollary \ref{maincorollary},
$$O^*(\Lk_V^\Sigma(y))=S(e_{\Co^*}(E'))/y=(\bar\gamma(c)\times c_{r-1}(E'))/y=\langle c_{r-1}(E'),y\rangle \bar\gamma(c).$$ Now observe that $\bar\gamma(c)\in H^1(\Co^*,\Z)$ is $\pm\mathbf{x}$ by Remark~\ref{gamma and transgression}, cf.\ {Example}~\ref{gln}.
\end{proof}

\begin{cor}\label{descentc*2ndrev}
Let $G$ be a connected complex reductive group that acts algebraically on a smooth compact complex algebraic variety $X$ of dimension $d$, and let $E$ be a $G$-equivariant algebraic vector bundle of rank $r$ over $X$. We suppose that $\aut(E/X)\cong\Co^*$ and $H^1(G,\Z)=0$. Set $E'=J(\Oh_{\PP^*(E)}(1))$ (see Section~\ref{secsmoothsec}), and suppose that $\delta=\langle c_{d+r-1}(E'),[\PP(E^*)]\rangle\neq 0$. Recall that $\widetilde{G}\cong\Co^*\times G$ (Remark~\ref{central}). Let $\tilde{\mathbf{x}}\in H^1(\widetilde{G},\Z)$ be the pullback of $\mathbf{x}\in H^1(\Co^*,\Z)$ under the projection $\widetilde{G}\to\Co^*$. 

\begin{enumerate}
\item The class $S(e_{\widetilde G}(E'))/[\PP(E^*)]\in H^1(\widetilde G,\Z)$ is $\pm\delta\tilde{\mathbf{x}}$. Moreover, if $y\in H_{<2(d+r-1)}(\PP(E^*),\Z)$, then modulo torsion the class $S(e_{\widetilde G}(E'))/y$ is the pullback of $S(e_{G}(E'))/y$ under $\widetilde{G}\to G$.

\item The ring $\Lambda^*_{E', G}$ has finite index in $\fr{*}(G,\Z)$ if and only if $\Lambda^*_{E', \widetilde{G}}$ has finite index in $\fr{*}(\widetilde G,\Z)$, and the index of $\Lambda^{\dim \widetilde G}_{E', \widetilde{G}}$ in $H^{\dim \widetilde G}(\widetilde G,\Z)$ is $\delta$ times the index of $\Lambda^{\dim G}_{E', {G}}$ in $H^{\dim G}(G,\Z)$.
\end{enumerate}
\end{cor}

\begin{proof}
To prove the first statement we use Lemma~\ref{lkfundx}, Remark~\ref{central}, the functoriality of the map $S$ (see diagram~\ref{functorialityofs}) and the following observations: Firstly, using Remark~\ref{imageofs} we see that the image of every $S(e_{\widetilde G}(E'))/y, y\in H_*(\PP(E^*),\Z)$ in $\fr{*}(\widetilde{G},\Z)$ is primitive. Secondly, the self-map of $\fr{i}(\Co^*\times G,\Z),i>1$ induced by the composition of the projection $\Co^*\times G\to G$ and the inclusion $G\to \Co^*\times G$ is the identity on the primitive elements.

The second statement of the corollary follows from the first.
\end{proof}

The class $c_{r-1}(E')$ that appears in Lemma \ref{lkfundx} has a geometric interpretation in terms of the degrees of discriminant varieties, which we will now describe.  
Here is a setup that suffices for many applications. 
Let $X$ be an irreducible compact complex algebraic variety of dimension $d$ and let $E'$ be an algebraic vector bundle of rank $r>d$ over $X$. {Let $V\subset\Gamma(X,E')$ be a vector subspace}, and let $\Sigma\subset V$ be the set of all sections with non-empty zero locus. Moreover, set $\widetilde\Sigma=\{(s,x)\in V\times X\mid s(x)=0\}$. We require that there should be a Zariski open $X'\subset X$ such that the restriction of $\widetilde\Sigma$ to $X'$ is dense in $\widetilde\Sigma$ and is the total space of an algebraic vector bundle over $X'$ of rank $\dim V-r$. Note that if $E'$ is globally generated by elements of $V$, then one can take $X'=X$.

\begin{prop}\label{degreediscr}
If $\langle c_d(E'),[X]\rangle =0$, then all elements of $\Sigma$ have infinitely many zeroes and $\dim\Sigma< d+\dim V-r$. If $\langle c_d(E'),[X]\rangle\neq 0$, then a Zariski general element of $\Sigma$ has $k>0$ zeroes, $\dim\Sigma=d+\dim V-r$, and $k\deg \Sigma= \langle c_d(E'),[X]\rangle$.
\end{prop}

\begin{rmk}\label{rmkdegreediscr}
This proposition is a variant of \cite{GKZ}*{Theorem 3.3.10}. Note however that the version given in~\cite{GKZ} is not correct because a general section $\in\Sigma$ may have more than one zero: one can take for instance $X=\mathbb{P}^1(\Co)$ and $E'=\mathcal{O}(2)\oplus \mathcal{O}$. Then $E'$ is very ample as defined in \cite{GKZ}*{p.\ 110}, and $\Sigma\subset V=\Gamma(X,E')$ is a hyperplane, but $\langle c_d(E),[X]\rangle=2$. Nevertheless, the argument given in \cite{GKZ} can be easily adapted to give a proof of Proposition~\ref{degreediscr}, which we do in Appendix \ref{app_proofs}.
\end{rmk}

\subsection{Summary}\label{summaryappl} Here we summarise the results of this section that we will need in Section \ref{secexamples}.

\begin{thm}\label{mainthmappl}
Let $G$ be a connected complex reductive group that acts algebraically on a smooth compact complex algebraic variety $X$ of dimension $d$, and let $L$ be an $G$-equivariant algebraic line bundle over $X$. Recall that the ring $\Lambda^*_{J(L)}\subset \fr{*}(G,\Z)$ is generated by $S\big(e_G(J(L))\big)\bigm/y, y\in H_*(X,\Z)$ (see Notation \ref{lambdae}). Let $O:G\to\reg{X,L}$ be an orbit map.

\begin{enumerate}
\item Suppose that $\langle c_{d}(J(L)),[X]\rangle\neq 0$, {the discriminant $\Sing(L)=\Gamma(X,L)\setminus\reg{X,L}$ is equidimensional,} and $\Lambda^*_{J(L)}$ has finite index in $\fr{*}(G,\Z)$. Then the map $O^*$ in rational cohomology is surjective, the geometric quotient $\reg{X,L}/G$ exists and is affine, and
$$H^*(\reg{X,L},\Q)\cong H^*(G,\Q)\otimes H^*(\reg{X,L}/G,\Q)$$ both as rings and mixed Hodge structures.
For every $s\in\reg{X,E}$ the order of the stabiliser $G_s$ divides the index of $\Lambda^{\dim G}_{J(L)}$ in $H^{\dim G}(G,\Z)$.
\item Suppose that $J(L)$ is globally generated by holonomic sections (see Section~\ref{secsmoothsec}), and the map $O^*:H^*(\reg{X,E},\Q)\to H^*(G,\Q)$ is surjective. Then $\Lambda^*_{J(L)}$ has finite index in $\fr{*}(G,\Z)$.
\end{enumerate}
\end{thm}

\begin{proof} {To prove the first part, we note that by Lemma~{\ref{lkfundx}} we have $H^1(\reg{X,L},\Z)\neq 0$, which by the Alexander duality and our assumption on $\Sing(L)$ implies that} $\reg{X,L}$ is the complement of a hypersurface in an affine space. 
We then apply Corollary \ref{maincorollary2}, Theorem \ref{thmquotslice} and Proposition \ref{stab}.

The second part of the theorem follows from Proposition~\ref{mhsmotivation2}
\end{proof}


\begin{cor}\label{maincorappl}
Let $G$ and $X$ be as in Theorem \ref{mainthmappl}, and let $E$ be a $G$-equivaviant algebraic vector bundle over $X$ of arbitrary rank $r$. Suppose $G'$ is an arbitrary connected reductive subgroup of the group $\widetilde G$ (see Notation \ref{tildeg}) and take $L=\Oh_{\PP(E^*)}(1)$ (see Section \ref{secsmoothsec}). Let $O:G'\to\reg{X,E}$ be an orbit map.

\begin{enumerate}
\item Suppose that $\langle c_{d+r-1}(J(L)),[\PP(E^*)]\rangle \neq 0$, {the discriminant $\Sing(E)=\Gamma(X,E)\setminus\reg{X,E}$ is equidimensional,} and $\Lambda^*_{J(L),G'}$ has finite index in $\fr{*}(G',\Z)$. Then the map $O^*$ in rational cohomology is surjective, the geometric quotient $\reg{X,E}/G'$ exists and is affine, and
$$H^*(\reg{X,E},\Q)\cong H^*(G',\Q)\otimes H^*(\reg{X,E}/G',\Q)$$ both as rings and mixed Hodge structures. Moreover, for every $s\in\reg{X,E}$ the order of the stabiliser $G'_s$ divides the index of $\Lambda^{\dim G'}_{J(L),G'}$ in $H^{\dim G'}(G',\Z)$.
\item Suppose that $J(L)$ is globally generated by holonomic sections and $$O^*:H^*(\reg{X,E},\Q)\to H^*(G',\Q)$$ is surjective. Then $\Lambda^*_{J(L),G'}$ has finite index in $\fr{*}(G',\Z)$.
\end{enumerate}
\qed
\end{cor}

\begin{rmk}\label{remarkonconditionsinmaincorappl}
If $J(L)$ is globally generated by holonomic sections, then $\Sing(E)$ is irreducible. So if $E$ satisfies the former condition and $\langle c_{d+r-1}(J(L)),[\PP(E^*)]\rangle \neq 0$, then $\Lambda^*_{J(L),G'}$ has finite index in $\fr{*}(G',\Z)$ if and only if $O^*$ is surjective in rational cohomology. We also note that if $L$ is very ample, then $J(L)$ is globally generated by holonomic sections, and that $\Sing(E)$ is irreducible if the action of $\widetilde{G}$ on $\PP(E^*)$ is transitive; for line bundles the latter holds if and only if the action of $G$ on $X$ is transitive.
\end{rmk}

Recall that for every $y\in H_*(\PP(E^*),\Z)$ the image of 
$\big(S(e_{G'}(J(L)))\big)\bigm/y$ in $\fr{*}(G',\Z)$ is primitive by Remark~\ref{imageofs}. Let $i_l$ be the order of the cokernel of the map 
\begin{equation}\label{yetanothers}
H_{2(r-1+d)-2(l-1)}(\PP(E^*),\Z)\rightarrow P^{2l-1}, y\mapsto \big(S(e_{G'}(J(L)))\big)\bigm/y
\end{equation}
where $P^*\subset \fr{*}(G',\Z)$ denotes the graded group of primitive elements. 



\begin{cor}
\label{stabilizerisproduct}
We {keep the notation and assumptions} of Corollary \ref{maincorappl}. 
Suppose all $i_l$ are finite.
Then $\Lambda^*_{J(L),G'}$ has finite index in $\fr{*}(G',\Z)$ and the index of $\Lambda^{\dim G'}_{J(L),G'}$ in $H^{\dim G'}(G',\Z)$ is $\prod_{l} i_{l}$. In particular, for every $s\in \reg{X,E}$ the order $|G'_s|$ divides $\prod_{l} i_{l}$. \qed
\end{cor}

Finally, we note that if $H^1(G,\Z)=0$, $\aut(E/X)=\Co^*$, and $E'=J(\Oh_{\PP(E^*)}(1))$, then Corollary \ref{descentc*2ndrev} allows one to deduce the information about $\Lambda^*_{E', \widetilde{G}}$ that we need to apply Corollary \ref{maincorappl} to $G'=\widetilde G$ from the same information about $\Lambda^*_{E', G}$.

\section{Examples}\label{secexamples}

This section contains examples and methods for computing them. For our applications it will suffice to work with rational cohomology whilst keeping track of the integral lattices.

%


In this section we take $G$ to be a connected complex affine algebraic group, $P$ a parabolic subgroup of $G$, $X=G/P$, $E$ an algebraic $G$-equivariant vector bundle on $X$, and $V=\Gamma(X,E)$, the space of holomorphic sections of $E$, and explicitly describe in several examples the cohomology classes $S(e_{\widetilde{G}}(E'))$ modulo torsion; here $\widetilde{G}$ is the group defined in Section~\ref{sectionsvsloci}, $y\in H_*(\PP(E^*),\Z), E'=J(\Oh_{\PP(E^*)}(1)$ (see Section~\ref{secsmoothsec}), and the classes $S(e_{\widetilde{G}}(E'))$ live in $H^*(\tilde G,\Z)$. Recall that by Theorem~\ref{maintheorem}, in $\fr{*}(G,\Z)$ the classes $S(e_{\widetilde{G}}(E'))/y$ coincide with the pullbacks $O^*(\Lk_V^{\Sing}(y))$ under the orbit map $O:\widetilde{G}\to\reg{X,E}$. Recall also that if $E$ is a line bundle, then $\PP(E^*)=X,E'=J(E)$ and $\widetilde{G}=\Co^*\times G$, see Remark~\ref{rankonecoincidence} and Lemma~\ref{linebundletrivial}, and by Corollary~\ref{descentc*2ndrev} to calculate the image of $S(e_{\widetilde{G}}(E'))/y$ in $\fr{*}$ it suffices to do the same for $S(e_{G}(E'))/y$.

In Section~\ref{genalg} we describe a general procedure that starts with an arbitrary equivariant vector bundle $E'$ over $X$ with vanishing Euler class and calculates the classes $S(e_G(E'))$ modulo torsion. We illustrate this procedure for the complete flag varieties of $G=SL_3(\Co), Sp_4(\Co)$ and $G_2$ in Section~\ref{completeflagvarieties}. If $X$ is ``small'', then one can calculate $S(e_G(E'))$ directly, which we do in Section~\ref{projquad} for the bundles $\Oh(d)$ on projective spaces and non-degenerate quadrics. In Section~\ref{fano} we consider higher rank bundles on Pl\"ucker embedded Grassmannians and give applications to certain Fano varieties. Finally, in Section~\ref{summaryexamples} we summarise our results and {explain how they are related to known results about automorphism groups}.

\subsection{Generalities.}\label{genalg} Let us first recall a few facts about principal and induced bundles.

If $H$ is a Lie group, $\mathcal{P}\to X$ is a principal $H$-bundle and $R:H\to GL(V)$ is a representation, then we will denote the induced vector bundle on $X$ constructed using the natural {\it left} action of $H$ on $\mathcal{P}\times V$ by $\mathcal{P}\times_R V$, or $\mathcal{P}\times_H V$ when it is clear which representation is meant. If $H$ is a Lie subgroup of a Lie group $H'$, then we can take $\mathcal{P}=H'$. The resulting vector bundle $H'\times_R V$ over $X=H'/H$ is then $H'$-equivariant, and all $H'$-equivariant vector bundles over $H'/H$ are obtained in this way.


\begin{exmp}\label{tautlinebund}
Suppose $H'=GL_n(\Co)$ and $H$ is the stabiliser of a vector supspace $W\subset\Co^n$ of rank $k$. The action of $H$ on $W$, respectively $\Co^n/W$ corresponds to the tautological bundle $U$ of rank $k$, respectively the quotient bundle $Q$ of rank $n-k$ over the Grassmannian $\mathrm{Gr}(n,k)=H'/H$. The cotangent bundle $\Omega_{\mathrm{Gr}(n,k)}\cong \Hom(Q,U)$ as $H'$-equivariant vector bundles.
Note that if $k=1$, then the bundle $\Oh(1)$ over $H'/H=\PP^{n-1}(\Co)$ is $U^*$.

If $H=GL_n(\Co)$ and $R$ is the standard action of $H$ on $\Co^n$, then $EH\times_R\Co^n$ is the tautological rank $n$ bundle $\gamma^n$ over $BH$.
\end{exmp}

\begin{lmm}\label{inducedvectbund}
If $H$ is a Lie subgroup of a Lie group $H'$ and $R:H\to GL(V)$ is a representation of $H$, then there is a commutative diagram
$$
\begin{tikzcd}
H'\times_R V\arrow[r]\arrow[d]&(H'\times_R V)_{hH'}\arrow[d] & EH\times_R V \arrow[d]\arrow[l]\\
H'/H\arrow{r}{\alpha}&(H'/H)_{hH'} & BH\arrow{l}[swap]{\simeq}
\end{tikzcd}
$$
in which the vertical arrows are vector bundle projections, the top horizontal arrows are fibrewise maps of vector bundles and the bottom right horizontal arrow is a homotopy equivalence. \qed
\end{lmm}


\smallskip

Let now $E'$ be a $G$-equivariant vector bundle over $X=G/P$ with Euler class 0. In this subsection we sketch an algorithm that allows one to calculate $S(e_G(E'))$ and the resulting cohomology classes $\in \fr{*}(G,\Z)$. Recall that the ring $\fr{*}(G,\Z)$ is a free graded commutative $\Z$-algebra, see Remark~\ref{imagebargammaprimitive}.

%
In order to apply \ref{formulas1} we need the following data.

\begin{enumerate}
\item A system of generators $s_1,\ldots, s_n$ of $H^*(BG,\Q)$ such that $\bar\gamma(s_1),\ldots, \bar\gamma(s_n)$ are free multiplicative generators of $\fr{*}(G,\Z)$. (Note that we do not require that $s_1,\ldots, s_n$ themselves should be integral.)
\item The class $e_G(E')$ expressed in terms of $\beta^*(s_1), \ldots, \beta^*(s_n)$.
\item The map $\alpha^*:H^*(BP,\Q)\to H^*(X,\Q)$.
\end{enumerate}

We explain these steps in more detail in Sections~\ref{genidealbargamma}-\ref{arbp}; before we do that we make a few observations.

\begin{rmk}\label{sufficessimplyconn}
There is a complex semi-simple simply connected group $H$, a parabolic subgroup $P_H\subset H$ and an algebraic group homomorphism $H\to G$ with finite kernel that takes $P_H$ to $P$ and induces an isomorphism $H/P_H\to G/P$. Using the functoriality of $S$ (see Diagram~\ref{functorialityofs}) we see that for the purposes of calculating $S$ we may assume $G$ semi-simple and simply connected.

Moreover, for such a $G$ one can obtain a system of free multiplicative generators for $\fr{*}(G,\Z)$ by first doing the same for the simple factors of $G$ and then taking the exterior products (this follows from  the K\"unneth formula, see e.g.\ \cite{Sp81}*{Theorem 5.6.1}). So in particular, in step 1 of the above strategy it would suffice to consider the case of $G$ simple and simply connected.
\end{rmk}

\begin{ntt}
If $H$ is a complex algebraic group, then we denote its {\it character group} $\Hom(H,\Co^*)$ by $\mathfrak{X}(H)$.
\end{ntt}

Let $T\subset P$ be a maximal torus of $G$. There is an isomorphism between the character group $\mathfrak{X}(T)$ and $H^2(BT,\Z)$ that takes a character $\chi:T\to\Co^*$ to $c_1(ET\times_\chi\Co)$.
Let $\mathfrak{t}$ be the Lie algebra of $T$. There is an embedding $\mathfrak{X}(T)\to \mathfrak{t}^*=\mathop{\mathrm{Hom}}(\mathfrak{t},\Co)$ that takes $\chi\in\mathfrak{X}(T)$ to the differential $d\chi$ at the identity element. This gives us an integral structure on $\mathfrak{t}^*$. We identify $H^*(BT,\Co)$ with the symmetric $\Co$-algebra $\mathrm{Sym}^*(\mathfrak{t}^*)$ on $\mathfrak{t}^*$ and $H^*(BT,\Z)$ with the symmetric $\Z$-algebra $\mathrm{Sym}^*(\mathfrak{X}(T))$ on $\mathfrak{X}(T)\subset \mathfrak{t}^*$.

\begin{ntt} We denote the $i$-th elementary symmetric polynomial in $x_1,\ldots, x_n$ by $\sigma_i(x_1,\ldots, x_n)$.
\end{ntt}

The following lemma is a straightforward consequence of Lemma \ref{inducedvectbund}.

\begin{lmm}\label{inducedvectbund1}
Let $R:P\to GL(V)$ be a complex representation and set $\bar E=G\times_R V$. Let $\chi_1,\ldots,\chi_{\dim V}\in\mathfrak{X}(T)$ be the weights of its restriction to $T$. Let us identify $H^*(BP,\Q)$ with its image in $H^*(BT,\Q)$ and the homotopy quotient $X_{hG}$ with $BP$. The total Chern class of the bundle $\bar E_{hG}$ over $X_{hG}$ is then
\begin{equation}\label{totalcherneg}
\prod_{i=1}^{\dim V}(1+\chi_i)\in H^*(BP,\Q) \subset H^*(BT,\Q).
\end{equation}
Moreover, $\alpha^*(\sigma_i(\chi_1,\ldots, \chi_{\dim V}))=c_i(\bar E)$ for every $i=1,\ldots, \dim V$. \qed
\end{lmm}

\smallskip

Until the beginning of Section~\ref{arbp} we suppose that $P$ is a Borel subgroup $B$ of $G$, and then we will indicate the modifications that are necessary if $P$ is arbitrary. The classifying space of $B$ is homotopy equivalent to $BT$.

The Weyl group $W=N_G(T)/Z_G(T)$ of $G$ acts on $H^*(BT,\Q)$, and $\beta^*:H^*(BG,\Q)\to H^*(BT,\Q)$ induces an isomorphism between $H^*(BG,\Q)$ and the ring of invariants $H^*(BT,\Q)^W$.
\begin{ntt}
In the sequel we will often identify an element $s\in H^*(BG,\Q)$ with its image in $H^*(BT,\Q)$, and vice versa. In particular, for $x\in H^*(BT,\Q)^W$ we will write $\bar\gamma(x)$ instead of $\bar\gamma((\beta^*)^{-1})(x))$.
\end{ntt}

\subsubsection{Generators of $I_1(X, \Q)$ and the map $\bar\gamma$.}\label{genidealbargamma} Recall that for our purposes it suffices to calculate $\bar\gamma$ for $G$ simple and simply-connected (see Remark~\ref{sufficessimplyconn}). There does not seem to be a way to do this uniformly for all $G$. For example, while calculating $H^*(G,\Q)$ is straightforward, identifying the integral lattice $\fr{*}(G,\Z)\subset H^*(G,\Q)$ in a functorial manner is less so, cf.\ the calculations below for orthogonal and Spin groups.

In order to make our procedure explicit and applicable to a reasonably large class of examples, in Section~\ref{genidealbargamma} we take $G$ to be a classical Lie group or $G_2$ and describe a system of generators $s_1,\ldots, s_n$ of $H^*(BG,\Q)$ such that $\bar\gamma(s_1),\ldots, \bar\gamma(s_n)$ are free multiplicative generators of $\fr{*}(G,\Z)$. The details for the exceptional Lie groups other than $G_2$ will be presented elsewhere. {If the quotient of the Lie algebra of $G$ by the solvable radical is a direct sum of copies of the classical simple complex Lie algebras, we say that $G$ is} {\it of classical type}.


We denote the set of all diagonal square complex matrices by $D$ and let $\varepsilon_i:D\to\Co$ be the map that takes a matrix $\in D$ to the $i$-th entry on the diagonal.

\begin{exmp}\label{gln} Set $G=GL_n(\Co)$. Let $T=G\cap D$, the group of all invertible diagonal matrices. Then the Lie algebra $\mathfrak{t}$ of $T$ is $D$, and the elements $\varepsilon_1,\ldots, \varepsilon_n$ generate the integral lattice $\mathfrak{X}(T)\subset\mathfrak{t}^*$ and freely generate the algebra $\mathrm{Sym}^*(\mathfrak{X}(T))=H^*(BT,\Z)$. The Weyl group $W$ is the group of permutations of $\{\varepsilon_1,\ldots, \varepsilon_n\}$. We set $s_i=\sigma_i(\varepsilon_1,\ldots,\varepsilon_n)$. The ring of invariants $H^*(BT,\Q)^W$ is generated by $s_1,\ldots, s_n$, so these polynomials generate $I_1(X,\Q)$. Note that by Example~\ref{tautlinebund} and Lemma~\ref{inducedvectbund1}, $s_i$ viewed as an element of $H^*(BG,\Q)$ is $c_i(\gamma^n)$, so $s_i, i=1,\ldots, n$ are in fact free multiplicative generators of $H^*(BG,\Z)$. 

It follows from the Leray-Serre spectral sequence of the bundle $EG\to BG$ that there are free multiplicative generators $\mathbf{x}_1,\ldots,\mathbf{x}_{2n-1}$ (where the subscript denotes the degree) of $H^*(G,\Z)=\fr{*}(G,\Z)$ such that $\mathbf{x}_{2i-1}$ transgresses to $s_i$ modulo the ideal generated by $s_{<i}$. By Remark~\ref{gamma and transgression} we have $\bar\gamma(s_i)=\mathbf{x}_{2i-1},i=1,\ldots, n$. Note that $\mathbf{x}_1,\ldots,\mathbf{x}_{2n-1}$ can also be constructed inductively as follows: we take the top generator $\mathbf{x}_{2n-1}$ to be the pullback of an appropriately chosen generator of $H^{2n-1}(\Co^n\setminus\{0\},\Z)$ under $o^*$ where $o:GL_n(\Co)\to\Co^n\setminus\{0\}$ is an orbit map, and we require the remaining generators to restrict to the same generators for $GL_{n-1}(\Co)$.
\end{exmp}

\begin{exmp}\label{sln}Set $G=SL_n(\Co)$ and $T=G\cap D$. For $i=1,\ldots, n$ let $\epsilon_i$ be the restriction of $\varepsilon_i$ to the Lie algebra $\mathfrak{t}$ of $T$. The elements $\epsilon_1,\ldots,\epsilon_n$ generate $\mathfrak{X}(T)$ subject to the relation $\sum\epsilon_i=0$, and we have $H^*(BT,\Q)\cong\Q[\epsilon_1,\ldots,\epsilon_n]/(\sum\epsilon_i)$. The Weyl group is the same as in the previous example, and the ring $H^*(BT,\Q)^W$ is generated by $s_i,i=2,\ldots,n$ where $s_i$ is the image of $\sigma_i(\epsilon_1,\ldots,\epsilon_n)$ in $H^*(BT.\Q)$. By abuse of notation we denote the restriction of $\mathbf{x}_i$ from the previous example to $G$ also by $\mathbf{x}_i$. Then for every $i=2,\ldots, n$ we have $\bar\gamma(s_i)=\mathbf{x}_i$.
\end{exmp}

\begin{exmp}\label{sp2n} Set $G=Sp_{2n}(\Co)$. We embed $G$ in $GL_{2n}(\Co)$ as the stabiliser of the skew-symmetric bilinear form with matrix $\begin{pmatrix} 0& I_n\\-I_n & 0\end{pmatrix}$. Set $T=G\cap D$. Note that the Lie algebra $\mathfrak{t}$ of $T$ is $$\{\mathrm{diag}(x_1,\ldots, x_n,-x_1,\ldots,-x_n)\mid\mbox{all }x_i\in\Co\}.$$
The elements $\varepsilon_i,i\leq n$ generate the group $\mathfrak{X}(T_1)\subset\mathfrak{t}^*$ and freely generate the algebra $\mathrm{Sym}^*(\mathfrak{X}(T))=H^*(BT,\Z)$. The Weyl group $W$ is the group of all permutations and sign changes of $\{\varepsilon_1,\ldots,\varepsilon_n\}$, so the ring $H^*(BT,\Q)^W$ is generated by $s_1,\ldots,s_n$ where each $s_i=\sigma_i(\varepsilon_1^2,\ldots,\varepsilon_n^2)$. By using power sums one can check that, modulo decomposable elements, the polynomial $\sigma_{i}(\varepsilon_1,\ldots,\varepsilon_{2n})$ restricted to $\mathfrak{t}$ is 0 if $i$ is odd and $(-1)^{\frac{i}{2}} s_{\frac{i}{2}}$ if $i$ is even. Note that it follows from the definition of $\bar\gamma$ (see Notation~\ref{gamma}) that $\bar\gamma$ is zero on decomposable elements.

Let us now construct a system of free multiplicative generators of $H^*(Sp_{2n}(\Co),\Z)$ and compare it with the one we constructed above for $H^*(GL_{2n}(\Co),\Z)$. The group $Sp_{2n}(\Co)$ is homotopy equivalent to $GL_n(\HH)\subset GL_{2n}(\Co)$, as both deformation retract onto the quaternionic unitary group $Sp_{2n}(\Co)\cap U(2n)$. We can construct a system of free multuplicative generators for $GL_n(\HH)$ in the same way as for $GL_n(\Co)$ above in Example~\ref{gln}. Moreover, using the commutative diagram
\begin{equation*}
\begin{tikzcd}
GL_{n-1}(\HH)\arrow[r]\arrow[d] & GL_n(\HH)\arrow[r]\arrow[d] & \HH^n\setminus\{0\} \arrow[d]\\
GL_{2n-2}(\Co)\arrow[r] & GL_{2n}(\Co)\arrow[r] & \Co^{2n}\setminus\{0\}
\end{tikzcd}
\end{equation*}
(in which both horizontal arrows on the right are orbit maps) we see that we get a system of free multiplicative generators of $H^*(Sp_{2n}(\Co),\Z)$ by restricting the classes $\mathbf{x}_i$ from the Example~\ref{gln} with $i\equiv 3\mod 4$ from $GL_{2n}(\Co)$ to $Sp_{2n}(\Co)$. The map $\bar\gamma$ is functorial in $G$, so using the previous paragraph we see that $\mathbf{y}_{4i-1}=\bar\gamma(s_i)\in H^*(G,\Q),i=1,\ldots,n$ are integral and form a system of free multiplicative generators of $H^*(Sp_{2n}(\Co),\Z)$.
\end{exmp} 

In order to handle the orthogonal and Spin groups we need the following lemma.

\begin{lmm}\label{auxso}
\begin{enumerate}
\item The inclusion $SO_{2n-1}(\R)\subset SO_{2n+1}(\R)$ induces a surjective map of $H^*(-,\Z)$ (and hence so does $SO_{2n-1}(\R)\subset SO_{2n}(\R)$).
\item Let $o:SO_{2n}(\R)\to S^{2n-1}$ be the orbit map of an element of $S^{2n-1}$. The group $o^*(H^{2n-1}(S^{2n-1},\Z))$ is $\cong\Z$ and is a direct summand of $H^{2n-1}(SO_{2n}(\R),\Z)$.
\item Let $V_2(\R^{2n+1})$, respectively $V_2(\Co^{2n+1})$ be the Stiefel manifold of orthogonal 2-frames in $\R^{2n+1}$, respectively unitary 2-frames in $\Co^{2n+1}$, and let $i:V_2(\R^{2n+1})\to V_2(\Co^{2n+1})$ be the natural inclusion. Then the image of the integral cohomology map $i^*$ in degree $4n-1$ has index 2.
\end{enumerate}
\end{lmm}

\begin{proof}
The first two statements can be checked using the explicit CW-structure on $SO_k(\R)$ given e.g.\ in \cite{H02}*{\S 3D}. Let us prove the third statement.

We equip $\Co^{2n+1}$ with the standard Hermitian metric. Let $e$ be the vector $\frac{1}{\sqrt{2}}(i,1,0,\ldots,0)$ in the unit sphere $S_{\Co}\cong S^{4n+1}\subset\Co^{2n+1}$, and set $K=\{ze\mid z\in\Co, |z|=1\}$. We construct a section $f$ of the projection $p:V_2(\Co^{2n+1})\to S_{\Co}$ over the complement $S_{\Co}\setminus K$ by taking $f(x),x\in S_{\Co}\setminus K$ to be the couple $(x,$ the normalised orthogonal projection of $e$ to the orthogonal complement of $x)$. Let $M$ be the image of $f$. 
The circle $K$ does not intersect the real unit sphere $S_{\R}\cong S^{2n}\subset \R^{2n+1}\subset \Co^{2n+1}$ and bounds a smooth embedded 2-disk $D\subset S_{\Co}$ that for dimension reasons may also be assumed not to intersect $S_{\R}$.

Let $D'$ be a closed neighbourhood of $D$ that is diffeomorphic to a $4n+1$-disk and does not intersect $S_{\R}$, and let $D''$ be $\overline{S_{\Co}\setminus D'}$. The intersection $M\cap p^{-1}(\partial D')=f(\partial D')$ is diffeomorphic to $S^{4n}$, and so it is null homologous in $p^{-1}(D')\simeq S^{4n-1}$. Using this one can construct a smooth singular cycle $c$ in $V_2(\Co^{2n+1})$ such that every singular simplex in $c$ lies either in $p^{-1}(D')$ or in $p^{-1}(D'')$, and the simplices of the second type form a triangulation of $M\cap D''$ and are transversal to $V_2(\R^{2n+1})$. Let $Dc$ be the Poincar\'e dual cohomology class of $c$. Note that $Dc$ is a generator of $H^{4n-1}(V_2(\Co^{2n+1}),\Z)$, which can be checked by restricting $Dc$ to $p^{-1}(x),x\in S_{\Co}\setminus D'$. Finally, a straightforward check shows that $M$ intersects $V_2(\R^{2n+1})$ transversally at two points
$$((\pm 1,0,\ldots, 0),(0,1,0,\ldots,0))$$
which have the same sign: if the signs were opposite, the map of $H^{4n-1}(-,\Q)$ induced by the inclusion $SO_{2n+1}(\R)\subset U_{2n+1}$ would be zero, which is not the case, see e.g.\ \cite{MT91}*{Theorem 6.7 (2)}. This proves the third statement of the lemma.
\end{proof}

\begin{exmp}\label{so2n+1} Set $G=SO_{2n+1}(\Co)$. We embed $G$ in $GL_{2n+1}(\Co)$ as the stabiliser of the symmetric bilinear form with matrix $\begin{pmatrix} 0& I_n&0\\I_n & 0& 0\\0 & 0 & 1\end{pmatrix}$. We again set $T=G\cap D$, and the Lie algebra $\mathfrak{t}$ of $T$ is $$\{\mathrm{diag}(x_1,\ldots, x_n,-x_1,\ldots,-x_n,0)\mid\mbox{all }x_i\in\Co\}.$$ The lattice $\mathfrak{X}(T)$ is again generated by $\varepsilon_1,\ldots,\varepsilon_n$ and the Weyl group $W$ is the same as in Example~\ref{sp2n}. As above we set $s_i=\sigma_i(\varepsilon_1^2,\ldots,\varepsilon_n^2),i=1,\ldots, n$ and find that these elements generate $H^*(BT,\Q)^W$ and that, modulo decomposable elements, $\sigma_{i}(\varepsilon_1,\ldots,\varepsilon_{2n+1})$ restricted to $\mathfrak{t}$ is 0 if $i$ is odd and $(-1)^{\frac{i}{2}} s_{\frac{i}{2}}$ if $i$ is even.

Using Lemma~\ref{auxso} we see that there exist elements $\mathbf{z}'_{i}\in \fr{i}(SO_{2n+1}(\R),\Z), i=3,7,\ldots,4n-1$ that form a system of free multiplicative generators and such that for $i\equiv 3\mod 4$ the restriction of $\mathbf{x}_i\in H^i(GL_{2n+1}(\Co),\Z)$ to $SO_{2n+1}(\R)$ is $2\mathbf{z}'_i$. (To define the classes $\mathbf{z}'_i$ we proceed by induction on $n$, cf.\ Example~\ref{gln}; the top generator $\mathbf{z'}_{4n-1}$ is set to be the pullback of a generator of $H^{4n-1}(V_2(\R^{2n+1}),\Z)$ under an orbit map.) As in Example~\ref{sp2n} we conclude that the classes $\mathbf{z}_{4i-1}=\frac{1}{2}\bar\gamma(s_i)\in H^*(G,\Q),i=1,\ldots,n$ are in fact integral and form a system of free multiplicative generators of $\fr{*}(SO_{2n+1}(\Co),\Z)$.
\end{exmp}

\begin{exmp}\label{spin2n+1} The case of $G=Spin_{2n+1}(\Co)$ follows from the previous one with minor modifications. Let $T,\mathfrak{t}, W$ and $s_i,i=1,\ldots,n$ be the same as in Example~\ref{so2n+1}. Let $T'\subset G$ be the torus with Lie algebra $\mathfrak{t}$. The group $W$ acts on $H^*(BT',\Q)$ and the ring $H^*(BT',\Q)^W$ is again generated by $s_1,\ldots, s_n$. Note however that the character group $\mathfrak{X}(T')\subset \mathfrak{t}^*$ is strictly larger than $\mathfrak{X}(T)$: it is the set of all $x\in\mathfrak{t}^*$ such that the coordinates of $x$ in the basis $\{\varepsilon_1,\ldots,\varepsilon_n\}$ are either all integer or all half-integer; this group contains $\mathfrak{X}(T)$ as a subgroup of index 2. 

Let $p:Spin_{2n+1}(\Co)\to SO_{2n+1}(\Co)$ be the covering map. The map
$$p^*:\fr{*}(SO_{2n+1}(\Co),\Z)\to \fr{*}(Spin_{2n+1}(\Co),\Z)$$ takes all classes $\mathbf{z}_i, i=3,7,\ldots, 4n-1$ but one to free multiplicative generators, and the remaining class to twice a free multiplicative generator. 
%
It follows from \cite{Pit91}*{p.~122 and Section 7.4} that this class is in fact $\mathbf{z}_{4t-1}$ where $t$ is the highest power of 2 that is $\leq n$.
So we get a system of free generators of $\fr{*}(Spin_{2n+1}(\Co),\Z)$ by taking $\frac{1}{4}\bar\gamma(s_t)$ and $\frac{1}{2} \bar\gamma(s_i), i=1,\ldots, n,i\neq t$.
\end{exmp}

\begin{exmp}\label{so2n} Set $G=SO_{2n}(\Co),n>1$. We embed $G$ in $GL_{2n}(\Co)$ as the stabiliser of the symmetric bilinear form with matrix $\begin{pmatrix} 0& I_n\\I_n & 0\end{pmatrix}$. Let $T,\mathfrak{t},s_i,i=1,\ldots,n$ be as in Example~\ref{sp2n}. The Weyl group this time will be the group of permutations and sign changes of $\{\varepsilon_1,\ldots,\varepsilon_n\}$ that only involve an even number of signs. So the ring $H^*(BT,\Q)^W$ will be generated by $s_1,\ldots,s_{n-1}$ and the Pfaffian $s=\prod \varepsilon_i$; note that $s_n=s^2$.

Let us calculate $\bar\gamma(s)$. Let $\mathbf{z}'\in H^{2n-1}(SO_{2n}(\R),\Z)$ be a generator of the group $o^*(H^{2n-1}(S^{2n-1},\Z))$ from Lemma~\ref{auxso}. We claim that $\bar\gamma(s)=\pm\mathbf{z}'$. To prove this notice that firstly, there is a commutative diagram
$$
\begin{tikzcd}
ESO_{2n}(\R)\arrow[r]\arrow[d] &  S(\gamma^{2n}) \arrow[d]\\
BSO_{2n}(\R)\arrow[r,"="] & BSO_{2n}(\R)
\end{tikzcd}
$$
of fibre bundles, where $S(\gamma^{2n})$ denotes the spherisation of the tautological (real) oriented vector bundle $\gamma^{2n}$ over $BSO_{2n}(\R)$, and secondly, the pullback of $\gamma^{2n}$ to $BT$ splits as a sum of $n$ oriented bundles of rank 2 and has Euler class $\prod \varepsilon_i=s$. From these observations it follows that the element $\mathbf{z}'\in E^{0,2n-1}_2$ of the Leray-Serre spectral sequence of $ESO_{2n}(\R)\to BSO_{2n}(\R)$ is transgressive; moreover, by Proposition~\ref{transgression} it transgresses to $\pm e(\gamma^{2n})\in H^{2n}(BSO_{2n},\Z)$, which pulls back to $\pm s\in H^{2n}(BT,\Z)$ under $BT\to BSO_{2n}(\R)$. This implies that $\bar\gamma(s)=\pm\mathbf{z}'$, see Remark~\ref{gamma and transgression}.

As in Example \ref{so2n+1} we deduce using Lemma~\ref{auxso} that there exist elements $\mathbf{z}_i'\in \fr{i}(SO_{2n}(\R),\Z), i=3,7,\ldots,4n-5$ that together with $\mathbf{z}'$ form a system of free multiplicative generators of $\fr{*}(SO_{2n}(\Co),\Z)$. For $i\equiv 3\mod 4$ the restriction of $\mathbf{x}_i\in H^i(GL_{2n}(\Co),\Z)$ to $SO_{2n}(\R)$ is $2\mathbf{z}_i'$. We can now conclude as above in Example~\ref{so2n+1} that the elements $\mathbf{z}_{4i-1}=\frac{1}{2}\bar\gamma(s_i)\in H^*(G,\Q),i=1,\ldots,n-1$ and $\mathbf{z}=\bar\gamma(s)$ are integral and form a system of free multiplicative generators of $\fr{*}(SO_{2n}(\Co),\Z)$.
\end{exmp}

\begin{exmp} The case $G=Spin_{2n}(\Co),n>1$ is very similar to Examples~\ref{spin2n+1} and \ref{so2n}, so we will just state the answer. Let $T,\mathfrak{t},W,s_i,i=1,\ldots,n$ and $s$ be as in Example~\ref{so2n}. Take $T'\subset G$ to be the torus with Lie algebra $\mathfrak{t}$. As in the previous example, the group $W$ acts on $H^*(BT',\Q)$ and the ring $H^*(BT',\Q)^W$ is generated by $s_1,\ldots, s_{n-1},s$. 
Moreover, let $t$ be the highest power of 2 that is $\leq n-1$. The elements $\bar\gamma(s), \frac{1}{4}\bar\gamma(s_t)$ and $\frac{1}{2}\bar\gamma(s_i), i=1,\ldots, n-1, i\neq t$ are integral and form a system of free multiplicative generators of $\fr{*}(Spin_{2n}(\Co),\Z)$.
\end{exmp}

\begin{exmp}\label{g2} Set $G$ to be the (unique) complex simple group of type $G_2$. Below we deduce information about~$G$ via an embedding $G\subset Spin_7(\Co)$, so we let $T,\mathfrak{t},W,\varepsilon_i, s_i$ and $\mathbf{z}_{4i-1},i=1,2,3$ be as in Example~\ref{so2n+1} for $n=3$. Let $K\subset G$ be a maximal compact subgroup. Then $K$ is a Lie subgroup of $Spin_7(\R)$, all such subgroups are conjugate and the quotient space $Spin_7(\R)/K\cong S^7$, see e.g.\ \cite{Var01}*{Section 2, Theorem 3}. It follows from the fact that $H^5(K,\Z)=0$ (see e.g.\ \cite{AK16}*{Theorem 2.18}) that all differentials of the Leray-Serre spectral sequence of the fibre bundle $K\subset Spin_7(\R)\to S^7$ that originate in $E^{0,11}_{\geq 2}$ are zero, so the inclusion $K\subset Spin_7(\R)$ induces a surjective map of the integral cohomology groups modulo torsion, and by Example~\ref{spin2n+1} so does the composite map $i:K\to Spin_7(\R) \to SO_7(\R)$, so $\mathbf{w}_3=i^*(\mathbf{z}_3)$ and $\mathbf{w}_{11}=i^*(\mathbf{z}_{11})$ are free multiplicative generators of $\fr{*}(K,\Z)$.

Let $\pi\colon Spin_7(\R)\to GL_8(\R)$ be the unique irreducible representation of dimension 8, and let $\pi_{\Co}: Spin_7(\Co)\to GL_8(\Co)$ be its complexification, which is also irreducible (see \cite{Var01}) and has weights $t_1\varepsilon_1+t_2\varepsilon_2+t_3\varepsilon_3$ where $t_1,t_2,t_3\in\left\{\pm\frac{1}{2}\right\}$. The group $K\subset Spin_7(\R)$ is the stabiliser of a non-zero element under $\pi$, so $G\subset Spin_7(\Co)$ is the stabiliser of an $x\in\Co^8\setminus\{0\}$ under $\pi_{\Co}$. By conjugating $G\subset Spin_7(\Co)$ if necessary we may assume that a maximal torus $T_G$ of $G$ is $\subset T$, so $x$ is a weight vector. The Weyl group $W$ acts transitively on the weights of $\pi_{\Co}$, so we may further assume that the weight that corresponds to $x$ is $\frac{1}{2}(\varepsilon_1+\varepsilon_2+\varepsilon_3)$. The Lie algebra $\mathfrak{t}_G\subset \mathfrak{t}$ of $T_G$ is then given by $\sum\varepsilon_i=0$ and we have $$H^*(BT_G,\Q)\cong H^*(BT,\Q)/(\varepsilon_1+\varepsilon_2+\varepsilon_3)=\Q[\varepsilon_1,\varepsilon_2,\varepsilon_3]/(\varepsilon_1+\varepsilon_2+\varepsilon_3);$$ note that the long roots of $G$ are the restrictions of the roots $\varepsilon_i-\varepsilon_j$ of $SO_7(\Co)$, while the short roots are the restrictions of $\pm(\varepsilon_i+\varepsilon_j)$ and $\pm\varepsilon_i$; here $i,j\in\{1,2,3\},i\neq j$. Set $s_{1,G_2}$, respectively $s_{3,G_2}$ to be the image of $s_1=\sum\varepsilon_i^2$, respectively $s_3=\prod\varepsilon_i^2$ in $H^*(BT_G,\Q)$. The elements $\frac{1}{2}\bar\gamma (s_{1,G_2})=\mathbf{w}_{3}$ and $\frac{1}{2}\bar\gamma (s_{3,G_2})=\mathbf{w}_{11}$ are free multiplicative generators of $\fr{*}(G,\Z)$.
\end{exmp}

\subsubsection{Expressing $e_G(E')$ in terms of the generators of $I_1(X,\Q)$} Let $X=G/B$ and $E'\to X$ be as above in the beginning of Section~\ref{genalg}. Suppose that we have a system $s_1,\ldots, s_n$ of homogeneous generators of $I_1(X,\Q)$. Since $e(E')=0$, the equivariant Euler class $e_G(E')\in I_1(X,\Q)$ by Propositions~\ref{seceulerse} and~\ref{s1sproj}, so we can write it as 
\begin{equation}\label{linsyst}
e_G(E')=\sum p_i s_i
\end{equation}
where each $p_i\in H^*(BT,\Q)$ is homogeneous of degree $\rk E'-\deg s_i$. Note that \ref{linsyst} is a finite-dimensional linear system of equations in $p_1,\ldots, p_n$. 
%

In practice however, instead of solving \ref{linsyst} directly it is more convenient to use Gr\"obner bases for the ideals generated by the Weyl invariants. In \cite{MS03} T.~Mora and M.~Sala construct a system of homogeneous generators for the ideal $I\subset \Q[t_1,\ldots,t_n]$ generated by the elementary symmetric polynomials and prove that it is a Gr\"obner basis with respect to any monomial ordering $<$ such that $t_1<t_2<\cdots <t_n$. This takes care of the case $G=GL_n(\Co)$, and a similar construction can be given for any other reductive group $G$. The details will be presented elsewhere.

\subsubsection{The map $\alpha^*$ and calculating the answer.}\label{the map alpha} Here we assume that that we have homogeneous generators $s_1,\ldots, s_n$ of the ideal $I_1(X,\Q)$ that have been chosen so that $\bar\gamma(s_1),\ldots, \bar\gamma(s_n)$ is a system of free multiplicative generators of $\fr{*}(G,\Z)$. Writing $e_G(E')=\sum p_i s_i$ with $p_i$ as in~\ref{linsyst}, and applying \ref{formulas1} we get
$$S(e_G(E'))=S_1(e_G(E'))=\sum \bar\gamma(s_i)\times \alpha^*(p_i)\in \fr{*}(G\times X,\Z)\subset H^*(G\times X,\Q).$$

So if $y\in H_*(X,\Z)$, then
\begin{equation}\label{explicits1gb}
S(e_G(E'))/y=\sum \langle \alpha^*(p_i),y\rangle\bar\gamma(s_i)\in \fr{*}(G,\Z)\subset H^*(G,\Q).
\end{equation}
The left hand side of \ref{explicits1gb} is integral for every $y\in H_*(X,\Z)$, so it follows from our assumption on $s_i$'s that all classes $\alpha^*(p_i)$ are integral too.

For our applications we will need to determine whether the ring $\Lambda^*_{E'}$ (see Notation \ref{lambdae}) has finite index in $\fr{*}(G,\Z)$, and if it does, to calculate the index of $\Lambda^k_{E'}$ in $H^k(G,\Z)$ where $k$ is the top degree in which the cohomology of $G$ is non-zero, see Theorem \ref{mainthmappl}. In other words, we have to determine the image of the homomorphism
\begin{equation}\label{anothers}
y\mapsto S(e_G(E'))/y.
\end{equation}
Recall that all integral elements in $\im\bar\gamma\subset H^*(G,\Q)$ are primitive (Remark~\ref{imagebargammaprimitive}) and note that both the homology and cohomology of $G/B$ are torsion free (\cite{BGG}). Therefore we can consider \ref{anothers} as a homomorphism acting between the free graded abelian groups $H_{*}(G/B,\Z)$ and $P^*\subset \fr{*}(G)$; note also that it maps $H_{2\rk(E')-2l}$ to $P^{2l-1}$ for every fixed $l\geq 1$. So acting degreewise, we set $m_l=\rk P^{2l-1}$ and let $r_l$, respectively $i_l$ be the rank, respectively the order of the cokernel of~\ref{anothers} restricted to $H_{2\rk(E')-2l}$, cf.\ Corollary~\ref{stabilizerisproduct}. 

\begin{prop}
The ring $\Lambda^*_{E'}$ has finite index in $\fr{*}(G,\Z)$ if and only if all $i_l$ are finite, in which case $\prod_l i_l$ is the index of $\Lambda^k_{E'}$ in $H^k(G,\Z)$. \qed
\end{prop}

In order to find the numbers $r_l$ and $i_l$ one can choose bases for the source and for the target and write down the corresponding matrix for \ref{anothers}. 
Bases for for $P^*$ for classical Lie groups and $G_2$ were described in Section~\ref{genidealbargamma}. In~\cite{BGG}, I.~Bernstein, I.~Gelfand and S.~Gelfand give a basis for $H_*(G/B,\Z)$ and an integration formula that allows one to evaluate $\alpha^*(f), f\in H^*(BT,\Z)$ on the elements of the basis.

In more detail, let us assume that $G$ is semi-simple and simply connected (which we are allowed to do by Remark~\ref{sufficessimplyconn}), and suppose we have chosen a system $\Pi$ of simple roots of $G$. For $\gamma$ a root of $G$ let $\sigma_\gamma$ denote the reflection in the orthogonal complement of $\gamma$. There is a $\Z$-basis of $H_*(G/B,\Z)$ the elements of which are in bijection with the elements of the Weyl group $W$ of $G$. An element $w\in W$ corresponds to a basis element $e_w\in H_{2\ell(w)}(G/B,\Z)$ where $\ell(w)$ is the length of $w$ with respect to the generators $\sigma_\gamma,\gamma\in \Pi$. (In \cite{BGG} $e_w$ is denoted $s_w$; we use the notation $e_w$ to avoid confusion with the generators $s_1,\ldots, s_n$ of $H^*(BG,\Q)$.) Theorems 3.4 and 4.1 of \cite{BGG} allow one to evaluate $\alpha^*(f)$ for a given $f\in\mathrm{Sym}^*(\mathfrak{X}(T))\otimes\Q=H^*(BT,\Q)$ on $e_w$ once $w$ has been written as a product of $\ell(w)$ reflections $\sigma_\gamma,\gamma\in \Pi$. 
We do not reproduce this procedure here, but we give explicit examples in Section~\ref{completeflagvarieties}.

Returning to the question of calculating $i_l$ and $r_l$, let $t_1,\ldots t_{m_l}\in \{s_1,\ldots, s_n \}$ be the elements of degree $2l$. Let $q_1,\ldots, q_{m_l}$ be the coefficients of $t_1,\ldots, t_{m_l}$ in $\sum p_i s_i$. To calculate $r_l$ and $i_l$ for a given $l$ we apply the integration formula to $f=q_j, j=1,\ldots, m_l$ and $w$ an arbitrary element of length $\rk E'-l$. Since we wish in fact to calculate $r_l$ and $i_l$ for all $l\geq 1$, we need to express every $w\in W$ of length $<\rk E'$ as a product of $\ell(w)$ reflections $\sigma_\gamma,\gamma\in \Pi$. This can be done in a straightforward way by induction on $\ell(w)$ using the fact that the word problem for $W$ is easy to solve e.g.\ by applying an explicit ``small'' faithful permutation action of $W$ (see e.g.\ \cite{BB05}*{Chapter 8} for details for $W$ of type $B,C$ and $D$).

%

\subsubsection{Arbitrary $P$.}\label{arbp} Now let $P$ be an arbitrary parabolic subgroup of $G$. Let $B\subset P$ be a Borel subgroup and let $p:G/B\to G/P$ be the projection. Using \ref{functorialityofs} and the functoriality of the $/$-product we get $$(S(e_G (E'))/p_*(y)=\big(S(e_G(p^*(E'))\big)\bigm/y$$ for every $y\in H_*(G/B,\Z)$. Note that $p_*(H_*(G/B,\Z))=H_*(G/P,\Z)$ by \cite{BGG}*{Corollary 5.4}, which implies $\Lambda^*_{p^*(E')}=\Lambda^*_{E'}$. So the case of arbitrary $P$ reduces to the case $P=B$ described above. Note that in the case $P\neq B$ our task is simplified somewhat by the fact that we don't need to take into account all basis elements $e_w\in H_*(G/B,\Z)$: I.~Bernstein, I.~Gelfand and S.~Gelfand define a subset $W^1_\Theta$ of $W$ (\cite{BGG}*{Proposition 5.1}) and prove that $p_*(e_w)=0$ for $w\not\in W^1_\Theta$ and that the classes $p_*(e_w),w\in W^1_\Theta$ form a $\Z$-basis of $H_*(G/P,\Z)$ (\cite{BGG}*{Corollary 5.3}).

Finally, we note that for ``small' $G/P$ it is possible to calculate the map $\alpha^*$ directly, i.e.\ without the integration formula. In the next subsection we do this for projective spaces and quadrics.

\subsection{Projective spaces and quadrics.}\label{projquad}

In the rest of this section we will write $\Lk(y),y\in H_*(X,\Z)$ instead of $\Lk_V^{\Sing}(y)$ (see Definition~\ref{linkingclasssing}) because in all examples below $V$ will be $\Gamma(X,E)$ for some algebraic vector bundle $E$ over $X$.

\begin{ntt}\label{mdni} For $d,n,i\in \Z$ set $\mathbf{m}(d,n,i)=(d-1)^{n+1}+(-1)^{i+1}(d-1)^{n+1-i}$. Note that if $n,i$ are fixed, $i$ is even and $i,n+1-i\geq 0$, then $\dfrac{\mathbf{m}(d,n,i)}{d-2}$ is an integral polynomial in $d$. We denote this polynomial by $\mathbf{m}'(d,n,i)$. We also note that $\mathbf{m}(d,n,i)\neq 0$ for $d\not\in\{0,1,2\}$ and $\mathbf{m}'(d,n,i)\neq 0$ for $d\not\in\{0,1\}$.
\end{ntt}

\subsubsection{Projective spaces}\label{explprojspace}

Set $X=\PP^n(\Co), E=\mathcal{O}(d), G=GL_{n+1}(\Co)$, and let $P$ be the stabiliser of a line $l\in X$. 
%
Using the notation of Example \ref{gln} we have $H^*(BG,\Q)\cong\Q[s_1,\ldots, s_{n+1}]$ where $s_i$ is the $i$-th elementary symmetric polynomial in $\varepsilon_1,\ldots,\varepsilon_{n+1}\in H^2(BT,\Q)$. Let us take $l$ to be the line spanned by $(1,0,\ldots, 0)$. The ring $H^*(BP,\Q)$ then also becomes a subring of $H^*(BT,\Q)$: we have $H^*(BP,\Q)\cong\Q[b,b_1,\ldots,b_n]$ where $b=\varepsilon_1$ and $b_i=\sigma_i(\varepsilon_2,\ldots, \varepsilon_{n+1})$. Note that under these identifications the map $\beta^*:H^*(BG,\Q)\to H^*(BP,\Q)$ becomes the inclusion, so we have $\beta^*(s_i)=b_i+b b_{i-1}$ where we set for convenience $b_0=1$ and $b_{n+1}=0$. The ideal $I_1(X,\Q)$ is generated by $a_i=b_i + bb_{i-1},i=1,\ldots, n+1$.

%
%
Set $E'=J(E)$. Using Example \ref{tautlinebund} we see that the weights of the representations of $P$ that correspond to $\Oh(d)$, respectively $\Omega_{\PP^n(\Co)}$ are $-d\e_1$, respectively $\e_1-\e_i,i\in\{2,\ldots, n+1\}$.
By exact sequence \ref{sesjetbundle} the weights of the representation of $P$ that corresponds to $E'$ are \begin{equation}\label{weightspn}
-d\e_1,(1-d)\e_1-\e_i,i\in\{2,\ldots,n+1\}.
\end{equation}
Multiplying these together and using Lemma \ref{inducedvectbund1} we get
\begin{multline}\label{explpn}
c_{n+1}(E'_{hG})=\prod_{i=1}^{n+1}((1-d)\e_1-\e_i)
=(-(d-1)b)^{n+1}+\sum_{i=1}^{n+1}(-1)^i(b_i + bb_{i-1})(-(d-1)b)^{n+1-i}\\
=(-1)^{n+1}\left(((d-1)b)^{n+1}+\sum_{i=1}^{n+1}a_i((d-1)b)^{n+1-i}\right)\in H^*_G(X,\Q)\cong H^*(BP,\Q)\subset H^*(BT,\Q).
\end{multline}

A straightforward check shows that
\begin{equation}\label{toppowerbgln}
b^{n+1}=\sum_{i=0}^n (-1)^i(b_{i+1}+b_i b)b^{n-i}=\sum_{i=1}^{n+1} (-1)^{i+1} a_i b^{n-i+1}.
\end{equation}
Substituting the right hand side in \ref{explpn} we get
$$(-1)^{n+1}c_{n+1}(E'_{hG})=\sum_{i=1}^{n+1} a_i b^{n-i+1}((d-1)^{n+1} (-1)^{i+1}+(d-1)^{n-i+1})=\sum_{i=1}^{n+1}(-1)^{i+1} a_i b^{n-i+1} \mathbf{m}(d,n,i).$$

Let us now calculate $\alpha^*(b)$. This is straightforward: the vector bundle over $X=G/P$ constructed using $b=\varepsilon_1$ is $\Oh(-1)$, so applying Lemma \ref{inducedvectbund1} and setting $c=c_1(\Oh(1))\in H^2(X,\Q)$ we get 
\begin{equation}\label{alphab}
\alpha^*(b)=-c_1(\mathcal{O}(1))=-c.
\end{equation}

Taking into account formulas \ref{alphab} and \ref{formulas1}, Example \ref{gln} and the fact that $a_i=\beta^*(s_i)$ we get
$$S(e_G(E'))=-\sum_{i=1}^{n+1}\mathbf{m}(d,n,i)\mathbf{x}_{2i-1} \times c^{n-i+1},$$
so using Corollary~\ref{maincorollary2} we obtain the following result.

\begin{prop}\label{propexplprojspace}
We have
\begin{equation}
O^*(\Lk([\PP^{n-i+1}(\Co)]))=S(e_G(E'))/[\mathbb{P}^{n-i+1}(\Co)]=-\mathbf{m}(d,n,i)\mathbf{x}_{2i-1}, i=1,\ldots,n+1
\end{equation}
where $\mathbf{x}_{2i-1}$ are the free multiplicative generators of $H^*(GL_{n+1}(\Co),\Z)$ from Example \ref{gln} and the coefficients $\mathbf{m}(d,n,i)$ are defined in Notation \ref{mdni}. \qed
\end{prop}

\subsubsection{Odd-dimensional quadrics.}\label{oddquad} Let $G,T,W$, and $\varepsilon_1,\ldots,\varepsilon_n\in \mathfrak{X}(T)$ be as in Example \ref{so2n+1}. So $G\cong SO_{2n+1}(\Co)$ preserves a non-degenerate quadratic form on $\Co^{2n+1}$, and we let $X\subset\PP^{2n}(\Co)$ be the zero locus of this form. We take $E$ to be the restriction of $\mathcal{O}(d)$ to $X$ and $P\subset G$ equal the stabiliser of $(1:0:\cdots :0)\in X$.
The Weyl group $W_P\subset W$ of $P$ is the stabiliser of $\varepsilon_1$. We identify $H^*(B G,\Q)$ with $\Q[s_1,\ldots, s_n]$, where $s_i$ is the $i$-th elementary symmetric polynomial in $\varepsilon^2_1,\ldots,\varepsilon_n^2$, and $H^*(BP,\Q)$ with $\Q[b, b_1,\ldots, b_{n-1}]$ where $b=\varepsilon_1$ and $b_i, i=1,\ldots, n-1$ is $\sigma_i(\varepsilon_2^2,\ldots,\varepsilon_n^2)$. For $i=1,\ldots, n$ we have $\beta^*(s_i)= b_i+ b^2 b_{i-1}$ where $b_0=1,b_n=0$. The ideal $I_1(X,\Q)$ is generated by $a_i=b_i+ b^2 b_{i-1},i=1,\ldots, n$.

We now proceed as in the case of the projective space. The bundle $E'=J(E)$ is $G$-equivariant, and we wish to find the weights of the representation of $P$ that induces $E'$. We have an exact sequence
\begin{equation}\label{exactseqjets}
0\to \Oh(d-2)|_X\to J(\Oh(d))|_X\to E'=J(\Oh(d)|_X)\to 0.
\end{equation}
(Note that this sequence exists for even-dimensional quadrics as well.) Restricting the weights we found above for $J(\Oh(d))$ (see formula \ref{weightspn}) to $\mathfrak{X}(T)$ and excluding the weight $(2-d)\e_1$, which corresponds to $\Oh(d-2)|_X$, we get
$$-d\e_1, (1-d)\e_1, (1-d)\e_1\pm\e_i,i\in\{2,\ldots, n\}.$$
Multiplying the weights we get by Lemma \ref{inducedvectbund1} (for $d\neq 2$)
\begin{multline*}
e(E'_{hG})=(-\varepsilon_1-(d-1)\varepsilon_1)(-(d-1)\varepsilon_1)\prod_{i=2}^n ((d-1)^2\varepsilon_1^2-\varepsilon_i^2)=\frac{d-1}{d-2}\prod_{i=1}^n ((d-1)^2\varepsilon_1^2-\varepsilon_i^2)\\ =\frac{d-1}{d-2}\left((d-1)^{2n} \varepsilon_1^{2n}+\sum_{i=1}^n(-1)^i (d-1)^{2n-2i} \varepsilon_1^{2n-2i}s_i\right)\\ =\frac{d-1}{d-2}\left((d-1)^{2n}b^{2n}+\sum_{i=1}^n(-1)^i (d-1)^{2n-2i} b^{2n-2i}a_i\right).
\end{multline*}

We have (cf.\ \ref{toppowerbgln})
\begin{equation}\label{toppowerbso2n+1}
b^{2n}=\sum_{i=1}^{n} (-1)^{i+1}(b_{i}+b_{i-1} b^2)b^{2(n-i)}=\sum_{i=1}^{n} (-1)^{i+1} a_i b^{2(n-i)},
\end{equation}
so
\begin{multline*}
e(E'_{hG})=\frac{d-1}{d-2}\sum_{i=1}^n (-1)^{i+1}((d-1)^{2n}-(d-1)^{2n-2i})a_i b^{2n-2i}=\sum_{i=1}^n (-1)^{i+1}\frac{\mathbf{m}(d,2n,2i)}{d-2} a_i b^{2n-2i}=\\
\sum_{i=1}^n (-1)^{i+1}\mathbf{m}'(d,2n,2i) a_i b^{2n-2i}.
\end{multline*}
Note that the formula $e(E'_{hG})$ we have just obtained is valid for $d=2$ as well.

As in Section \ref{explprojspace} we find that $\alpha^*(b)=-c$ where $c$ is $c_1(\mathcal{O}(1))$ restricted to $X$. Recall that $H^i(X,\Z)$ and $H_i(X,\Z)$ are $\cong \Z$ if $i\in\{0,2,\ldots, 4n-2\}$ and are 0 otherwise. The class $c^i$ is a generator of $H^{2i}(X,\Z)$ if $i \leq n-1$ and twice a generator if $n\leq i\leq 2n-1$: this follows from the integral Poincar\'e duality and the fact that $c^{2n-1}([X])=2$. Applying formula \ref{formulas1}, Corollary~\ref{maincorollary2} and Example \ref{so2n+1} we get the following proposition.

\begin{prop}\label{propexploddquadric}
There exist additive generators $Z_j$ of $H_{2j}(X,\Z),j=0,\ldots, 2n-1$ such that modulo torsion
$$O^*(\Lk(Z_j))=S(e_G(E'))/Z_j=\left\{
\begin{array}{ll}
0 & \mbox{if $j$ is odd,}\\
2\mathbf{m}'(d,2n,2i)\mathbf{z}_{4i-1} & \mbox{if $j=2n-2i$ and $0\leq j\leq n-1$,}\\
4\mathbf{m}'(d,2n,2i)\mathbf{z}_{4i-1} & \mbox{if $j=2n-2i$ and $n\leq j\leq 2n-2$}
\end{array}
\right.
$$
where $\mathbf{z}_{4i-1}$ are the free multiplicative generators of $\fr{*}(SO_{2n+1}(\Co),\Z)$ introduced in Example \ref{so2n+1}, and $\mathbf{m}'(d,2n,2i)$ are the coefficients defined in Notation \ref{mdni}.\qed
\end{prop}

\subsubsection{Even-dimensional quadrics.}\label{evenquad} Let $G,T,W$, and $\varepsilon_1,\ldots,\varepsilon_n\in \mathfrak{X}(T)$ be as in Example \ref{so2n}. We take $X\subset\mathbb{P}^{2n-1}(\Co)$ to be the zero locus of the quadratic form on $\Co^{2n}$ preserved by $G$. The line bundle $E$ and the subgroup $P\subset G$ are defined as in the case of odd-dimensional quadrics above. The Weyl group $W_P$ of $P$ is again the stabiliser of $\varepsilon_1$, so setting $s_i=\sigma_i(\varepsilon_1^2,\ldots,\varepsilon_n^2), s=\prod_{i=1}^n \varepsilon_i, b=\varepsilon_1, b_i=\sigma_i(\varepsilon_2^2,\ldots,\varepsilon_n^2),s'=\prod_{i=2}^n \varepsilon_i$ we have
$$H^*(BG,\Q)\cong\Q[s_1,\ldots, s_n,s], H^*(BP,\Q)\cong\Q[b, b_1,\ldots, b_n,s']$$ with the map $\beta^*:H^*(BG,\Q)\to H^*(BP,\Q)$ being the inclusion. The ideal $I_1(X,\Q)$ is generated by $a_i=b_i+ b^2 b_{i-1},i=1,\ldots, n-1$ and $a=bs'$ (we again set $b_0=1$).

As in Section \ref{oddquad} we set $E'=J(E)$ and find that  
the weights of the corresponding representation of $P$ are
$$-d\e_1,(1-d)\e_1\pm \e_i,i\in\{2,\ldots,n\},$$ so 
we get using Lemma \ref{inducedvectbund1} and temporarily assuming $d\neq 2$
\begin{multline*}
e(E'_{hG})=(-\varepsilon_1-(d-1)\varepsilon_1)\prod_{i=2}^n ((d-1)^2\varepsilon_1^2-\varepsilon_i^2)=\frac{-1}{\varepsilon_1(d-2)}\prod_{i=1}^n ((d-1)^2\varepsilon_1^2-\varepsilon_i^2)\\ =\frac{-1}{\varepsilon_1(d-2)}\left((d-1)^{2n} \varepsilon_1^{2n}+\sum_{i=1}^n(-1)^i (d-1)^{2n-2i} \varepsilon_1^{2n-2i}s_i\right)\\ 
=\frac{-1}{d-2}\left((d-1)^{2n}b^{2n-1}+\sum_{i=1}^{n-1}(-1)^i (d-1)^{2n-2i} b^{2n-2i-1}a_i+ (-1)^n as'\right).
\end{multline*}

We have (cf.\ \ref{toppowerbgln} and \ref{toppowerbso2n+1})
\begin{equation*}
b^{2n-1}=\sum_{i=1}^{n-1} (-1)^{i+1}(b_{i}+b_{i-1} b^2)b^{2(n-i)-1}+(-1)^{n+1}b_{n-1}b=\sum_{i=1}^{n-1} (-1)^{i+1} a_i b^{2(n-i)-1}+(-1)^{n+1} as',
\end{equation*}
which implies
\begin{multline*}
e(E'_{hG})=\frac{-1}{d-2}\left(\sum_{i=1}^{n-1} (-1)^{i+1}((d-1)^{2n}-(d-1)^{2n-2i})a_i b^{2n-2i-1}+(-1)^{n+1}((d-1)^{2n}-1) as'\right)\\ =\frac{1}{d-2}\left(\sum_{i=1}^{n-1} (-1)^{i}\mathbf{m}(d,2n-1,2i) a_i b^{2n-2i-1}+(-1)^{n}\mathbf{m}(d,2n-1,2n)as'\right) \\ =\sum_{i=1}^{n-1} (-1)^{i}\mathbf{m}'(d,2n-1,2i) a_i b^{2n-2i-1}+(-1)^{n}\mathbf{m}'(d,2n-1,2n)as'.
\end{multline*}
The final expression for $e(E'_{hG})$ is valid for $d=2$ as well.

We have
$$H^i(X,\Z)\cong H_i(X,\Z)\cong\left\{\begin{array}{ll} \mathbb{Z}& \mbox{if $i\in\{0,2,\ldots, 4n-4\}$ but $i\neq 2n-2$,}\\ \Z^2 &\mbox{if $i=2n-2$,}\\  0 & \mbox{otherwise.}\end{array}\right.
$$
Let $c$ be $c_1(\mathcal{O}(1))$ restricted to $X$. As in Section \ref{oddquad} we conclude that $\alpha^*(b)=-c$ and the class $c^i$ is a generator of $H^{2i}(X,\Z)$ if $0\leq i<n-1$ and twice a generator if $n-1<i\leq 2n-2$.

\begin{lmm}
We identify $H^{4n-4}(X,\Z)\cong \Z$ using the fundamental class. Set $\alpha_1=c^{n-1}$ and $\alpha_2= \alpha^*(s')$. We have then $\alpha_1^2=2, \alpha_1\alpha_2=0$ and $\alpha_2^2=(-1)^{n-1}\cdot 2$. Moreover, there is a basis $\{\Lambda_1,\Lambda_2\}$ of $H^{2n-2}(X,\Z)$ such that $\alpha_1=\Lambda_1+\Lambda_2$ and $\alpha_2=\Lambda_1-\Lambda_2$.
\end{lmm}

\begin{proof} We have $\alpha_1^2=c^{2n-2}=2$, as $c=c_1(\mathcal{O}(1))$, and $\alpha_1\alpha_2=(-1)^{n-1} \alpha^*(b s')=(-1)^{n-1}\alpha^*(a)=0$. Moreover, $\alpha_2^2=\alpha^*(b_{n-1})$. Using $a_i=b_i+b^2 b_{i-1}$ and $\alpha^*(a_i)=0$ we deduce by induction that $\alpha(b_i)=(-1)^i \alpha^*(b^{2i})=(-1)^i c^{2i}$, which implies $\alpha_2^2=(-1)^{n-1}\cdot 2$.

%

The intersection form of $X$ represents 2. Using the classification of integral unimodular quadratic forms of small rank (see e.g.\ \cite{mh}*{Theorem 2.2, Chapter 2}) we conclude that there is a basis of $H^{2n-2}(X,\Z)$ in which the intersection form has matrix $\begin{pmatrix} 1& 0\\0& 1\end{pmatrix}$ if $n$ is odd and $\begin{pmatrix} 0& 1\\1& 0\end{pmatrix}$ if $n$ is even. Multiplying some or all elements of this basis by $-1$ if necessary we get the required basis $\{\Lambda_1,\Lambda_2\}$ of $H^{2n-2}(X,\Z)$.
\end{proof}

We now state an analogue of Propositions \ref{propexplprojspace} and \ref{propexploddquadric}, which again follows by applying formula \ref{formulas1}, Corollary~\ref{maincorollary2} and Example \ref{so2n}.
\begin{prop}
There exist generators $Z_j$ of $H_{2j}(X,\Z),\allowbreak j=0, \ldots, 2n-2, j\neq n-1$ and generators $W_1, W_2$ of $H_{2n-2}(X,\Z)$ such that modulo torsion
$$O^*(\Lk(Z_j))=S(e_G(E'))/Z_j=\left\{
\begin{array}{ll}
0 & \mbox{if $j$ is even,}\\
{2} \mathbf{m}'(d,2n-1,2i)\mathbf{z}_{4i-1} & \mbox{if $j=2n-2i-1$ and $1\leq j< n-1$,}\\
{4} \mathbf{m}'(d,2n-1,2i)\mathbf{z}_{4i-1} & \mbox{if $j=2n-2i-1$ and $n-1< j\leq 2n-3$;}
\end{array}
\right.
$$
$$O^*(\Lk(W_1))=S(e_G(E'))/W_1=\left\{
\begin{array}{ll}
\mathbf{m}'(d,2n-1,2n)\mathbf{z}+{2}\mathbf{m}'(d,2n-1,2i)\mathbf{z}_{4i-1} & \mbox{if $n-1=2n-2i-1$,}\\
\mathbf{m}'(d,2n-1,2n)\mathbf{z} & \mbox{if $n$ is odd;}\\
\end{array}
\right.
$$
$$O^*(\Lk(W_2))=S(e_G(E'))/W_2=\left\{
\begin{array}{ll}
-\mathbf{m}'(d,2n-1,2n)\mathbf{z}+{2}\mathbf{m}'(d,2n-1,2i)\mathbf{z}_{4i-1} & \mbox{if $n-1=2n-2i-1$,}\\
-\mathbf{m}'(d,2n-1,2n)\mathbf{z} & \mbox{if $n$ is odd.}\\
\end{array}
\right.
$$
Here $\mathbf{z}_{4i-1}$ and $\mathbf{z}$ are the free multiplicative generators of $\fr{*}(SO_{2n}(\Co),\Z)$ introduced in Example \ref{so2n}, and $\mathbf{m}'(d,2n-1,2i)$ are the coefficients defined in Notation \ref{mdni}. \qed
\end{prop}

\subsubsection{Degree of the discriminant.} We will need the following proposition later to apply Corollaries \ref{descentc*2ndrev} and \ref{maincorappl}. A straightforward check shows that in all cases considered in Section \ref{projquad} above a general singular section has one singular point, so by Proposition \ref{degreediscr} the proposition also gives us the degrees of the discriminant varieties.

\begin{prop}\label{degreediscrhypprojquad}
If $X=\PP^n(\Co)$, then $\big\langle c_n(J(\mathcal{O}(d)),[X]\big\rangle=(n+1)(d-1)^n$. Similarly, if $Q$ a non-degenerate quadric in $\PP^n(\Co)$, then $\big\langle c_{n-1}(J(\mathcal{O}(d)|_Q),[Q]\big\rangle=2\sum_{i=0}^{n-1} (i+1)(d-1)^i$.
\end{prop}

\begin{proof}
We use $c(-)$ to denote the total Chern class of a vector bundle. In the case of the projective space~$\PP^n(\Co)$ we have then $c\big(J(\Oh(d))\big)= c\big((\Omega_{\PP^n(\Co)}\oplus\Oh)\otimes\Oh(d)\big)=c\big(\Oh(d-1)^{\oplus (n+1)}\big)$ by exact sequence~\ref{sesjetbundle}. For~$Q$ using exact sequence \ref{exactseqjets} we have
$$c\big(J(\Oh(d)|_Q)\big)=\frac{c\big(J(\Oh(d))|_Q\big)}{c\big(\Oh(d-2)|_Q\big)}=\frac{c\big(\Oh(d-1)^{\oplus(n+1)}|_Q\big)}{c\big(\Oh(d-2)|_Q\big)},$$
which implies the formula for $c_{n-1}(J(\mathcal{O}(d)|_Q)$. 
\end{proof}

\subsection{Complete flag varieties}
\label{completeflagvarieties}

Let $G$ be a connected complex reductive Lie group, $B\subset G$ a Borel subgroup, $T\subset B$ a maximal torus of $G$, and set $X=G/B$. 
Let $R\subset \mathfrak{X}(T)$ be the root system of $G$. For a character $\chi\in \mathfrak{X}(B)=\mathfrak{X}(T)$ we denote the line bundle $G\times_\chi \Co$ over $X$ by $\Oh(\chi)$. By the Borel-Weil-Bott theorem (see e.g.\ \cite{S95}, and also \cite{S07}*{\S7.4} for a detailed account), $\Gamma(X,\Oh(\chi))\neq 0$ if and only if $\chi$ is dominant with respect to the system $\Pi$ of simple roots determined by $B\supset T$.

Recall that in Section~\ref{the map alpha} we introduced a $\Z$-basis $\{e_w\}_{w\in W}$ for the integral homology groups $H_*(G/B,\Z)$. Here $W$ denotes the Weyl group of $G$; recall that $e_w\in H_{2\ell(w)}(G/B,\Z)$. Assume $\chi$ dominant and set $E'=J(\Oh(\chi))$. We will work out below the classes $S(e_G(E'))/e_w\in \fr{2d+1-2\ell(w)}(G,\Z), w\in W$ explicitly in several examples. 

To begin with, let us calculate $E'$ and $e_G(E')$. Let $H\subset G$ be a subgroup that contains $T$, and let $E_1=G\times_{R_1} V$ and $E_2=G\times_{R_2} V$ be two equivariant vector bundles over $G/H$ where $R_1,R_2$ are representations of $H$. We will write $E_1\cong_T E_2$ if $R_1|_T\cong R_2|_T$.

Let $R_+\subset R$ denote the set of positive roots. Using sequence~\ref{sesjetbundle} we deduce that
\begin{equation}\label{decjetbund}
\begin{aligned}
\Omega_{X}\cong G\times_B(\mathfrak{g}/\mathfrak{b})^*\cong_T \bigoplus_{\alpha\in R^{+}} \Oh(-\alpha), \\
E'=J(\Oh(\chi))\cong_T \Oh(\chi)\oplus \bigoplus_{\alpha\in R_{+}} \Oh(\chi-\alpha),
\end{aligned}
\end{equation}
which by Lemma \ref{inducedvectbund1} implies (setting $d=\dim X$)
\begin{equation*}
	e_G(E')=e_G\big(J(\Oh(\chi))\big)=\chi\prod_{\alpha\in R_+} (\chi-\alpha) \in \mathrm{Sym}^{d+1}(\mathfrak{X}(T))\subset H^{2d+2}(BT,\Q)\cong H_G^{2d+2}(G/B,\Q).
\end{equation*}

\begin{rmk} If $P\supset B$ is an arbitrary parabolic subgroup of $G$ and $X_P=G/P$, the cotangent bundle $\Omega_{X_P}$ has a similar description, namely this bundle is $\cong_T \bigoplus_{\alpha\in R^{+}\setminus\triangle_{\Theta}}\Oh(-\alpha)$ where $\triangle_{\Theta}$ is the subgroup of $\mathfrak{X}(T)$ generated by a subset $\Theta\subset\Pi$, see e.g.\ \cite{Steinberg67}. So if $\chi\in\mathfrak{X}(P)$, $E'=J(G\times_\chi \Co), d_P = \dim X_P$ and $p:G/B\to G/P$ is the projection, then we have
\begin{equation*}
	e_G(p^*(E'))=\chi\prod_{\alpha\in R_+\setminus\triangle_{\Theta}} (\chi-\alpha) \in \mathrm{Sym}^{d_P+1}(\mathfrak{X}(T))\subset H^{2d_P+2}(BT,\Q)\cong H_G^{2d_P+2}(G/B,\Q).
\end{equation*}
\end{rmk}

\begin{rmk}\label{polynomial}
Let $P$ and $X_P$ be as in the previous remark. For a character $\chi\in\mathfrak{X}(P)$ set $\Oh(\chi)=G\times_\chi \Co$. The map $H_*(X_P,\Z)\to \fr{*}(G,\Z)$ given by
$$y\mapsto S\Big(e_G\big(J(\Oh(\chi))\big)\Big)\Bigm/ y$$
depends polynomially on $\chi$.
\end{rmk}

The next step after calculating $e_G(E')$ is to find a system of generators for the ideal $I_1(X,\Q)$ and express $e_G(E')$ in terms of these generators, cf.\ Section \ref{genidealbargamma}. Below we only do this for $G=SL_3(\Co), Sp_4(\Co)$ and $G_2$. Note that the case of $G=$ any other connected complex reductive group of rank $\leq 2$ is either not interesting or is covered by the results of Section \ref{projquad}.

Set $G$ to be any of the groups $SL_3(\Co), Sp_4(\Co),G_2$. We let $\alpha_1,\alpha_2$ be the elements of $\Pi$ and denote the corresponding dominant weights by $\omega_1$ and $\omega_2$. We will write $\chi\in\mathfrak{X}(T)$ as $\chi=m\omega_1+n\omega_2$, $m,n\in\Z$, and express the classes $S(e_G(E'))/w,w\in W$ in terms of $m,n$. Note that $\chi$ is dominant (i.e.\ $\Gamma(X,\Oh(\chi))\neq 0$) if and only if $m,n\geq 0$. 

There are generators $\tilde s_1,\tilde s_2$ of the ring $H^*(BT,\Q)^{W}$ such that $\tilde s_1 \in H^4(BT,\Q)^{W}$ and $\tilde s_2 \in H^{2\ell(w_0)}(BT,\Q)^{W}$ where $w_0\in W$ is the longest element. Moreover, one can choose $\tilde s_1,\tilde s_2$ so that $\bar\gamma(\tilde s_1), \bar\gamma(\tilde s_2)$ are free multiplicative generators of $\fr{*}(G,\Z)$. This was done in Examples~\ref{sln} (for $n=3$),~\ref{sp2n} (for $n=2$), and~\ref{g2}. Let us fix such a pair of generators. There is a decomposition $e_G(J(\Oh(\chi)))=p_1\tilde s_1 +p_2\tilde s_2$ where $p_1\in H^{2\ell(w_0)-2}(BT,\Q)$ and $p_2\in H^2(BT,\Q)$. 

We now describe how to compute $\alpha^*(p_i)$. Let $\sigma_i \in W$ be the reflection of $\mathfrak{X}(T)\otimes\R$ in the hyperplane orthogonal to $\alpha_i, i=1,2$. Following~\cite{BGG}, we define operators $A_w\colon \mathrm{Sym}^*(\mathfrak{X}(T))\otimes\Q \to \mathrm{Sym}^*(\mathfrak{X}(T))\otimes\Q, w\in W$ by setting $A_{w_1w_2}= A_{w_1}A_{w_2}$ if $\ell(w_1 w_2)=\ell(w_1)+\ell(w_2)$ and, for $w=\sigma_1,\sigma_2$ and $f\in \mathrm{Sym}^*(\mathfrak{X}(T))\otimes\Q$,
\begin{align*}
A_{\sigma_1}f(\omega_1,\omega_2)&=\frac{f(\omega_1,\omega_2)-f(\sigma_1(\omega_1),\omega_2)}{\omega_1-\sigma_1(\omega_1)},\\
A_{\sigma_2}f(\omega_1,\omega_2)&=\frac{f(\omega_1,\omega_2)-f(\omega_1,\sigma_2(\omega_2))}{\omega_2-\sigma_2(\omega_2)}.
\end{align*}
By~\cite{BGG}*{Theorem 4.1} we have then $\langle \alpha^*(f), e_w\rangle = (A_w f)(0)$. Note that each of the operators $A_{\sigma_1},A_{\sigma_2}$ has degree $-1$. 
Applying Corollary \ref{maincorollary2} and formula \ref{explicits1gb}, and setting $E'=J(\Oh(\chi))$ we get the following identity in $\fr{*}(G,\Z)$:
\begin{equation}\label{explevallinkclasses}
O^*(\Lk(e_w)) = S(e_G(E'))/e_w=
\left\{\begin{array}{lll}
\langle \alpha^*(p_2), e_w \rangle \bar\gamma(\tilde s_1) = (A_w p_2)(0)\cdot\bar\gamma(\tilde s_2) &\text{for $w=\sigma_i, i=1,2$;}\\
\langle \alpha^*(p_1), e_w \rangle \bar\gamma(\tilde s_2) = (A_wp_1)(0)\cdot\bar\gamma(\tilde s_1) &\text{for $w=\sigma_iw_0, i=1,2$;}\\
0 &\text{otherwise.}
\end{array}\right. 
\end{equation}

Below for each of the groups $SL_3(\Co), Sp_4(\Co),G_2$ we give a table that contains all information necessary to compute $A_w p_i,i=1,2$, and after that we give formulas for the resulting classes $O^*(\Lk(e_w))$, $w=\sigma_i$, and $w=\sigma_iw_0$. 

We also calculate $c_d(E')([X])$. By Proposition \ref{degreediscr}, if this number is non-zero, then a general singular section of $\Oh(\chi)$ has $k$ singular points where $k$ is an integer $>0$, and if $k=1$, then $c_d(E')([X])$ is the degree of the discriminant variety $\Sing(\Oh(\chi))\subset \Gamma(X,\Oh(\chi))$. We have $[X]=e_{w_0}$. 
The equivariant Chern class $c_d(E'_{hG})$ can be read off \ref{decjetbund}. We then use the integration formula again to evaluate the result on $[X]$.

All calculations (i.e.\ finding $p_1,p_2$ and evaluating $A_w$) for complete flag varieties below were done using Singular~\cite{Singular}.

\subsubsection{$SL_3(\Co)$. }\label{sl3short}
\label{SL3}
Let $G=SL_3(\Co)$. We use the notation of Example \ref{sln} for $n=3$. The simple roots in table~\ref{tab:a3} correspond to the Borel subgroup of upper triangular matrices in $G$. Recall that we have $H^*(BT,\Q)\cong\Q[\epsilon_1, \epsilon_2, \epsilon_3] / (\epsilon_1+\epsilon_2+\epsilon_3)$.
Recall also that $\mathbf{x}_{2i-1} = \bar\gamma(s_i) \in H^{2i-1}(SL_3,\Z), i=1,2$ are free multiplicative generators of $H^*(SL_3,\Z)$. 

\begin{center}
\begin{table}[h]
\caption{Root system $A_3$.}
\begin{tabular}{|m{5.1cm}|m{10.5cm}|}
\hline 
Simple roots & \begin{tabular}{l} $\alpha_1=\epsilon_1-\epsilon_2, \alpha_2=\epsilon_1 + 2\epsilon_2$\end{tabular}\\
\hline
Fundamental weights & \begin{tabular}{l} $\omega_1=\epsilon_1,\omega_2=\epsilon_1+\epsilon_2$\end{tabular}\\
\hline
Positive roots in terms of $\omega_1, \omega_2$& \begin{tabular}{l}$\alpha_1=2\omega_1-\omega_2,\alpha_2=-\omega_1+2\omega_2,\alpha_1+\alpha_2=\omega_1+\omega_2$\end{tabular}\\
\hline
Action of $W$ on the fundamental weights & \begin{tabular}{l} $\sigma_1(\omega_2)=\omega_2$, $\sigma_2(\omega_1)=\omega_1$, \\
$\sigma_1(\omega_1)=\omega_2-\omega_1, \sigma_2(\omega_2)=\omega_1-\omega_2$\end{tabular}\\
\hline 
Generators $\tilde s_1,\tilde s_2$ of $H^*(BT,\Q)^{W}$ in terms of
$\omega_1, \omega_2$ & \begin{tabular}{l}
$\tilde s_1=s_1=\epsilon_1\epsilon_2+\epsilon_2\epsilon_3+\epsilon_3\epsilon_1=-\omega_1^2+\omega_1\omega_2-\omega_2^2$,\\
$\tilde s_2=s_2=\epsilon_1\epsilon_2\epsilon_3=\omega_1\omega_2(\omega_1-\omega_2)$
\end{tabular}\\
\hline
The longest element $w_0$ of $W$ & \begin{tabular}{l} $w_0 = \sigma_1\sigma_2\sigma_1$, $\ell(w_0)=3$\end{tabular}\\
\hline
\end{tabular}
\label{tab:a3}
\end{table}
\end{center}

Applying formula \ref{explevallinkclasses} we get
\begin{align*}
O^*(\Lk(e_{\sigma_1}))  = &\mathbf{x}_5 \big(m^{4} - 6 \, m^{2} n^{2} - 4 \, m n^{3} - 2 \, m^{3} + 6 \, m^{2} n + 12 \, m n^{2} + 2 \, n^{3} - 12 \, m n - 6 \, n^{2} + 3 \, m + 6 \, n\big) , \\
O^*(\Lk(e_{\sigma_2}))  = &\mathbf{x}_5 \big(4 \, m^{3} n + 6 \, m^{2} n^{2} - n^{4} - 2 \, m^{3} - 12 \, m^{2} n - 6 \, m n^{2} + 2 \, n^{3} + 6 \, m^{2} + 12 \, m n - 6 \, m - 3 \, n\big) , \\
O^*(\Lk(e_{\sigma_1 w_0}))  = &  \mathbf{x}_3\big(-m^{4} - 4 \, m^{3} n - 6 \, m^{2} n^{2} - 4 \, m n^{3} + 4 \, m^{3} + 12 \, m^{2} n + 12 \, m n^{2} + 2 \, n^{3} -5 \, m^{2} - 10 \, m n - \\
&6 \, n^{2} + 2 \, m + 4 \, n \big),\\
O^*(\Lk(e_{\sigma_2 w_0}))  = &\mathbf{x}_3 \big(-4 \, m^{3} n - 6 \, m^{2} n^{2} - 4 \, m n^{3} - n^{4} + 2 \, m^{3} + 12 \, m^{2} n + 12 \, m n^{2} + 4 \, n^{3} - 6 \, m^{2} - 10 \, m n - \\
&5 \, n^{2} + 4 \, m + 2 \, n\big).
\end{align*}
We let $\mathbf{m}(\chi,SL_3(\Co),2)$, respectively $\mathbf{m}(\chi,SL_3(\Co),3)$ be the least common multiple of the coefficients of $\mathbf{x}_3$, respectively $\mathbf{x}_5$ in these formulas.

Moreover,
\begin{equation*}
c_d(E')([X])=A_{w_0}(c_3(J(\Oh(\chi))_G))=12 \, m^{2} n + 12 \, m n^{2} - 6 \, m^{2} - 24 \, m n - 6 \, n^{2} + 12 \, m + 12 \, n - 6.
\end{equation*}

\subsubsection{$Sp_4(\Co)$.}\label{sp4short}
We use the notation introduced in Example \ref{sp2n} for $n=2$. The set of simple roots in table~\ref{tab:c2} corresponds to the Borel subgroup of upper triangular matrices in $G$. Recall that we have $H^*(BT,\Q)\cong\Q[\e_1,\e_2]$.
Recall also that $\mathbf{y}_{2i-1} = \bar\gamma(s_i) \in H^{2i-1}(Sp_4(\Co),\Z), i=1,2$ are free multiplicative generators of $H^*(Sp_4(\Co),\Z)$.

\begin{center}
\begin{table}[h]
\caption{Root system $C_2$.}
\begin{tabular}{|m{5.1cm}|m{10.5cm}|}
\hline 
Simple roots & \begin{tabular}{l} $\alpha_1=\e_1-\e_2, \alpha_2=2\e_2$\end{tabular}\\
\hline
Fundamental weights & \begin{tabular}{l} $\omega_1=\e_1,\omega_2=\e_1+\e_2$\end{tabular}\\
\hline
Positive roots in terms of $\omega_1, \omega_2$& \begin{tabular}{l}$\alpha_1=2\omega_1-\omega_2,\alpha_2=-2\omega_1+2\omega_2,\alpha_1+\alpha_2=\omega_2,2\alpha_1+2\alpha_2=2\omega_1$\end{tabular}\\
\hline
Action of $W$ on the fundamental weights & \begin{tabular}{l} $\sigma_1(\omega_2)=\omega_2$, $\sigma_2(\omega_1)=\omega_1$, \\
$\sigma_1(\omega_1)=\omega_2-\omega_1, \sigma_2(\omega_2)=2\omega_1-\omega_2$\end{tabular}\\
\hline
Generators $\tilde s_1, \tilde s_2$ of $H^*(BT,\Q)^{W}$ in terms of
$\omega_1, \omega_2$ & \begin{tabular}{l}
$\tilde s_1=s_1=\e_1^2+\e_2^2=2\omega_1^2-2\omega_1\omega_2+\omega_2^2$,\\
$\tilde s_2=s_2=(\e_1\e_2)^2=\omega_1^2(\omega_1-\omega_2)^2$
\end{tabular}\\
\hline
The longest element $w_0$ of $W$ & \begin{tabular}{l} $w_0 = (\sigma_1\sigma_2)^2$, $\ell(w_0)=4$\end{tabular}\\
\hline
\end{tabular}
\label{tab:c2}
\end{table}
\end{center}

Applying formula \ref{explevallinkclasses} we get
\begin{align*}
O^*(\Lk(e_{\sigma_1}))  = &\mathbf{y}_7\big(-m^{5} + 20 \, m^{3} n^{2} + 40 \, m^{2} n^{3} + 20 \, m n^{4} + 2 \, m^{4} - 16 \, m^{3} n - 72 \, m^{2} n^{2} - 64 \, m n^{3} - 8 \, n^{4} + \\
&2 \, m^{3} + 48 \, m^{2} n + 84 \, m n^{2} + 24 \, n^{3} - 8 \, m^{2} - 48 \, m n - 32 \, n^{2} + 8 \, m + 16 \, n\big), \\
O^*(\Lk(e_{\sigma_2}))  = &\mathbf{y}_7\big( -5 \, m^{4} n - 20 \, m^{3} n^{2} - 20 \, m^{2} n^{3} + 4 \, n^{5} + 2 \, m^{4} + 24 \, m^{3} n + 48 \, m^{2} n^{2} + 16 \, m n^{3} - 8 \, n^{4} - \\
&8 \, m^{3} - 42 \, m^{2} n - 36 \, m n^{2} + 4 \, n^{3} + 12 \, m^{2} + 32 \, m n + 8 \, n^{2} - 8 \, m - 8 \, n\big), \\
O^*(\Lk(e_{\sigma_1 w_0})) = &\mathbf{y}_3\big(  5 \, m^{4} n + 20 \, m^{3} n^{2} + 30 \, m^{2} n^{3} + 20 \, m n^{4} + 4 \, n^{5} - 2 \, m^{4} - 24 \, m^{3} n - 60 \, m^{2} n^{2} - 56 \, m n^{3} - \\
&16 \, n^{4} + 8 \, m^{3} + 39 \, m^{2} n + 54 \, m n^{2} + 24 \, n^{3} - 10 \, m^{2} - 24 \, m n - 16 \, n^{2} + 4 \, m + 4 \, n\big), \\
O^*(\Lk(e_{\sigma_2 w_0})) = &\mathbf{y}_3\big( m^{5} + 10 \, m^{4} n + 30 \, m^{3} n^{2} + 40 \, m^{2} n^{3} + 20 \, m n^{4} - 6 \, m^{4} - 40 \, m^{3} n - 84 \, m^{2} n^{2} - 64 \, m n^{3} - \\
&8 \, n^{4} + 13 \, m^{3} + 54 \, m^{2} n + 72 \, m n^{2} + 24 \, n^{3} - 12 \, m^{2} - 32 \, m n - 24 \, n^{2} + 4 \, m + 8 \, n\big).
\end{align*}
We let $\mathbf{m}(\chi,Sp_4(\Co),2)$, respectively $\mathbf{m}(\chi,Sp_4(\Co),4)$ be the least common multiple of the coefficients of $\mathbf{y}_3$, respectively $\mathbf{y}_7$ in these formulas.

Moreover,
\begin{align*}
c_d(E')([X])=&A_{w_0}\big(c_4(J(\Oh(\chi))_G)\big)=20 \, m^{3} n + 60 \, m^{2} n^{2} + 40 \, m n^{3} - 8 \, m^{3} - 72 \, m^{2} n - 96 \, m n^{2} - 16 \, n^{3} +  \\
&24 \, m^{2} + 84 \, m n +36 \, n^{2} - 24 \, m - 32 \, n + 8.
\end{align*}

\subsubsection{$G_2$.}\label{G2short} 
We use the notation of Example \ref{g2}, with two exceptions: we let $T$ be the maximal torus of $G$ that was denoted $T_G$ in Example \ref{g2}, and we use $W$ to denote the Weyl group of $G$. Recall that we have $H^*(BT,\Q)\cong\Q[\e_1,\e_2,\e_3]/(\e_1+\e_2+\e_3)$. 
Recall also that $\mathbf{w}_{3} = \frac{1}{2}\bar\gamma(s_{1,G_2}) \in \fr{3}(G_2,\Z)$ and $\mathbf{w}_{11}=\frac{1}{2}\bar\gamma(s_{3,G_2}) \in \fr{11}(G_2,\Z)$ are free multiplicative generators of $\fr{*}(G_2,\Z)$, where $s_{1,G_2}$ and $s_{3,G_2}$ are the polynomial generators of $H^*(BT,\Q)^{W}$ from Example~\ref{g2}.

\begin{center}
\begin{table}[h]
\caption{Root system $G_2$.}
\begin{tabular}{|m{5.1cm}|m{10.7cm}|}
\hline 
Simple roots & \begin{tabular}{l} $\alpha_1=\e_2, \alpha_2=\e_1-\e_2$\end{tabular}\\
\hline
Fundamental weights & \begin{tabular}{l} $\omega_1=\e_1+\e_2,\omega_2=\e_1-\e_3=2\e_1+\e_2$\end{tabular}\\
\hline
Positive roots in terms of $\omega_1, \omega_2$& \begin{tabular}{l}$\alpha_1=2\omega_1-\omega_2,\alpha_2=-3\omega_1+2\omega_2,\alpha_1+\alpha_2=-\omega_1+\omega_2,$\\
$2\alpha_1+\alpha_2=\omega_1, 3\alpha_1+2\alpha_2=\omega_2,3\alpha_1+\alpha_2=3\omega_1-\omega_2$\end{tabular}\\
\hline
Action of $W$ on the fundamental weights & \begin{tabular}{l} $\sigma_1(\omega_2)=\omega_2$, $\sigma_2(\omega_1)=\omega_1$, \\
$\sigma_1(\omega_1)=\e_1=\omega_2-\omega_1, \sigma_2(\omega_2)=\e_1+2\e_2=3\omega_1-\omega_2$\end{tabular}\\
\hline
Generators $\tilde s_1, \tilde s_2$ of $H^*(BT,\Q)^{W}$ in terms of
$\omega_1, \omega_2$ & 
\begin{tabular}{l}
$\tilde s_1=\dfrac{s_{1,G_2}}{2}=\dfrac{\rule{0pt}{11pt} \e_1^2+\e_2^2+\e_3^2}{2}=3\omega_1^2-3\omega_1\omega_2+ \omega_2^2$,\\
$\tilde s_2=\dfrac{s_{3,G_2}}{2}=\dfrac{(\e_1\e_2\e_3)^2}{2}=\dfrac{4\omega_1^{6} - 12\omega_1^{5}\omega_2 + 13\omega_1^{4}\omega_2^{2} - 6\omega_1^{3}\omega_2^{3} + \omega_1^{2}\omega_2^{4}}{\strut 2}$
\end{tabular}\\
\hline
The longest element $w_0$ of $W$ & \begin{tabular}{l} $w_0 = (\sigma_1\sigma_2)^3$, $\ell(w_0)=6$\end{tabular}\\
\hline
\end{tabular}
\label{tab:g2}
\end{table}
\end{center}

Applying formula \ref{explevallinkclasses} we get
\begin{align*}
O^*(\Lk(e_{\sigma_1})) = 2&\mathbf{w}_{11}\big( m^{7} - 63 \, m^{5} n^{2} - 315 \, m^{4} n^{3} - 630 \, m^{3} n^{4} - 567 \, m^{2} n^{5} - 189 \, m n^{6} - \\
&2 \, m^{6} + 36 \, m^{5} n + 360 \, m^{4} n^{2} +300 \, m^{3} n^{3} + 1350 \, m^{2} n^{4} + 648 \, m n^{5} + 54 \, n^{6} - \\
&4 \, m^{5} - 150 \, m^{4} n - 1440 \, m^{2} n^{3} - 990 \, m n^{4} - 162 \, n^{5} + \\
&20 \, m^{4} + 264 \, m^{3} n + 828 \, m^{2} n^{2} + 864 \, m n^{3} + 234 \, n^{4} - \\
&30 \, m^{3} - 234 \, m^{2} n - 432 \, m n^{2} - 198 \, n^{3} + 24 \, m^{2} + 108 \, m n + 90 \, n^{2} - 9 \, m - 18 \, n\big), \\
O^*(\Lk(e_{\sigma_2}))  = 2&\mathbf{w}_{11}\big( 7 \, m^{6} n + 63 \, m^{5} n^{2} + 210 \, m^{4} n^{3} + 315 \, m^{3} n^{4} + 189 \, m^{2} n^{5} - 27 \, n^{7} - \\
&2 \, m^{6} - 48 \, m^{5} n - 270 \, m^{4} n^{2} - 600 \, m^{3} n^{3} - 540 \, m^{2} n^{4} - 108 \, m n^{5} + 54 \, n^{6} + \\
&10 \, m^{5} + 130 \, m^{4} n + 480 \, m^{3} n^{2} + 660 \, m^{2} n^{3} + 270 \, m n^{4} - 36 \, n^{5} - \\
&22 \, m^{4} - 184 \, m^{3} n - 432 \, m^{2} n^{2} - 312 \, m n^{3} - 18 \, n^{4} + \\
&26 \, m^{3} + 144 \, m^{2} n + 198 \, m n^{2} + 54 \, n^{3} - 18 \, m^{2} - 60 \, m n - 36 \, n^{2} + 6 \, m + 9 \, n\big),\\
O^*(\Lk(e_{\sigma_1w_0})) = &\mathbf{w}_3\big(28 \, m^{6} n + 252 \, m^{5} n^{2} + 910 \, m^{4} n^{3} + 1680 \, m^{3} n^{4} + 1638 \, m^{2} n^{5} + 756 \, m n^{6} + 108 \, n^{7} - \\
&8 \, m^{6} - 192 \, m^{5} n -  1140 \, m^{4} n^{2} - 2960 \, m^{3} n^{3} - 3780 \, m^{2} n^{4} - 2232 \, m n^{5} - 432 \, n^{6} + \\
&40 \, m^{5} + 500 \, m^{4} n + 2000 \, m^{3} n^{2} + 3520 \, m^{2} n^{3} + 2760 \, m n^{4} + 732 \, n^{5} -\\
&80 \, m^{4} - 640 \, m^{3} n - 1704 \, m^{2} n^{2} - 1856 \, m n^{3} - 684 \, n^{4} + \\
&80 \, m^{3} + 420 \, m^{2} n + 696 \, m n^{2} + 372 \, n^{3} - 40 \, m^{2} - 128 \, m n - 108 \, n^{2} + 8 \, m + 12 \, n\big),\\
O^*(\Lk(e_{\sigma_2w_0})) = &\mathbf{w}_3\big( 4 \, m^{7} + 84 \, m^{6} n + 546 \, m^{5} n^{2} + 1680 \, m^{4} n^{3} + 2730 \, m^{3} n^{4} + 2268 \, m^{2} n^{5} + 756 \, m n^{6} - \\
&32 \, m^{6} - 456 \, m^{5} n - 2220 \, m^{4} n^{2} - 5040 \, m^{3} n^{3} - 5580 \, m^{2} n^{4} - 2592 \, m n^{5} - 216 \, n^{6} + \\
&100 \, m^{5} + 1000 \, m^{4} n + 3520 \, m^{3} n^{2} + 5520 \, m^{2} n^{3} + 3660 \, m n^{4} + 648 \, n^{5} - \\
&160 \, m^{4} - 1136 \, m^{3} n - 2784 \, m^{2} n^{2} - 2736 \, m n^{3} - 816 \, n^{4} + \\
&140 \, m^{3} + 696 \, m^{2} n + 1116 \, m n^{2} + 552 \, n^{3} - 64 \, m^{2} - 216 \, m n - 192 \, n^{2} + 12 \, m + 24 \, n\big).
\end{align*}

We let $\mathbf{m}(\chi,G_2,2)$, respectively $\mathbf{m}(\chi,G_2,6)$ be the least common multiple of the coefficients of $\mathbf{w}_3$, respectively $\mathbf{w}_{11}$ in these formulas.

Moreover,
\begin{align*}
c_d(E')([X])=& A_{w_0}\big(c_5(J(\Oh(\chi))_G)\big) = 84 \, m^{5} n + 630 \, m^{4} n^{2} + 1680 \, m^{3} n^{3} + 1890 \, m^{2} n^{4} + 756 \, m n^{5} - 24 \, m^{5} -  \\
&480 \, m^{4} n -2160 \, m^{3} n^{2} - 3600 \, m^{2} n^{3} - 2160 \, m n^{4} - 216 \, n^{5} + 100 \, m^{4} + 1040 \, m^{3} n + 2880 \, m^{2} n^{2} + \\
&2640 \, m n^{3} + 540 \, n^{4} - 176 \, m^{3} - 1104 \, m^{2} n - 1728 \, m n^{2} - 624 \, n^{3} + 156 \, m^{2} + 576 \, m n + 396 \, n^{2} - \\
&72 \, m - 120 \, n + 12.
\end{align*}

\subsection{Complete intersection of type \texorpdfstring{$(d,d,\ldots, d)$}{Lg} in Grassmann varieties.}\label{fano} Here we illustrate how our methods work for bundles of rank $>1$ and give a few applications to some Fano varieties. Let $n,k,d>0$ and $r\geq 0$ be integers. We let $G=SL_{n+k}(\Co)$ and take $X$ to be the Grassmann variety $\mathrm{Gr}(k,n+k)$ of $k$-planes in $\Co^{n+k}$. Let $\Oh(1)$ be the very ample line bundle over $X$ that corresponds to the Pl\"{u}cker embedding $\mathrm{Gr}(k,n+k)\hookrightarrow \PP^{\binom{n+k}{k}-1}$, and set $\Oh(d)=\Oh(1)^{\otimes d}$. Let $E$ be the direct sum of $r+1$ copies of the line bundle $\Oh(d), r< \dim X$. The group $\widetilde G$ (see Section \ref{sectionsvsloci}) in this case is a split extension of $G=SL_{n+k}(\Co)$ by $\aut(E/X)\cong GL_{r+1}(\Co)$ (see Lemma~\ref{tildeGandG}), so it is isomorphic to $SL_{n+k}(\Co)\times GL_{r+1}(\Co)$. This group acts on $\PP(E^*)$, and the line bundle $\Oh_{\PP(E^*)}(1)$ is $\widetilde G$-equivariant. Moreover, by the Cayley trick (see Section \ref{secsmoothsec}) we have $\reg{X,E}\cong\reg{\PP(E^*),\Oh_{\PP(E^*)}(1)}$, so $\widetilde G$ acts on $\reg{X,E}$. Set $L=\Oh_{\PP(E^*)}(1)$ and $E'=J(L)$. In this subsection we give a recipe for
calculating 
the classes $S(e_{\tilde G}(E'))/y$ where $y\in H_*(\PP(E^*),\Z)$.

Let us identify $\PP(E^*)$ with $\mathrm{Gr}(k,n+k)\times \PP^r$. Then $L$ is identified with $\Oh(d)\boxtimes \Oh_{\PP^r}(1)$, and the action of $\widetilde{G}\cong SL_{n+k}(\Co)\times GL_{r+1}(\Co)$ on $L$ with the direct product of the action of~$SL_{n+k}(\Co)$ on $\Oh(d)$ and the action of $GL_{r+1}(\Co)$ on $\Oh_{\PP^r}(1)$.

Let $T_1\subset SL_{n+k}(\Co)$ and $T_2\subset GL_{r+1}(\Co)$ be the subgroups of diagonal matrices, and set $T=T_1\times T_2\subset \widetilde{G}$. Then $T$ is a maximal torus of $\widetilde{G}$. Let $P_1\subset SL_{n+k}(\Co)$ be the stabiliser of the plane spanned on the first~$k$ basis vectors in $\Co^{n+k}$ and $P_2\subset GL_{r+1}(\Co)$ be the stabiliser of the point $[1:0:\ldots:0]\in \PP^r$, and set $P=P_1\times P_2$. Then $P$ is a parabolic subgroup of $\widetilde{G}$, and $\widetilde{G}/P\cong \PP(E^*)$. 

We now identify the rational cohomology of $BP_1, BP_2, BP$ and $BG$ with subrings of $H^*(BT,\Q)$. Let $\epsilon_1,\ldots, \epsilon_{n+k}\in \mathfrak{X}(T_1)$ be the same elements as in Example~\ref{sln}; in this subsection we denote the elements $\varepsilon_1, \ldots, \varepsilon_{r+1}\in \mathfrak{X}(T_2)$ from Example~\ref{gln} by $\zeta_1,\ldots, \zeta_{r+1}$ respectively to avoid confusion. We have $H^*(BT,\Q)\cong \Q[\epsilon_1,\ldots,\epsilon_{n+k},\zeta_1,\ldots \zeta_{r+1}]/(\epsilon_1+\cdots + \epsilon_{n+k})$. We set 
\begin{align*}
&w_i=\sigma_i(\epsilon_1,\ldots,\epsilon_{k}), 1\leq i\leq k; b_i=\sigma_i(\epsilon_{k+1},\ldots,\epsilon_{n+k}), 1\leq i\leq n;\\
&u=\zeta_1, u_i=\sigma_{i}(\zeta_2,\ldots,\zeta_{r+1}), 1\leq i \leq r;\\
&s_i=\sigma_i(\epsilon_1,\ldots, \epsilon_{n+k}), i=2,\ldots, n+k; t_i=\sigma_i(\zeta_1,\ldots,\zeta_{r+1}), i=1,\ldots,r+1.
\end{align*}
Note that $w_1+b_1=0$. We have then 
\begin{align*}
&H^*(BP_1,\Q)\cong\Q[w_1,\ldots,w_k, b_2,\ldots, b_n], H^*(BP_2,\Q)\cong \Q[u,u_1,\ldots, u_r],\\
&H^*(BG,\Q)\cong \Q[s_2,\ldots, s_{n+k}, t_1,\ldots,t_{r+1}], H^*(BP,\Q)\cong \Q[w_1,\ldots,w_k,b_2,\ldots,b_n,u,u_1,\ldots, u_{r}].
\end{align*}

With these identifications the map $\beta^*\colon H^*(BG,\Q)\to H^*(BP,\Q)$ is simply the inclusion, cf.\ Section~\ref{explprojspace}, so $\beta^*(t_i)=u_i+uu_{i-1}$ (where we set $u_0=1$, $u_{r+1}=0$), and $\beta^*(s_i)$ is the degree $2i$ part of
\begin{equation*}
(1+w_1+w_2+\cdots+w_k)(1+b_1+b_2+\cdots+b_n).
\end{equation*}

The weight of the representation of $P$ which corresponds to the line bundle $L\cong \Oh(d)\boxtimes \Oh_{\PP^r}(1)$ is $-d\epsilon_1-\zeta_1$. The cotangent bundle $\Omega_{\PP(E^*)}$ is isomorphic to the direct sum $\pi_1^*\Omega_{\mathrm{Gr}(k,n+k)}\oplus \pi_2^*\Omega_{\PP^r}$ where $\pi_1\colon \PP(E^*)\to \mathrm{Gr}(k,n+k)$ and $\pi_2\colon \PP(E^*)\to\PP^r$ are the projections. Let $U$, respectively $Q$ be the tautological rank $k$ vector bundle, respectively the universal quotient bundle over $\mathrm{Gr}(k,n+k)$. By Example \ref{tautlinebund},
$\Omega_{\mathrm{Gr}(k,n+k)}\cong\Hom(Q,U)$ is obtained from the representation of $P_1$ with weights $\epsilon_i-\epsilon_j,i\in\{1,\ldots, k\},j\in\{k+1,\ldots, n+k\}$. Similarly, we know from Section~\ref{explprojspace} that the weights of the representation of $P_2$ that induces $\Omega_{\PP^r}$ are $\zeta_1-\zeta_i,i\in\{2,\ldots, r+1\}$.

So by exact sequence \ref{sesjetbundle} the weights of the representation of $P$ such that the associated vector bundle over $\widetilde G/P$ is $J(L)$ are
$$-d\epsilon_1-\zeta_i,i\in\{1,\ldots, r+1\}, \epsilon_i-\epsilon_j+(-d\epsilon_1-\zeta_1),i\in\{1,\ldots, k\},j\in\{k+1,\ldots, n+k\},$$ and by Lemma \ref{inducedvectbund1} the product of these is the Euler class $e_{\widetilde G}(J(L))\in H^*(BT,\Q)\cong H^*_{\widetilde G}(\widetilde G/P,\Q)$.

Set $a_i=\beta^*(s_i)$, $2\leq i\leq n+k$, and $d_i=\beta^*(t_i)$, $1\leq i\leq r+1$. Suppose we have found a decomposition $e_G(E')=\sum a_ip_i+\sum d_jq_j$ with $p_i,q_i\in H^*(BP,\Q)$. Then using Theorem \ref{maintheorem} and formula \ref{formulas1} we see that for every $y\in H_*(\PP(E^*),\Z)$ we have
\begin{equation}\label{lkgrassmannian}
	O^*(\Lk(y))=S(e_G(E'))/y=\left.\left(\sum^{n+k}_{i=2}\bar\gamma(s_i)\times\alpha^*(p_i) +\sum_{j=1}^{r+1}\bar\gamma(t_j)\times\alpha^*(q_j) \right)\middle/y\right..
\end{equation}

The cohomology classes $\bar\gamma(s_i)$, and $\bar\gamma(t_j)$ were calculated in Examples~\ref{gln} and~\ref{sln}. In particular, these classes are free multiplicative generators of $H^*(\widetilde{G},\Z)$. Let us describe the map $\alpha^*$. As in Section~\ref{explprojspace}, we have $\alpha^*(u)=-c_1(\Oh_{\PP^r}(1))$, $\alpha^*(u_i)=(-1)^{i-1}(\alpha^*(u))^i$.

Calculating $\alpha^*(w_i)$ and $\alpha^*(b_i)$ is more complicated, but to do this we only need the classical Schubert calculus, and not its generalised version \cite{BGG}. First we note that $\alpha^*(w_i)=c_i(U)$ and $\alpha^*(b_i)=c_i(Q)$ by Lemma \ref{inducedvectbund1}. 
Next, recall that a {\it partition} $\lambda$ of length $l$ is a non-increasing sequence $\lambda_1\geq \lambda_2\geq \ldots \geq \lambda_l$ of positive integers. Set $|\lambda|=\lambda_1+\ldots+\lambda_l$. Let $\mathcal{P}_{n,k}$ be the set of all partitions $\lambda$ of length at most $n$ and such that $|\lambda|\leq nk$. 
\begin{prop}\label{schubert}
There is a $\Z$-basis $\{c_{\lambda}\}_{\lambda\in\mathcal{P}_{n,k}}$ of $H^{*}(\mathrm{Gr}(k,n+k),\Z)$ such that
\begin{enumerate}
    \item $c_\lambda\in H^{2|\lambda|}(\mathrm{Gr}(k,n+k),\Z)$, $\lambda\in \mathcal{P}_{n,k}$;
    \item $\alpha^*(w_i)=c_i(U)=(-1)^ic_{(1,\ldots,1)}\in H^{2i}(\mathrm{Gr}(k,n+k),\Z), i=1,\ldots,k$;
    \item $\alpha^*(b_i)=c_i(Q)=c_{(i)}\in H^{2i}(\mathrm{Gr}(k,n+k),\Z), i=2,\ldots n$.
\end{enumerate}
\end{prop}
\begin{proof}
See~\cite{Ful98}*{Sections 14.5--14.7}.
\end{proof}
Moreover, if $c_{\lambda}c_{\mu}=\sum N_{\lambda,\mu,\rho}c_{\rho}$, then the numbers $N_{\lambda,\mu,\rho}$ can be calculated in terms of Young diagrams by the Littlewood-Richardson rule (see e.g.\ \cite{Ful98}*{Lemma 14.5.3}).

The description of the cup product, formula \ref{lkgrassmannian} and Proposition \ref{schubert} allow one to calculate for every $l\geq 1$ the order $i_l$ of the cokernel of map~\ref{yetanothers}.
For illustration we give here the answer for smooth intersections of the Pl\"ucker embedded $\mathrm{Gr}(2,6)$ and a projective subspace $L^5\subset \PP^{14}(\Co)$ of codimension 5 ($k=2$, $n=4$, $d=1$, $r=4$).

\begin{center}
\begin{table}[h!]
    \caption{Smooth intersections of $\mathrm{Gr}(2,6)$ and $L^5$ ($k=2$, $n=4$, $d=1$, $r=4$)}
    \begin{tabular}[c]{|c|c|c|c|c|c|c|c|c|} \hline
	$l$ & 1 & 2 & 3 & 4 & 5 & 6 \\ \hline
	$\dim H_{26-2l}(\PP(E^*),\Z)$ & 1 & 2 & 4 & 6 & 9 & 10 \\ \hline 
	$\dim P^{2l-1}$ & 1 & 2 & 2 & 2 & 2 & 1 \\ \hline
	$i_l$ & 48 & 384 & 2 & 10 & 11 & 3 \\ \hline
    \end{tabular}
    \label{tab:linearsectionofGr26_dim3}
\end{table}
\end{center}

More examples are given in Appendix~\ref{examples of fano}. By Corollary~\ref{stabilizerisproduct}, if none of the orders in the bottom row of the table for given parameters $k,n,d,r$ is infinite, then the product of these numbers 
is the index of $\Lambda^{\dim\widetilde{G}}_{\widetilde G,E'}$ in $\fr{\dim\widetilde{G}}(\widetilde{G},\Z)$, and $\Lambda^*_{E',\widetilde G}$ has finite index in $\fr{*}(\widetilde G,\Z)$; note that in this case the product is divisible by every $|\widetilde G_s|, s\in\reg{X,E}$. For comparison we also include a case when the automorphism group is known to be infinite, namely transversal intersections of $\mathrm{Gr}(2,5)\subset\PP^6(\Co)$ and projective subspaces of codimension 3 (see e.g.\ \cite{KPS18}*{Theorem~5.1.1}). As expected, one of the cokernels turns out to be infinite. As in the case of complete flag varieties, the answers were obtained using Singular~\cite{Singular}.

\subsection{Summary}\label{summaryexamples}

Here we summarise {what we have done in} Sections \ref{projquad}-\ref{fano}, {prove a few corollaries and discuss several related results}. 

\begin{thm}\label{mainthmprojhyp}
Let $G,X$ and $E$ be as in table \ref{tab:autogroupstable}; in the table $Q_k$ denotes a non-singular quadratic hypersurface in $\PP^{k+1}(\Co)$. Let $\widetilde{G}, \bar G$ and $\aut (E/X)$ be the groups defined in Section \ref{sectionsvsloci}. (Recall that 
the group $\widetilde{G}$ contains both $G$ and $\aut(E/X)$.)
Suppose that $\reg{X,E}\neq \varnothing$ and that the number in column 4 of the table is non-zero. (Note that this holds if $d\geq 3$ for $X=\PP^n(\Co),E=\Oh(d)$, and for $d\geq 2$ if $X=Q_k$ and $E=\Oh(d)$.) Then

\begin{enumerate}
\item The geometric quotient $\reg{X,E}/\widetilde{G}$ exists and is affine. The group $\aut(E/X)$ is reductive, and the geometric quotients $Z_{\mathrm{reg}}= \reg{X,E})/\aut(E/X)$ and $Z_{\mathrm{reg}}/\bar G$ also exist and are affine. Moreover, we have
\begin{align*}
H^*(\reg{X,E},\Q) & \cong H^*(\widetilde{G},\Q)\otimes H^*(\reg{X,E}/\widetilde{G},\Q),\\ 
H^*(Z_{\mathrm{reg}},\Q) & \cong H^*(\bar G,\Q)\otimes H^*(Z_{\mathrm{reg}}/\bar G,\Q)
\end{align*}
both as rings and mixed Hodge structures.

\item For any $s\in\reg{X,E}$ the order of the stabiliser $\widetilde{G}_s$ divides the number in column 4 of the table, and the order of $\bar G_{Z(s)}$ divides the same number divided by $|Z(G)|$.
\end{enumerate}

\end{thm}

\begin{center}
\begin{table}[h]
\caption{}
\begin{tabular}{|c|c|c|c|}\hline
 $G$ & $X$ & $E$ & \begin{tabular}{c} Multiple of \\ $|\widetilde{G}_s|,s\in\reg{X,E}$ \end{tabular} \\ \hline
 $SL_{n+1}(\Co)$ & $\PP^n(\Co)$ & $\Oh(d)$ &  \begin{tabular}{c}$(n+1)(d-1)^n \prod_{i=2}^{n+1} \mathbf{m}(d,n,i)$\\ (see Section \ref{projquad}) \end{tabular} \\ \hline
 &&& \\[-1.4em]
 $SO_{2n+1}(\Co)$ & $Q_{2n-1}$ & $\Oh(d)$ &  \begin{tabular}{c}
 $2^{\left\lfloor\frac{3n}{2}\right\rfloor+1}\left(\sum_{i=0}^{2n-1} (i+1)(d-1)^i\right)\prod_{i=1}^{n} \mathbf{m}'(d,2n,2i)$\\ (See Section \ref{projquad}) \end{tabular}\\ \hline
&&& \\[-1.4em]
 $SO_{2n}(\Co)$ & $Q_{2n-2}$  & $\Oh(d)$ &  \begin{tabular}{c} $2^{\left\lfloor\frac{3n}{2}\right\rfloor}\left(\sum_{i=0}^{2n-2} (i+1)(d-1)^i\right)\prod_{i=1}^{n} \mathbf{m}'(d,2n-1,2i)$\\ (See Section \ref{projquad}) \end{tabular}\\ \hline
 $SL_3(\Co)$ & $G/B$ & \begin{tabular}{c} $\Oh(\chi)$\\ $\chi=m\omega_1+n\omega_2$\end{tabular} & \begin{tabular}{c} $\mathbf{m}(\chi,SL_3(\Co),2)\cdot \mathbf{m}(\chi,SL_3(\Co),3)\cdot \deg\Sing(\Oh(\chi))$\\ (see Section \ref{completeflagvarieties}) \end{tabular}\\ \hline
 $Sp_4(\Co)$ & $G/B$ & \begin{tabular}{c} $\Oh(\chi)$\\ $\chi=m\omega_1+n\omega_2$\end{tabular} & \begin{tabular}{c} $\mathbf{m}(\chi,Sp_4(\Co),2)\cdot \mathbf{m}(\chi,Sp_4(\Co),4)\cdot \deg\Sing(\Oh(\chi))$\\ 
 (see Section \ref{sp4short}) \end{tabular}\\ \hline
 $G_2$ & $G/B$ & \begin{tabular}{c} $\Oh(\chi)$\\ $\chi=m\omega_1+n\omega_2$\end{tabular} & \begin{tabular}{c} $\mathbf{m}(\chi,G_2,2)\cdot \mathbf{m}(\chi,G_2,6)\cdot \deg\Sing(\Oh(\chi))$\\ 
 (see Section \ref{G2short}) \end{tabular}\\
\hline
$SL_{5}(\Co)$ & $\PP^4(\Co)$ & $\Oh(2)^{\oplus 3}$ & \begin{tabular}{c} $2^{18}\cdot 3\cdot 5^2$ (see table \ref{tab:intersectionofthreequadricsinP4}) \end{tabular}\\
\hline
 $SL_{6}(\Co)$ & $\PP^5(\Co)$ & $\Oh(2)^{\oplus 2}$ & \begin{tabular}{c} $2^{15}\cdot 3^{3}\cdot 5$ (see table \ref{tab:intersectionoftwoquadrics}) \end{tabular}\\
\hline
 $SL_{7}(\Co)$ & $\PP^6(\Co)$ & $\Oh(2)^{\oplus 3}$ & \begin{tabular}{c} $2^{24}\cdot 3^2\cdot 5^2\cdot 7^3$ (see table \ref{tab:intersectionofthreequadrics}) \end{tabular}\\
\hline
 $SL_{6}(\Co)$ & $\mathrm{Gr}(2,6)$ & $\Oh(1)^{\oplus 7}$ & \begin{tabular}{c} $2^{16}\cdot 3^3 \cdot 7\cdot 11$ (see table \ref{tab:linearsectionofGr26_dim1}) \end{tabular}\\
\hline
$SL_{6}(\Co)$ & $\mathrm{Gr}(2,6)$ & $\Oh(1)^{\oplus 5}$ & \begin{tabular}{c} $2^{13}\cdot 3^3 \cdot 5\cdot 11$ (see table \ref{tab:linearsectionofGr26_dim3}) \end{tabular}\\
\hline
 $SL_{7}(\Co)$ & $\mathrm{Gr}(2,7)$ & $\Oh(1)^{\oplus 6}$ & \begin{tabular}{c} $2^{13}\cdot 3^{2}\cdot 5\cdot 7^{2}\cdot 19\cdot 73$ (see table \ref{tab:linearsectionofGr27}) \end{tabular}\\
\hline
$SL_{6}(\Co)$ & $\mathrm{Gr}(3,6)$ & $\Oh(1)^{\oplus 5}$ & \begin{tabular}{c} $2^{17}\cdot 3^{11}\cdot 5^{2}$ (see table \ref{tab:linearsectionofGr36}) \end{tabular}\\
\hline
\end{tabular}
\label{tab:autogroupstable}
\end{table}
\end{center}

\begin{rmk}
Here we explain the exponent of $2$ in column 4 of table \ref{tab:autogroupstable} for quadrics. For $X=Q_{2n-1}$ the exponent is $\left\lfloor\frac{3n}{2}\right\rfloor+1=\left\lceil\frac{n}{2}\right\rceil+2\left\lfloor\frac{n}{2}\right\rfloor+1$. The first, respectively second summand in the right hand side comes from the linking classes $\Lk(Z)$ with the degree of $Z$ being $<\dim X$, respectively $>\dim X$. The last summand, 1, is there because of the 2 in the formula in Proposition \ref{degreediscrhypprojquad} for quadrics.

For $Q_{2n-2}, n$ odd the exponent is $\left\lfloor\frac{3n}{2}\right\rfloor=\frac{n-1}{2}+2\frac{n-1}{2}+1$ with each of the summands having the same origin as in the previous paragraph. A similar result holds for $Q_{2n-2}, n$ even: we have $\left\lfloor\frac{3n}{2}\right\rfloor=\frac{n}{2}-1+2(\frac{n}{2}-1)+1+2$. The source of the first three summands is the same as before, and the last one, 2, comes from the classes $\Lk(W)$ with $W$ of middle degree.
\end{rmk}

\begin{proof} Let us first prove the theorem for projective hypersurfaces. Let $G,X,E$ be as in the first row of table \ref{tab:autogroupstable}. Set $E'=J(E)$. The group $\widetilde{G}$ in this case is simply $\Co^*\times G$ by Lemma \ref{linebundletrivial}. We find using Proposition~\ref{propexplprojspace} that $\Lambda^*_{E',G}$ has finite index in $\fr{*}(G,\Z)=H^*(G,\Z)$, and that the index of $\Lambda^{\dim G}_{E',G}$ in $H^*(G,\Z)$ is $\prod_{i=2}^{n+1} \mathbf{m}(d,n,i)$. We deduce the analogues of these results for $\widetilde G$ from Corollary \ref{descentc*2ndrev} and Proposition \ref{degreediscrhypprojquad}. All statements of Theorem \ref{mainthmprojhyp} that involve the group $\widetilde G$ now follow from Corollary \ref{maincorappl} by setting $G'=\widetilde G$, and to prove the remaining statements involving $\bar G=PGL_{n+1}(\Co)$ we use Theorem \ref{thmquotslice} (applied to $U=Z_{\mathrm{reg}}$), Lemma~\ref{linebundletrivial} again and Remark \ref{autobarg}.

For hypersurfaces of quadrics and complete flag varieties the proof is very similar; instead of Proposition~\ref{propexplprojspace} we use analogous results from Sections~\ref{oddquad}, \ref{evenquad} and \ref{completeflagvarieties}. For higher rank bundles over Grassmannians we use the results of Section~\ref{fano} to calculate $\Lambda^*_{E',\widetilde{G}}$ for $E'=J(\Oh_{\PP(E^*)}(1))$, and to prove the result about the automorphism groups of the zero loci we use Corollary~\ref{Sum of line bundles} (instead of Lemma~\ref{linebundletrivial}) and Remark \ref{autobarg}.
\end{proof}

\begin{rmk}\label{preprint}
A version of theorem~\ref{mainthmprojhyp} for projective hypersurfaces (with somewhat worse bounds on the orders of the automorphism groups) first appeared in the preprint \cite{Gor12} by the first named author. This preprint is superseded by the present paper.
\end{rmk}

\begin{rmk}\label{pointsp1}
The zero locus of a regular section of $\Oh(d)$ over $\PP^1(\Co)$ is a $d$-element subset $Z\subset\PP^1(\Co)$. Theorem \ref{mainthmprojhyp} in this example tells us that the order of the projective automorphism group $\bar G_Z$ of $Z$ divides $d(d-1)(d-2)$. Here is an alternative proof of this communicated to us by F.\ Catanese: observe that $\bar G_Z$ acts freely on the set of ordered triples of elements of $Z$.
\end{rmk}

\begin{rmk}
For smooth hypersurfaces in projective spaces and non-degenerate quadrics, as well as for smooth complete intersections of two quadrics in $\PP^5(\Co)$ or $\PP^6(\Co)$ the finiteness of the automorphism groups 
is known, see \cite{Benoist13} and references therein. For regular sections of $\Oh(1)^{\oplus 5}$ over $\mathrm{Gr}(2,6)$ and of $\Oh(1)^{\oplus 5}$ over $\mathrm{Gr}(2,7)$ this result follows from \cite{KPS18}*{Corollary 4.3.5} and \cite{Manivel15} respectively, and in the remaining cases it seems to be new. The explicit bound on the orders of the automorphism groups given in Theorem \ref{mainthmprojhyp} is new except for smooth complete intersections of quadrics and in the cases mentioned in Remark \ref{knownlcm}.
\end{rmk}

\begin{rmk} For smooth complete intersections in characteristic 0 the projective automorphism group coincides with the full automorphism group, except for curves and K3 surfaces, see \cite{Benoist13}, Theorem 3.1 and the references after that theorem.

Similarly, every automorphism of the zero locus $Z$ of a regular section of $\Oh(1)^{\oplus 5}$ over the Grassmannian $\mathrm{Gr}(2,6)$ is induced by an element of $\aut(\mathrm{Gr}(2,6))=PSL_6(\Co)$. To see this note that $Z$ is a smooth Fano threefold of index $1$, Picard rank $1$ and genus $8$. By Theorem B.1.1 from~\cite{KPS18}, $Z$ can be equipped with a stable vector bundle $\mathcal{E}_2$ of rank~$2$ such that $\mathcal{E}_2$ is globally generated, $\dim \Gamma(Z,\mathcal{E}_2)=6$ and $\Lambda^2\mathcal{E}_2\cong \Oh(1)$. Moreover, by Proposition B.1.5 ibid.\ any two such bundles are isomorphic. Therefore $\mathcal{E}_2$ is an $\aut(Z)$-equivariant vector bundle. Since $\mathcal{E}_2$ is globally generated, it defines an $\aut(Z)$-equivariant closed embedding $Z\hookrightarrow \mathrm{Gr}(2,\Gamma(Z,\mathcal{E}_2))\subset \PP^{14}$ such that the image is the transversal intersection of $\mathrm{Gr}(2,\Gamma(Z,\mathcal{E}_2))$ and a projective subspace of codimension $5$. So every automorphism of $Z$ can be extended to an automorphism of $\mathrm{Gr}(2,6)$.

Note also that the image of $SO_{k+2}(\Co)$ in $\aut (Q_k)$ is the whole group if $k$ is odd and has index 2 if $k$ is even. As a consequence, the order of the full automorphism group of a regular section of the bundle $\Oh(d),d\geq 3$ over $Q_k$ divides the expression in column 4 of table \ref{tab:autogroupstable}.
\end{rmk}

\begin{rmk}
V.~Gonz\'alez-Aguilera and A.~Liendo show in \cite{GAL13}*{Theorem~1.3} 
that if $p$ is a prime number that does not divide $d$ or $d-1$ and $p^l$ is the order of an automorphism of a smooth degree $d$ hypersurface $Z$ of $\PP^n(\Co)$, then $p^l$ divides the least common multiple of $\mathbf{m}(d,n,i)$ for $i=2,\ldots, n+1$. This implies that the order of any abelian subgroup of the projective automorphism group $\paut(Z)=PGL_{n+1}(\Co)_Z$ of $Z$ divides $$\left( \mathrm{LCM}_{i=2}^{n+1} \mathbf{m}(d,n,i)\right)^n,$$ as every such subgroup is included in some maximal torus of $PGL_{n+1}(\Co)$, which has dimension $n$.

Recall that by Jordan's theorem there is an integer $J(n+1)$, called the {\it Jordan constant for $GL_{n+1}(\Co)$}, such that every finite subgroup of $GL_{n+1}(\Co)$ has an abelian normal subgroup of index $\leq J(n+1)$; moreover, $J(n+1)= (n+2)!$ for $n\geq 70$, see \cite{Col07} and references therein. If we drop the requirement that the abelian subgroup should be normal, we get another constant, which we will call the {\it weak Jordan constant} following \cite{PS17} and denote $J_w(n+1)$. Note that we have $J(n+1)\geq J_w(n+1)\geq \sqrt{J(n+1)}$ by the Chermak-Delgado theorem (see e.g.\ \cite{Isaacs08}*{Theorem 1.41}). The same results hold for $PGL_{n+1}(\Co)$, with the same constants $J(n+1)$ and $J_w(n+1)$.

So using the results of \cite{GAL13} 
one can deduce that if $p$ does not divide $d(d-1)$, then the order of the Sylow $p$-subgroup of $\paut(Z)$ divides
$$\left( \mathrm{LCM}_{i=2}^{n+1} \mathbf{m}(d,n,i)\right)^n p^{\left\lfloor\log_p J_w(n+1)\right\rfloor}.$$
Note that under the same assumptions on $p$ Theorem \ref{mainthmprojhyp} gives a better bound, namely $\prod_{i=2}^{n+1} \mathbf{m}(d,n,i)$.

Let us briefly explain what happens when $p$ divides $d$ or $d-1$. We call a projective automorphism of $Z$ {\it $F$-liftable} if it has a lift to $GL_{n+1}(\Co)$ that preserves an equation of $Z$. In~\cite{GAL20}*{Theorem 2.1} the authors extend their previous results for $p\nmid d(d-1)$ to $F$-liftable automorphisms of order $p^l$, $p|d(d-1)$. The answer is different, but still can be given in terms of the numbers $\mathbf{m}(d,n,i)$. Assume now that every element of the Sylow $p$-subgroup of $\paut(Z)$ is $F$-liftable. Then a similar argument as above shows that Theorem~\ref{mainthmprojhyp} still gives a better bound, with one potential exception: if $p$ divides both $d$ and $n+1$, then it is not clear how to compare our bound on $|\paut(Z)|$ with the one obtained by applying the results of \cite{GAL20}; also, we do not know whether or not smooth hypersurfaces whose projective automorphism group has all these properties in fact exist.

Finally, we note that in~\cite{Zheng20} Z.~Zheng obtained results similar to those in~\cite{GAL13} and~\cite{GAL20} that we have just described. More precisely, suppose $p^l$ is the order of a (not necessarily $F$-liftable) automorphism of $Z$; then \cite{Zheng20}*{Theorem 1.5} gives one a bound on $l$ that is identical to the one by V.~Gonz\'alez-Aguilera and A.~Liendo (see above) except when $p$ divides $d$ but does not divide $d-1$, in which case it is slightly weaker.

\end{rmk}

\begin{rmk}\label{knownlcm}
In some cases there is a classification of the automorphism groups of smooth degree $d$ hypersurfaces in $\PP^n(\Co)$, which allows one to find the least common multiple of the orders. The examples that we know of are given in the following table, where we set $U_{d,n}=\reg{\PP^n(\Co),\Oh(d)}$:
\begin{center}
\begin{table}[h]
\begin{tabular}{|c|c|c|c|}\hline
$(n,d)$ & \begin{tabular}{c} Multiple of $|\bar G_{Z(s)}|,$\\ $s\in U_{d,n}$ from Theorem \ref{mainthmprojhyp}\end{tabular} &\begin{tabular}{c} LCM of $|\bar G_{Z(s)}|,$\\ $s\in U_{d,n}$\end{tabular} & References\\
\hline
$(2,3)$ & $2^3\cdot 3^3$ &  $2^2\cdot 3^3$ & \cite{Dolgachev12}*{section 3}\\
\hline
$(2,4)$ & $2^5\cdot 3^4\cdot 7$ & $ 2^5\cdot 3^3\cdot 7$ &\cite{Dolgachev12}*{section 6}\\
\hline
$(3,3)$ & $2^6\cdot3^4\cdot 5$ & $2^3\cdot 3^4\cdot 5$ &\cite{Hosoh97}, \cite{Segre42}, \cite{Dolgachev12}*{section 9}\\
\hline
\end{tabular}
\end{table}
\end{center}
\end{rmk}

\begin{rmk}\label{m5}
The moduli space $\mathcal{M}_5$ of smooth complex genus 5 curves is the union of the hyperelliptic locus, the trigonal locus and the locus $\mathcal{M}^{\mathrm{gen}}_5$ of curves that are complete intersections of three quadrics in $\PP^4(\Co)$. For curves of the last type the canonical class is the hyperplane section, so the projective automorphism group $\subset PGL_5(\Co)$ coincides with the full automorphism group. Using Theorem~\ref{mainthmprojhyp} we obtain an isomorphism
\begin{equation}\label{isom5}
H^*(U_{(2,2,2),4},\Q)\cong H^*(SL_5(\Co)\times GL_3(\Co),\Q)\otimes H^*(\mathcal{M}^{\mathrm{gen}}_5,\Q)
\end{equation}
where $U_{(2,2,2),4}=\reg{\PP^4(\Co),\Oh(2)^{\oplus 3}}$. This isomorphism suggests a strategy for calculating $H^*(\mathcal{M}_5,\Q)$. Namely, one could start by calculating $H^*(U_{(2,2,2),4},\Q)$ (using conical resolutions or otherwise), and then deduce $H^*(\mathcal{M}^{\mathrm{gen}}_5,\Q)$ using~\ref{isom5}. The rational cohomology of the trigonal locus has been calculated by A.~Zheng \cite{AngelinaZheng20}, and that of the hyperelliptic locus is trivial. Putting all this together we get a spectral sequence with known $E_1$ that converges to $H^*_c(\mathcal{M}_5,\Q)$. It is reasonable to expect that all differentials will be zero or at least tractable, cf.\ the case of genus 4 considered in~\cite{Tom05}. 

Implementing this plan is beyond the scope of this paper, but here is an illustration. Note that $U_{(2,2,2),4}$ is the complement of a hypersurface in an affine space, so the reduced rational cohomology of $U_{(2,2,2),4}$ has no classes of weight equal the degree. By~\ref{isom5} the same is true for $\mathcal{M}^{\mathrm{gen}}_5$; combining this with \cite{AngelinaZheng20}*{Corollary~1.2} we conclude that $W_i H^i(\mathcal{M}_5,\Q)=\Q(-i/2)$ for $i=0,2,4,6$, and 0 otherwise.
\end{rmk}

\begin{rmk}\label{m8}
Similarly, a general curve of genus 8 is the intersection of the Pl\"ucker embedded $\mathrm{Gr}(2,6)$ and a transversal projective subspace of codimension 7. It follows from the results of S.~Mukai \cite{Mukai93}*{Theorem~A} that the embedding is uniquely determined by the curve up to an automorphism of the Grassmannian. So if we denote the locus of such curves by $\mathcal{M}^{\mathrm{gen}}_8\subset\mathcal{M}_8$, then by Theorem~\ref{mainthmprojhyp} we have an isomorphism
$$H^*(\reg{\mathrm{Gr}(2,6),\Oh(1)^{\oplus 7}},\Q)\cong H^*(SL_6(\Co)\times GL_7(\Co),\Q)\otimes H^*(\mathcal{M}^{\mathrm{gen}}_8,\Q).$$
\end{rmk}

\section{Conjectures}\label{secconjectures}

In this section we propose and discuss a few conjectures. Below $G$ is a complex semi-simple group with finite fundamental group, $P\subset G$ is a parabolic subgroup, $X=G/P$, and $T\subset P$ is a maximal torus of $G$.

\subsection{Genericity conjectures}\label{genconj}
\begin{conj}[Genericity conjecture, weak version]
For every $G$ and $P\subsetneq G$ there is a $G$-equivariant line bundle $L$ on $G/P$ such that the division theorem (see the beginning of Section \ref{sectapplications}) holds for the action of $G$ on $\reg{X,L}$.
\end{conj}

The character group $\mathfrak{X}(P)$ of $P$ is a subgroup of $\mathfrak{X}(T)\cong\Z^{\dim T}$. We will say that a character $\chi\in \mathfrak{X}(P)$ is {\it exceptional} if the division theorem does not hold for the action of $G$ on $\reg{X,G\times_{\chi}\Co}$. (Note that if $\Gamma(X,G\times_{\chi}\Co)=0$, then $\reg{X, G\times_{\chi}\Co}=\varnothing$, and for this space the division theorem is vacuously true.) It follows from our results in Sections \ref{genalg} and \ref{completeflagvarieties} that if the weak form of the genericity conjecture is true, then exceptional elements form a proper Zariski closed subset of $\mathfrak{X}(P)$, see Remark \ref{polynomial}.

\begin{conj}[Genericity conjecture, strong version]
For every $G$ and $P\subsetneq G$ there are finitely many exceptional elements in $\mathfrak{X}(P)$.
\end{conj}

\begin{rmk}\label{pp1}
Suppose $P_1\supset P$ is another parabolic subgroup and set $X_1=G/P_1$. Let $p:X\to X_1$ be the projection. On the one hand, if the strong genericity conjecture holds for $X$, then it also holds for $X_1$. On the other hand, suppose $L_1$ is a $G$-equivariant line bundle over $X_1$ such that the ring $\Lambda^*_{J(L_1),G}$ (see Notation \ref{lambdae}) has finite index in $\fr{*}(G,\Z)$. We claim that then so does $\Lambda^*_{J(L),G}$ where $L=p^*(L_1)$. Let us sketch a proof. Note that there is a $G$-equivariant exact sequence
$$0\to p^*(J(L_1))\to J(L)\to \Omega_{X/X_1}\to 0$$ where $\Omega_{X/X_1}$ denotes the relative cotangent bundle. We identify both $H^*(BP,\Q)$ and $H^*(BP_1,\Q)$ with subrings of $H^*(BT,\Q)$, cf.\ Section \ref{projquad}, and deduce using formula \ref{formulas1} that in the rational cohomology of $G\times X$ we have $$S\big(e_G(J(L))\big)= (\id\times p)^*\Big(S\big(e_G(J(L_1))\big)\Big)\smile(1\times e(\Omega_{X/X_1})).$$ The map $H^*(X_1,\Q)\to H^{*+2\dim X-2\dim X_1}(X,\Q)$ given by $x\mapsto p^*(x)\smile e(\Omega_{X/X_1})$ is injective, which implies the claim.
\end{rmk}

By Theorem~\ref{mainthmprojhyp}, the strong genericity conjecture is true for projective spaces and non-degenerate quadrics, and one can prove it for the complete flag variety of $G=SL_3(\Co)$ as follows. By inspection, there are finitely many exceptions when $m=n$ or $m,n\in\{0,1\}$. For $w\in W$, the Weyl group of $G$, set $S_w=S(e_G(E'))/e_w=O^*(\Lk(e_w))$, and make the substitution $m'=m-1,n'=n-1$. Then all coefficients of $S_{\sigma_1 w_0}+S_{\sigma_2 w_0}$ will be negative, which implies $\mathbf{m}(\chi,SL_3(\Co),2)\neq 0$ for $m,n>1$. Moreover, $\frac{1}{m'-n'}\left(S_{\sigma_1}+S_{\sigma_2}\right)$ becomes a polynomial with positive coefficients, so $\mathbf{m}(\chi,SL_3(\Co),3)\neq 0$ when $m,n>1$ and $m\neq n$. These observations indicate that there may be a positive formula (cf.\ \cite{CW97}) for the analogues of $\mathbf{m}(\chi,SL_3(\Co),2)$ and $\mathbf{m}(\chi,SL_3(\Co),3)$ for general $G$.

Using Remark \ref{pp1} we can now extend both the strong and weak versions of the conjecture to more examples. For instance, the strong genericity conjecture is true for $G=SL_3(\Co)$ and $P$ arbitrary, and the weak one for $G=SL_{n+1}(\Co)$ and $P$ such that $G/P$ admits a $G$-map to $\PP^n(\Co)$.

\subsection{Parsimony conjecture}\label{parsimconj}

\begin{conj}[Parsimony conjecture]
Let $L$ be a $G$-equivariant line bundle over $X$, and set $E'=J(L)$. Let~$\mathbf{m}$, respectively $\delta$ denote the index of $\Lambda^{\dim G}_{E',G}$ (see Notation \ref{lambdae}) in $H^{\dim G}(G,\Z)$, respectively the integer $c_{\dim X}(E')([X])$. Suppose $\reg{X,L}\neq\varnothing,\mathbf{m}<\infty$ and $\delta\mathbf{m}\neq 0$. Then we conjecture that the number $\frac{\delta\mathbf{m}}{|Z(G)|}$, which is an integer, is parsimonious in the following sense: every prime that occurs in $\frac{\delta\mathbf{m}}{|Z(G)|}$ can be realised as the order of a $G$-automorphism of $Z(s)$ for some $s\in\reg{X,L}$.
\end{conj}

The procedure for calculating $\mathbf{m}$ for any given $G,P,L$ is described in Sections \ref{genalg} and \ref{completeflagvarieties}.

As we mentioned in the Introduction, the parsimony conjecture is true for $G=SL_{n+1}(\Co), X=\PP^n(\Co), L=\Oh(d),d\geq 3$ by the results of V.~Gonz\'alez-Aguilera and A.~Liendo \cite{GAL13}, and this is the only infinite family of examples we are aware of.

\appendix

\section{}\label{app_proofs}

\subsection{Proof of Proposition \ref{transgression}} Let $Y$ be a topological space. We let $C_*(Y)$ denote the integral singular chain complex of $Y$, and we set $\widetilde{C}_*(Y)$ to be the kernel of the natural augmentation $C_*(Y)\to\Z$. We will denote the singular cochain complex $\Hom(C_*(Y),R)$ of $Y$ by $C^*(Y,R)$. {Recall that the differential in $C^*(Y,R)$ is the dual of the differential in $C_*(Y)$, see Remark~{\ref{consistency}}; we will use $\delta$ to denote the differentials in singular cochain complexes.}
%

\begin{lmm}\label{directlimtrees}
A path-connected polyhedron $X$ is set-theoretically the direct limit of a family $\mathcal{T}$ of subspaces homeomorphic to trees.
\end{lmm}

\begin{proof} Let $T_0$ be a maximal subtree of the 1-skeleton of $X$. For each face $F$ of $X$ we choose a vertex $v(F)\in T_0\cap F$. If $x\in X$, we let $I(x)$ be the closed line segment that joins $x$ and $v(F)$ where $F$ is the unique face such that $x$ belongs to the interior of $F$. The required family $\mathcal{T}$ of subspaces of $X$ is formed by $T_0\cup\bigcup_{x\in X'} I(x)$ where $X'\subset X$ is an arbitrary finite subset.
\end{proof}

\begin{lmm}\label{zeroondegenerate}
Let $X$ and $\mathcal{T}$ be as in the previous lemma. Then every closed singular cochain $\varepsilon'\in C^{m}(X,R),m\geq 2$ is homologous to a cochain $\varepsilon$ that is zero on every singular simplex $\in C_*(T), T\in\mathcal{T}$.
\end{lmm}

\begin{proof} Set $\widetilde{C}_*=\varinjlim_{T\in \mathcal{T}} \widetilde{C}_*(T)$. This is a complex of free $R$-modules that is acyclic in all degrees, so the zero and identity endomorphisms of $\widetilde C_*$ are chain homotopic, which implies that the same is true for $\Hom(\widetilde{C}_*,R)$. We now apply the long exact sequence of
$$0\to \Hom(\widetilde{C}_*(X)/\widetilde{C}_*,R)\to \Hom(\widetilde{C}_*(X),R)\to \Hom(\widetilde{C}_*,R)\to 0.$$
%
\end{proof}

{\it Proof of the proposition.} Let $p:\Tot(E)\to X$ be the projection; recall that $p_0$ denotes the restriction of $p$ to $\Tot_0(E)$. 
After pulling $E$ back if necessary to a space weakly equivalent to $X$ we may assume that the base $X$ is a polyhedron. By the construction of the transgression given e.g.\ in \cite{Mc01}*{Section~6.2} there are singular cochains $\psi'\in C^{r-1}(\Tot_0(E), R)$ and $\varepsilon'\in C^r(X,\{x_0\},R)$ such that $\delta \psi'=p_0^*(\varepsilon')$ and $\psi'$ extends a cocycle $\in C^{r-1}(E_{0,x_0},R)$ that represents the preferred generator of $H^{r-1}(E_{0,x_0},R)$. The cohomology class of $\varepsilon'$ will then be $d_r(\mathbf{a})$.

Applying Lemmas \ref{directlimtrees} and \ref{zeroondegenerate} we find a cochain $\varepsilon''\in C^{r-1}(X,R)$ such that $\varepsilon=\varepsilon'+\delta\varepsilon''$ is zero on every singular simplex in every $T\in\mathcal{T}$. Then the restriction of the cochain $\psi=\psi'+p_0^*(\varepsilon'')$ to every fibre of $p_0$ is closed, and $\psi|_{E_{0,x_0}}=\psi'|_{E_{0,x_0}}$. Moreover, if $x\in X$, then using the lifting homotopy property we can extend the identity of $E_{0,x_0}$ to a map $E_{0,x_0}\times [0,1]\to \Tot_0(E)$ that covers a path from $x_0$ to $x$ in some $T\in\mathcal{T}$. We conclude that the restriction of $\psi$ to the fibre $E_{0,x}$ represents the preferred generator of $H^{r-1}(E_{0,x},R)\cong H^r(E_{x},E_{0,x},R)$. Note also that $\delta\psi=p_0^*(\varepsilon)$ naturally extends as a closed cochain $p^*(\varepsilon)$ on the whole of $\Tot(E)$.

The rest of the proof is adapted from the construction of smooth representatives of the Thom classes \cite{BottTu82}*{\S 12}. Let $\rho(x)$ be $-1$ for $x\in[0,1]$ and $0$ for $x\in(1,\infty)$. We introduce a metric on $E$ and denote the resulting norm by $||-||$. By setting $x\mapsto \rho(||x||)$ we get a function $\Tot(E)\to R$, which we also denote by $\rho$ and view as a singular $0$-cochain. Set $U_0=\{y\in\Tot(E)\mid ||y||\neq 0\}$, $U_1=\{y\in\Tot(E)\mid ||y||<1\}$, and $U_2=\{y\in\Tot(E)\mid ||y||>1\}$. The cochain $\delta(\rho\psi)=(\delta\rho)\smile\psi+\rho\delta\psi=
(\delta\rho)\smile\psi+\rho p^*_0(\varepsilon)\in C^*(U_0,R)$ extends as $\rho p^*(\varepsilon)$ to singular simplices in $U_1$. In fact, this extension belongs to $$C^*(U_0+U_1,U_2+\varnothing,R)=\{f\in \Hom(C_*(U_0)+C_*(U_1),R)\mid f|_{C_*(U_2)}=0\},$$ and by abuse of notation we will denote it also by $\delta(\rho\psi)$. We have an exact sequence of cochain complexes (cf.\ \cite{H02}*{\S 3.1})
$$0\to C^*(U_0+U_1,U_2+\varnothing,R)\to C^*(U_0+U_1,R)\to C^*(U_2,R)\to 0$$
where $C^*(U_0+U_1,R)=\Hom(C_*(U_0)+C_*(U_1),R)$, and $$H^*(C^*(U_0+U_1,U_2+\varnothing,R))\cong H^*(\Tot(E),U_2,R)\cong H^*(\Tot(E),\Tot_0,R).$$
The cochain $\delta(\rho\psi)$ is closed in $C^*(U_0+U_1,U_2+\varnothing,R)$: the restriction of $\delta(\rho\psi)$ to $U_0$ is exact, and after restricting to $U_1$ we get the closed cochain $\rho p^*(\varepsilon)$. Moreover, the restriction of $\delta(\rho\psi)$ to each fibre $E_{x}$ represents the preferred generator of $H^r(E_{x},E_{x}\cap U_2,R)\cong H^r(E_{x},E_{0,x},R)$. To see this recall that $\psi$ restricted to $E_{x}\cap U_2$ is closed and represents the preferred generator; the cochain $(1+\rho)\psi$ is an extension of $\psi|_{U_2}$ to $C^*(U_0+U_1,R)$. So $\delta\big(((1+\rho)\psi)|_{E_{x}}\big)=(\delta(\rho\psi))|_{E_{x}}$ represents the preferred generator of $H^r(E_{x},E_{x}\cap U_2,R)$ as we claimed.

We conclude that $\delta(\rho\psi)$ represents the Thom class of $E$, and $-\varepsilon$, the restriction of $\delta(\rho\psi)$ to the zero section, represents the Euler class.\qed

\subsection{Proof of Proposition \ref{degreediscr}}

\begin{lmm}\label{analytic_transversal}
Let $U_1$ and $U_2$ be open subsets of $\Co^m$ and $\Co^n$ respectively, and let $A$ be a smooth analytic subvariety of $U_1\times U_2$ of dimension $m$. Suppose $A$ intersects every subset $\{x\}\times U_2,x\in U_1$ at one point. Then all these intersections are transversal.
\end{lmm}

\begin{proof}
The projection $A\to U_1$ is a holomorphic bijection between smooth complex analytic varieties, so it is an isomorphism, which implies that the tangent space to $A$ at any point does not intersect $0\oplus\Co^n$.
\end{proof}

Let $Y$ be a smooth complex analytic variety, and let $V$ be a complex vector space. The projective space $\PP=\PP(V\oplus\Co)$ contains $V$ as an open subset. Suppose that $Z$ is a vector subbundle of the trivial vector bundle $V\times Y\to Y$, and that $W\subset V$ is an affine subspace of dimension $l$ such that $0\not\in W$ and $\overline{\Tot Z}$ intersects $\overline{W\times Y}$ transversally at one point $(s_0,y_0)\in W\times Y$, where the bar denotes the closure inside $\PP\times Y$. Let $\tilde s_1,\ldots, \tilde s_{l+1}\in W$ be elements that generate $W$ as an affine subspace, and let $s_1,\ldots,  s_{l+1}$ be the corresponding sections of the quotient bundle $Z'=(V\times Y)/Z$. Then $s_i$ are linearly dependent only at $y_0$, and their values at this point span a vector subspace of dimension $l$.

Let ${\tilde D}_{l+1}\subset\Tot(Z'{}^{\oplus (l+1)})$ be the {\it linear dependency locus}; its fibre over $y\in Y$ is the set of all $(l+1)$-tuples of elements of $Z'_y$ that are linearly dependent. Note that ${\tilde D}_{l+1}$ is an irreducible analytic subvariety of $\Tot(Z'{}^{\oplus (l+1)})$. It follows from our assumptions that the image of $s=(s_1,\ldots, s_{l+1})$ intersects ${\tilde D}_{l+1}$ at one point. Moreover, if $U$ is a compact neighbourhood of $y_0$, then the same will be true for $s'(U)$ for $s'$ an $(l+1)$-tuple of sections $\in\Gamma (Y,Z')$ that is sufficiently close to $s$: any intersection point $\in s'(U)\cap {\tilde D}_{l+1}$ gives one an intersection point $\in (W'\times U)\cap \Tot(Z)$ where $W'$ is the affine subspace of $V$ spanned by the components of $s'$.

\begin{lmm}\label{localcalc}
The intersection $s(Y)\cap {\tilde D}_{l+1}$ is transversal.
\end{lmm}

\begin{proof}
Let $W'\subset V$ be a vector subspace that maps isomorphically to $Z'_{y_0}$, and take a neighbourhood $U'\subset W'$ of the origin and a neighbourhood $U$ of $y_0$ in $Y$. The map $F:{U'}^{\times(l+1)}\times U\to\Tot Z'$ given by $$(\varepsilon_1,\ldots, \varepsilon_{l+1},y)\mapsto (s_1(y)+\varepsilon_1(y),\ldots, s_{l+1}(y)+\varepsilon_{l+1}(y))$$ is a local analytic isomorphism at $(0,y_0)$. Moreover, if $U$ and $U'$ are small enough, then for every $\varepsilon\in {U'}^{\times(l+1)}$ the image $F(\{\varepsilon\}\times U)$ intersects ${\tilde D}_{l+1}$ at a single point. Using Lemma~\ref{analytic_transversal} we see that all these intersections are transversal.
\end{proof}

{\it Proof of the proposition.} Let $p_V:V\times X\to V$ be the projection. We may and will assume that every point of $X'$ is smooth. Since $p_V$ is proper and $\widetilde\Sigma$ is closed in $V\times X$, we see that $\Sigma=p_V(\widetilde\Sigma)\subset V$ is closed and has dimension $\leq d+\dim V-r=\dim\widetilde\Sigma$. Set $\widetilde\Sigma'=\widetilde\Sigma|_{X'}, q=p_V|_{\widetilde\Sigma}$ and
$$\Theta=p_V\left(\widetilde\Sigma\setminus \widetilde\Sigma'\right)\cup p_V\left(\{y\in\widetilde\Sigma'\mid \ker dq|_{y}\neq 0\}\right).$$ Note that $\dim\Theta<d+\dim V-r$.

Suppose first that $\dim\Sigma= d+\dim V-r$. Let $W$ be an affine subspace of $V$ of dimension $r-d$ that $\not\ni 0$ and that intersects $\Sigma$ transversally at $\deg\Sigma$ smooth points $\not\in \Theta$. Set $\PP=\PP(V\oplus\Co)$. Note that the closures of $\Sigma$ and $W$ in $\PP$ have no intersection points at infinity, i.e.\ in $\PP\setminus V$. It follows from our assumptions that for any point in $\Sigma\cap W$ the preimage under $q$ consists of $k> 0$ transversal intersection points of $W\times X$ and $\widetilde\Sigma$, which all belong to $\widetilde\Sigma'$. Choose elements $s_1,\ldots, s_{r-d+1}\in W$ that generate $W$ as an affine subspace. Applying Lemma~\ref{localcalc} we see that these sections are linearly dependent at $k\deg\Sigma$ points of $X'$, and at each of these the tuple $s=(s_1,\ldots,s_{r-d+1})$ is transversal to the linear dependency locus. We conclude that, after pulling $E'$ and $s$ back to a resolution of singularities of $X$ if necessary, the class of the degeneracy cycle $D_{r-d+1}$ for $E'$ (see e.g.\ \cite{GH78}*{Chapter 3, \S 3}) is $k\deg\Sigma\in H_0(X,\Z)$. We now use the fact that this class is Poincar\'e dual to~$c_d(E')$ (see ibid.). 

If $\dim\Sigma<d+\dim V-r$, then $q^{-1}$ of every point of $\Sigma$ has positive dimension, and one can choose an affine subspace $W\subset V$ of dimension $r-d$ such that $\Sigma\cap W=\varnothing$, which by the above argument implies that $c_d(E)=0$.

\section{Tables}\label{app_tables}

\subsection{Examples for Sections~\ref{projquad} and \ref{completeflagvarieties}}

Tables~\ref{tab:projectivehypersurfaces}-\ref{tab:completeflagsg2} contain the multiples of $|\bar G_{Z(s)}|,s\in\reg{X,E}$ given by Theorem~\ref{mainthmprojhyp} for varieties considered in Sections~\ref{projquad} and \ref{completeflagvarieties} for small values of the parameters. (Recall that $\bar G$ is the image of $G$ in $\aut X$, and $\bar G_Z$ is the subgroup of $\bar G$ that preserves $Z$ but not necessarily pointwise, see Section \ref{automorphismgroups}.) In table~\ref{tab:hypersurfacesofquadrics} $Q_n$ is a non-singular quadric in $\PP^{n+1}(\Co)$; the notation used in tables~\ref{tab:completeflagssl3}-\ref{tab:completeflagsg2} was introduced in Section~\ref{completeflagvarieties}.

\begin{center}
\begin{table}[h]
	\caption{$G=SL_{n+1}(\Co)$, $X=\PP^n(\Co)$, $E=\Oh(d)$}
    \begin{tabular}[c]{|c|c|c|c|c|c|c|c|c|} \hline
	\diagbox{$d$}{$n$} & $2$ & $3$ & $4$  \\ \hline
	$3$ &$2^{3}\cdot 3^{3} $  &$2^{6}\cdot 3^{4}\cdot 5 $  &$2^{10}\cdot 3^{5}\cdot 5\cdot 11 $ \\ \hline 
	$4$ &$2^{5}\cdot 3^{3}\cdot 7 $  &$2^{9}\cdot 3^{6}\cdot 5\cdot 7 $  &$2^{11}\cdot 3^{10} \cdot 5\cdot 7\cdot 61 $ \\ \hline
	$5$ &$2^{6}\cdot 3 \cdot 5^{2} \cdot 13$  &$2^{12}\cdot 3^{2}\cdot 5^{3} \cdot 13 \cdot 17 $  &$2^{20}\cdot 3^{2}\cdot 5^{5}\cdot 13\cdot 17 \cdot 41 $ \\ \hline
    \end{tabular}
    \label{tab:projectivehypersurfaces}
\end{table}
\end{center}

\begin{center}
\begin{table}[h]
	\caption{$G=SO_{n+2}(\Co)$, $X=Q_{n}$, $E=\Oh(d)$}
    \begin{tabular}[c]{|c|c|c|c|c|c|c|c|c|} \hline
	\diagbox{$d$}{$n$} & $2$ & $3$ & $4$  \\ \hline
	$2$ &
	$2^{7}\cdot 3 $  &
	$2^{8}\cdot 5 $  &
	$2^{8}\cdot 3^{2}\cdot 5 $ \\ \hline 
	
	$3$ &
	$2^{5}\cdot 3^{2}\cdot 5 \cdot 17 $  &
	$2^{8}\cdot 3^{2}\cdot 5\cdot 7^{2} $  &
	$2^{10}\cdot 3^{5} \cdot 5\cdot 7\cdot 43 $ \\ \hline

	$4$ &
	$2^{9}\cdot 3^{2} \cdot 5 \cdot 17$  &
	$2^{10}\cdot 3^{4}\cdot 5 \cdot 71 $  &
	$2^{11}\cdot 3^{6}\cdot 5 \cdot 7 \cdot 13 \cdot 547 $ \\ \hline
    \end{tabular}
    \label{tab:hypersurfacesofquadrics}
\end{table}
\end{center}

\begin{center}
\begin{table}[h]
	\caption{$G=SL_3(\Co)$, $X=SL_3(\Co)/B$, $E=\Oh(m\omega_1+n\omega_2)$}
    \begin{tabular}[c]{|c|c|c|c|c|c|c|c|c|} \hline
	\diagbox{$m$}{$n$} & $1$ & $2$ & $3$ & $4$ \\ \hline
	$1$ &
	$0 $  &
	$2^{6}\cdot 3 \cdot 7 $  &
	$2^{5}\cdot 3^{3} \cdot 13$  &
	$2^{9}\cdot 3^{3}\cdot 5\cdot 7 $ \\ \hline 
	
	$2$ &
	$2^{6}\cdot 3 \cdot 7 $  &
	$2^{5}\cdot 3^{6}\cdot 5 $  &
	$2^{11}\cdot 3$  &
	$2^{5}\cdot 3 \cdot 5^{3}\cdot 11 $ \\ \hline

	$3$ &
	$2^{5}\cdot 3^{3} \cdot 13$  &
	$2^{11}\cdot 3 $  &
	$2^{8}\cdot 3^{5}\cdot 5 \cdot 7 \cdot 13 $  &
	$2^{6}\cdot 3^{4} $ \\ \hline

	$4$ &
	$2^{9}\cdot 3^{3}\cdot 5\cdot 7$  &
	$2^{5}\cdot 3 \cdot 5^{3}\cdot 11 $  &
	$2^{6}\cdot 3^{4} $  &
	$2^{6}\cdot 3^{5}\cdot 5^{2} \cdot 7 \cdot 13 \cdot 31 $ \\ \hline
    \end{tabular}
    \label{tab:completeflagssl3}
\end{table}
\end{center}

\begin{center}
\begin{table}[h]
	\caption{$G=Sp_4(\Co)$, $X=Sp_4(\Co)/B$, $E=\Oh(m\omega_1+n\omega_2)$}
    \begin{tabular}[c]{|c|c|c|c|c|c|c|c|c|} \hline
	\diagbox{$m$}{$n$} & $1$ & $2$ & $3$ & $4$ \\ \hline
	$1$ &
	$2^{3}\cdot 3\cdot 5 $  &
	$2^{8}\cdot 3^{3} \cdot 5 \cdot 7 \cdot 13$  &
	$2^{3}\cdot 3^{6} \cdot 5^{2} \cdot 13 \cdot 17$  &
	$2^{13}\cdot 3 \cdot 5^{2} \cdot 31 \cdot 41 $ \\ \hline 
	
	$2$ &
	$2^{8}\cdot 5 \cdot 37 $  &
	$2^{11}\cdot 3^{2}\cdot 5 \cdot 107$  &
	$2^{10}\cdot 3^{4} \cdot 5^{4}$  &
	$2^{14}\cdot 5^{2}\cdot 683 $ \\ \hline

	$3$ &
	$2^{3}\cdot 3^{2} \cdot 5 \cdot 7 \cdot 11 \cdot 17$  &
	$2\cdot 5^{2} \cdot 67 $  &
	$2^{3}\cdot 3^{2}\cdot 5 \cdot 709$  &
	$2^{8}\cdot 3^{2} \cdot 5 \cdot 7 \cdot 13 $ \\ \hline

	$4$ &
	$2^{8}\cdot 3^{5}\cdot 5\cdot 13$  &
	$2^{9}\cdot 3^{4} \cdot 5\cdot 7^{2}\cdot 59 $  &
	$2^{8}\cdot 3^{5} \cdot 5 \cdot 97 $  &
	$2^{11}\cdot 3^{3}\cdot 5 \cdot 19 $ \\ \hline
    \end{tabular}
    \label{tab:completeflagssp4}
\end{table}
\end{center}

\begin{center}
\begin{table}[H]
	\caption{$G=G_2$, $X=G_2/B$, $E=\Oh(m\omega_1+n\omega_2)$}
    \begin{tabular}[c]{|c|c|c|c|c|c|c|c|c|} \hline
	\diagbox{$m$}{$n$} & $1$ & $2$ & $3$ & $4$ 
	\\ \hline
	$1$ \rule{0pt}{9pt}&
	$2^{8}\cdot 3\cdot 7\cdot 229 $  &
	$\begin{array}{c} \rule{0pt}{9pt} 2^{10}\cdot 3^{2} \cdot 7 \\ \cdot 13 \cdot 19 \cdot 89\end{array}$  &
	$\begin{array}{c} \rule{0pt}{9pt} 2^{8}\cdot 3^{2} \cdot 7 \cdot 13 \\ \cdot 37 \cdot 163 \cdot 239\end{array}$ & $\begin{array}{c} \rule{0pt}{9pt}2^{9}\cdot 3^{2} \cdot 7^{2}\\ \cdot 61 \cdot 71 \cdot 2213 \end{array}$ 
	\\ \hline 
	
	$2$ \rule{0pt}{9pt}&
	$\begin{array}{c} \rule{0pt}{9pt} 2^{8}\cdot 3^{2} \cdot 7 \\ \cdot 13 \cdot 2473 \end{array}$  &
	$\begin{array}{c} \rule{0pt}{9pt} 2^{6}\cdot 3^{3}\cdot 5^{2} \cdot 7 \\ \cdot 19 \cdot 73 \cdot 467\end{array}$  &
	$\begin{array}{c} \rule{0pt}{9pt}2^{10}\cdot 3^{3} \cdot 7 \\ \cdot 37 \cdot 20333\end{array}$ & 
	$\begin{array}{c} \rule{0pt}{9pt} 2^{6}\cdot 3^{2}\cdot 5^{3}\cdot 7\\ \cdot 61 \cdot 199 \cdot 641 \end{array}$ 
	\\ \hline

	$3$ \rule{0pt}{9pt}&
	$\begin{array}{c} \rule{0pt}{9pt} 2^{9}\cdot 3^{4} \cdot 5\\ \cdot 7 \cdot 787 \end{array}$  &
	$2^{8}\cdot 3^{8} \cdot 7 \cdot 4273 $  &
	$\begin{array}{c} \rule{0pt}{9pt} 2^{9}\cdot 3^{4}\cdot 7 \cdot 11\\ \cdot 13 \cdot 61 \cdot 257\end{array}$ & $\begin{array}{c} \rule{0pt}{9pt}2^{8}\cdot 3^{4} \cdot 5^{2} \cdot 7\\ \cdot 538073 \end{array}$ 
	\\ \hline

	$4$ \rule{0pt}{9pt}&
	$\begin{array}{c} \rule{0pt}{9pt}2^{9}\cdot 3^{2}\cdot 5^{2}\cdot 7\\ \cdot 11 \cdot 31 \cdot 3491\end{array}$ &
	$\begin{array}{c} \rule{0pt}{9pt} 2^{6}\cdot 3 \cdot 7^{2}\cdot 13\\ \cdot 53\cdot 1877 \end{array}$  &
	$\begin{array}{c} \rule{0pt}{9pt} 2^{8}\cdot 3^{2} \cdot 5 \cdot 7 \\ \cdot 11 \cdot 139 \cdot 241 \end{array}$ & $\begin{array}{c} \rule{0pt}{9pt} 2^{8}\cdot 3\cdot 7\\ \cdot 743 \cdot 4517 \end{array}$
	\\ \hline
    \end{tabular}
    \label{tab:completeflagsg2}
\end{table}
\end{center}

\subsection{Examples for Section~\ref{fano}}\label{examples of fano} In all tables \ref{tab:intersectionofthreequadricsinP4}-\ref{tab:linearsectionofGr36} we use the notation of Section~\ref{fano}. In tables~\ref{tab:linearsectionofGr26_dim1} and~\ref{tab:linearsectionofGr25}-\ref{tab:linearsectionofGr36} $L^m$ denotes a projective subspace of codimension $m$ that is transversal to the Grassmannian. 

\begin{center}
\begin{table}[H]
    \caption{Smooth complete intersections of three quadrics in $\PP^4(\Co)$ ($k=1$, $n=4$, $d=2$, $r=2$)}
    \begin{tabular}[c]{|c|c|c|c|c|c|c|c|c|} \hline
	$l$ & 1 & 2 & 3 & 4 & 5 \\ \hline
	$\dim H_{14-2l}(\PP(E^*),\Z)$ & 1 & 2 & 2 & 2 & 2 \\ \hline 
	$\dim P^{2l-1}$ & 1 & 2 & 2 & 1 & 1 \\ \hline
	$i_l$ & 40 & 480 & 8 & 16 & 8 \\ \hline
    \end{tabular}
    \label{tab:intersectionofthreequadricsinP4}
\end{table}
\end{center}

\begin{center}
\begin{table}[h]
    \caption{Smooth complete intersections of two quadrics in $\PP^5(\Co)$ ($k=1$, $n=5$, $d=2$, $r=1$)}
    \begin{tabular}[c]{|c|c|c|c|c|c|c|c|c|} \hline
	$l$ & 1 & 2 & 3 & 4 & 5 & 6 \\ \hline
	$\dim H_{14-2l}(\PP(E^*),\Z)$ & 1 & 2 & 2 & 2 & 2 & 2 \\ \hline 
	$\dim P^{2l-1}$ & 1 & 2 & 1 & 1 & 1 & 1 \\ \hline
	$i_l$ & 30 & 96 & 4 & 8 & 4 & 12 \\ \hline
    \end{tabular}
    \label{tab:intersectionoftwoquadrics}
\end{table}
\end{center}

\begin{center}
\begin{table}[h]
    \caption{Smooth intersections of $\mathrm{Gr}(2,6)$ and $L^7$ ($k=2$, $n=4$, $d=1$, $r=6$)}
    \begin{tabular}[c]{|c|c|c|c|c|c|c|c|c|} \hline
	$l$ & 1 & 2 & 3 & 4 & 5 & 6 & 7\\ \hline
	$\dim H_{30-2l}(\PP(E^*),\Z)$ & 1 & 2 & 4 & 6 & 9 & 11 & 13 \\ \hline 
	$\dim P^{2l-1}$ & 1 & 2 & 2 & 2 & 2 & 2 & 1 \\ \hline
	$i_l$ & 66 & 1152 & 2 & 8 & 2 & 56 & 1 \\ \hline
    \end{tabular}
    \label{tab:linearsectionofGr26_dim1}
\end{table}
\end{center}

\begin{center}
\begin{table}[h]
    \caption{Smooth complete intersections of three quadrics in $\PP^6(\Co)$ ($k=1$, $n=6$, $d=2$, $r=2$)}
    \begin{tabular}[c]{|c|c|c|c|c|c|c|c|c|} \hline
	$l$ & 1 & 2 & 3 & 4 & 5 & 6 & 7\\ \hline
	$\dim H_{16-2l}(\PP(E^*),\Z)$ & 1 & 2 & 3 & 3 & 3 & 3 & 3 \\ \hline 
	$\dim P^{2l-1}$ & 1 & 2 & 2 & 1 & 1 & 1 & 1\\ \hline
	$i_l$ & 140 & 3360 & 112 & 16 & 8 & 24 & 8 \\ \hline
    \end{tabular}
    \label{tab:intersectionofthreequadrics}
\end{table}
\end{center}

\begin{center}
\begin{table}[h]
    \caption{Smooth intersections of $\mathrm{Gr}(2,5)$ and $L^3$ ($k=2$, $n=3$, $d=1$, $r=2$)}
    \begin{tabular}[c]{|c|c|c|c|c|c|c|c|c|} \hline
	$l$ & 1 & 2 & 3 & 4 & 5 \\ \hline
	$\dim H_{18-2l}(\PP(E^*),\Z)$ & 1 & 2 & 4 & 5 & 6 \\ \hline 
	$\dim P^{2l-1}$ & 1 & 2 & 2 & 1 & 1 \\ \hline
	$i_l$ & 5 & $\infty$ & 12 & 2 & 1 \\ \hline
    \end{tabular}
    \label{tab:linearsectionofGr25}
\end{table}
\end{center}

\begin{center}
\begin{table}[h]
    \caption{Smooth intersections of $\mathrm{Gr}(2,7)$ and $L^6$ ($k=2$, $n=5$, $d=1$, $r=5$)}
    \begin{tabular}[c]{|c|c|c|c|c|c|c|c|c|} \hline
	$l$ & 1 & 2 & 3 & 4 & 5 & 6 & 7\\ \hline
	$\dim H_{30-2l}(\PP(E^*),\Z)$ & 1 & 2 & 4 & 6 & 9 & 12 & 14 \\ \hline 
	$\dim P^{2l-1}$ & 1 & 2 & 2 & 2 & 2 & 2 & 1\\ \hline
	$i_l$ & 266 & 13140 & 4 & 8 & 4 & 56 & 1 \\ \hline
    \end{tabular}
    \label{tab:linearsectionofGr27}
\end{table}
\end{center}

\begin{center}
\begin{table}[H]
    \caption{Smooth intersections of $\mathrm{Gr}(3,6)$ and $L^5$ ($k=3$, $n=3$, $d=1$, $r=4$)}
    \begin{tabular}[c]{|c|c|c|c|c|c|c|c|c|} \hline
	$l$ & 1 & 2 & 3 & 4 & 5 & 6 \\ \hline
	$\dim H_{28-2l}(\PP(E^*),\Z)$ & 1 & 2 & 4 & 7 & 10 & 12 \\ \hline 
	$\dim P^{2l-1}$ & 1 & 2 & 2 & 2 & 2 & 1 \\ \hline
	$i_l$ & 240 & 15552 & 648 & 120 & 2 & 1  \\ \hline
    \end{tabular}
    \label{tab:linearsectionofGr36}
\end{table}
\end{center}


\end{document}